\documentclass[12pt,a4paper,reqno]{amsart}

\usepackage{amsmath,amsthm,amsfonts,graphicx,yhmath,enumerate}

\newtheorem{thm}{Theorem}[section]
\newtheorem{lemma}[thm]{Lemma}

\newtheorem{prop}[thm]{Proposition}
\newtheorem{remark}[thm]{Remark}
\newtheorem{definition}[thm]{Definition}

\numberwithin{equation}{section}

\def\XXint#1#2#3{{\setbox0=\hbox{$#1{#2#3}{\int}$} \vcenter{\vspace{-1pt}\hbox{$#2#3$}}\kern-.5\wd0}}
\def\Xint#1{\mathchoice {\XXint\displaystyle\textstyle{#1}}{\XXint\textstyle\scriptstyle{#1}}{\XXint\scriptstyle\scriptscriptstyle{#1}}{\XXint\scriptscriptstyle\scriptscriptstyle{#1}}\!\int}
\def\intmed{\hbox{\ }\Xint{\hbox{\vrule height -0pt width 10pt depth 1pt}}}
\def\klintmed{\Xint{\hbox{\vrule height -0pt width 6pt depth 1pt}}}
\def\arc#1{\wideparen{#1}}
\def\step#1#2{\par\noindent{\underline{\it Step~#1.}}\emph{ #2}\\}

\def\bal{\begin{aligned}}
\def\eal{\end{aligned}}
\def\angle#1#2#3{#1\widehat{#2}#3}
\def\u#1{\hbox{\boldmath $#1$}}
\def\matx{\bigg(\begin{matrix}1\ 0\\ 0\ 0 \end{matrix}\bigg)}

\def\A{\mathcal{A}}
\def\B{\mathcal{B}}
\def\G{\mathcal{G}}
\def\C{\mathcal{C}}
\def\eps{\varepsilon}
\def\H{\mathcal{H}}
\def\M{\mathcal{M}}
\def\Q{\mathcal{Q}}
\def\D{\mathcal{D}}
\def\T{\mathcal{T}}

\def\RR{\mathcal{R}}
\def\R{\mathbb R}
\def\N{\mathbb N}
\def\S{\mathcal{S}}

\newtoks\by
\newtoks\paper
\newtoks\book
\newtoks\jour
\newtoks\yr
\newtoks\pages
\newtoks\vol
\newtoks\publ
\def\ota{{\hbox\vol{???}}}
\def\cLear{\by=\ota\paper=\ota\book=\ota\jour=\ota\yr=\ota
\pages=\ota\vol=\ota\publ=\ota}
\def\endpaper{\the\by, {\the\paper},
\textit{\the\jour} \textbf{\the\vol} (\the\yr), \the\pages.\cLear}
\def\endbook{\the\by, \textit{\the\book}, \the\publ.\cLear}
\def\endprep{\the\by, \textit{\the\paper}, \the\jour.\cLear}
\def\name#1#2{#2 #1}
\def\nom{ \rm no. }

\title[Diffeomorphic Approximation of $W^{1,1}$ Planar Sobolev Homeomorphisms]{Diffeomorphic Approximation of\\ $W^{1,1}$ Planar Sobolev Homeomorphisms}
\author{Stanislav Hencl}
\address{Department of Mathematical Analysis, Charles University, So\-ko\-lovsk\'a 83, 186~00 Prague 8, Czech Republic}
\email{\tt hencl@karlin.mff.cuni.cz}
\author{Aldo Pratelli}
\address{Department of Mathematics, University of Erlangen, Cauerstrasse 11, 90158 Erlangen, Germany}
\email{\tt pratelli@math.fau.de}
\subjclass[2000]{46E35}
\keywords{Mapping of finite distortion, approximation}

\begin{document}

\begin{abstract}
Let $\Omega\subseteq\R^2$ be a domain and let $f\in W^{1,1}(\Omega,\R^2)$ be a homeomorphism (between $\Omega$ and $f(\Omega)$). Then there exists a sequence of smooth diffeomorphisms $f_k$ converging to $f$ in $W^{1,1}(\Omega,\R^2)$ and uniformly.
\end{abstract}

\maketitle

\section{Introduction}

The general problem of finding suitable approximations of homeomorphisms $f:\R^n\supseteq\Omega\longrightarrow f(\Omega)\subseteq\R^n$ with piecewise affine homeomorphisms has a long history. As far as we know, in the simplest non-trivial setting (i.e. $n=2$, approximations in the $L^{\infty}$-norm) the problem was solved by Rad\'{o}~\cite{Rado}. Due to its fundamental importance in geometric topology, the problem of finding piecewise affine homeomorphic approximations in the $L^{\infty}$-norm and dimensions $n>2$ was deeply investigated in the '50s and '60s. In particular, it was solved by Moise~\cite{Moise1} and Bing~\cite{Bing} in the case $n=3$ (see also the survey book~\cite{Moise2}), while for contractible spaces of dimension $n\geq5$ the result follows from theorems of Connell~\cite{Conn}, Bing~\cite{Bing2}, Kirby~\cite{Kirby} and Kirby, Siebenmann and Wall~\cite{Kirbetal} (for a proof see, e.g., Rushing~\cite{Rush} or Luukkainen~\cite{Lukk}). Finally, twenty years later, while studying the class of quasi-conformal varietes, Donaldson and Sullivan~\cite{DonSull} proved that the result is false in dimension 4 in the context of manifolds.\par\medskip

Once completely solved in the uniform sense, the approximation problem suddenly became of interest again in a completely different context, namely, for variational models in nonlinear elasticity. Let us briefly explain why. In the setting of nonlinear elasticity (see for instance the pioneering work by Ball~\cite{Ball1}), one is led to study existence and regularity properties of minimizers of energy functionals of the form 
\begin{equation}\label{energyI}
I(f)=\int_{\Omega} W(Df)\,dx\,,
\end{equation}
where $f:\R^n\supseteq\Omega\to\Delta\subseteq\R^n$ ($n=2,3$) models the deformation of a homogeneous elastic material with respect to a reference configuration $\Omega$ and prescribed boundary values, while $W:\R^{n\times n}\to\R$ is the stored-energy functional. In order for the model to be physically relevant, as pointed out by Ball in~\cite{Ball,Ball2}, one has to require that $u$ is a homeomorphism --this corresponds to the non-impenetrability of the material-- and that
\begin{align}\label{E_WA}
W(A)\to+\infty\quad\text{as $\det A\to 0$}\,, && W(A)=+\infty\quad\text{if $\det A\leq0$}\,.
\end{align}
The first condition in~\eqref{E_WA} prevents from too high compressions of the elastic body, while the latter guarantees that the orientation is preserved.\par
Another property of $W$ that appears naturally in many problems of nonlinear elasticity is the \emph{quasiconvexity} (see for instance~\cite{Ball3}). Unfortunately, no general existence result is known under condition~(\ref{E_WA}), not even if the quasiconvexity assumption is added: one has either to drop condition~\eqref{E_WA} and impose $p$-growth conditions on $W$ (see~\cite{Mor1,AceFus}), or to require that $W$ is \emph{polyconvex} and that some coercivity conditions are satisfied (see~\cite{Ball3, MQT}). Moreover, also in the cases in which the existence of $W^{1,p}$ minimizers is known, very little is known about their regularity.

As pointed out by Ball in~\cite{Ball, Ball2} (who ascribes the question to Evans~\cite{Evans}), an important issue toward the understanding of the regularity of the minimizers in this setting (i.e., $W$ quasiconvex and satisfying~\eqref{E_WA}) would be to show the existence of minimizing sequences given by piecewise affine homeomorphisms or by diffemorphisms. In particular, a first step would be to prove that any homeomorphism $u\in W^{1,p}(\Omega;\R^n)$, $p\in[1,+\infty)$, can be approximated in $W^{1,p}$ by piecewise affine ones or smooth ones. One of the main reasons why one should want to do that, is that the usual approach for proving regularity is to test the weak equation or the variation formulation by the solution itself; but unfortunately, in general this makes no sense unless some apriori regularity of the solution is known. Therefore, it would be convenient to test the equation with a smooth test mapping in the given class which is close to the given homeomorphism. More in general, a result saying that one can approximate a given homeomorphism with a sequence of smooth (or piecewise affine) homeomorphisms would be extremely useful, because it would significantly simplify many other known proofs, and it would easily lead to stronger new results. It is important to mention here that the choice of dimension $n=2,\,3$ is not only motivated by the physical model, but also by the fact that the approximation is false in dimension $n\geq 4$, as shown in the very recent paper~\cite{HV}.\par

However, the finding of diffeomorphisms near a given homeomorphism is not an easy task, as the usual approximation techniques like mollification or Lipschitz extension using maximal operator destroy, in general, the injectivity. And on the other hand, we need of course to approximate our homeomorphism not just with smooth maps, but with smooth homeomorphisms (otherwise the approximating sequence would be not even admissible for the original problem).\par

Few words have to be said about the choice of the required property for the approximating sequence, namely, either smooth or piecewise affine. Actually, both results would be interesting in different contexts. Luckily, the two things are equivalent: more precisely, it is clear that an approximation with diffeomorphisms easily generates another approximation with piecewise affine homeomorphisms; the converse is not immediate but, at least in the plane, it is anyway known (see~\cite{MP}). Therefore, one can approximate in either of the two ways, and the other one automatically follows (for instance, in this paper we will look only for piecewise affine approximations). It is important to clarify a point: whenever we say that a map is piecewise affine, we mean that there is a \emph{locally finite} triangulation of $\Omega$ such that the map is affine on every triangle. It is actually possible to find \emph{finite} triangulations whenever this makes sense; but for instance, if $\Omega$ is not a polygon, then the triangles must obviously become smaller and smaller near the boundary, so a finite triangulation does clearly not exist.\par\medskip

Let us now describe the results which are known in the literature in this direction. The first ones were obtained in 2009 by Mora-Corral~\cite{M} (for planar bi-Lipschitz mappings that are smooth outside a finite set) and by Bellido and Mora-Corral~\cite{BMC}, in which they prove that if $u, u^{-1}\in C^{0,\alpha}$ for some $\alpha\in (0,1]$, then one can find piecewise affine approximations $v$ of $u$ in $C^{0,\beta}$, where $\beta\in(0,\alpha)$ depends only on $\alpha$.\par

More recently, Iwaniec, Kovalev and Onninen~\cite{IKO} almost completely solved the approximation problem of planar Sobolev homeomorphisms, proving that any homeomorphism $f\in W^{1,p}(\Omega,\R^2)$, for any $1<p<+\infty$, can be approximated by diffeomorphisms $f_\eps$ in the $W^{1,p}$ norm (improving the previous result, of the same authors, for the case $p=2$, see~\cite{IKO2}).\par
Later on, it was shown by Daneri and Pratelli in~\cite{DP} and~\cite{DP2} that any planar bi-Lipschitz mapping $f$ can be approximated by diffeomorphisms $f_k$ such that $f_k$ converge to $f$ in $W^{1,p}$ norm and simultaneously $f_k^{-1}$ converge to $f^{-1}$ in $W^{1,p}$, giving the first result in which also the distance of the inverse mappings is approximated.

The goal of the present paper is to prove the approximation of planar $W^{1,1}$ homeomorphism in the $W^{1,1}$ sense, so basically dealing with the important case $p=1$ which was left out in~\cite{IKO}. In particular, our main result is the following.

\begin{thm}\label{main}
Let $\Omega\subseteq\R^2$ be an open set and $f\in W^{1,1}(\Omega,\R^2)$ be a homeomorphism. For every $\eps>0$ there is a smooth diffeomorphism (as well as a countably --but locally finitely-- piecewise affine homeomorphism) $f_\eps\in W^{1,1}(\Omega,\R^2)$ such that $\|f_\eps-f\|_{W^{1,1}}+\|f_\eps-f\|_{L^\infty}<\eps$. If in addition $f$ is continuous up to the boundary of $\Omega$, then the same holds for $f_\eps$, and $f_\eps=f$ on $\partial\Omega$.
\end{thm}

Actually, our piecewise affine functions $f_\eps$ will be globally \emph{finitely} piecewise affine, thus also bi-Lipschitz, as soon as $\Omega$ is a polygon and $f$ is piecewise affine on $\partial\Omega$, see Theorem~\ref{main2}. If even just one of these two conditions does not hold, then this is clearly impossible (see Remark~\ref{remmain2}).\medskip

We conclude the introduction with a short comparison between the techniques of this paper and those of the other papers discussed above. The proofs in~\cite{M,BMC} are based on a clever refinement of the supremum norm approximation of Moise~\cite{Moise1}, while the approach of~\cite{IKO} and of the other contributions of the same authors makes use of the identification $\R^2\simeq \mathbb C$ and involves coordinate-wise $p$-harmonic functions. The techniques of the present paper are completely different with respect to them; basically, our proof is constructive, and it is based on an explicit subdivision of the domain $f$ which depends on the Lebesgue points of $Df$.\par
Our techniques resembles the basic ideas of~\cite{DP} and~\cite{DP2}, and we will also use some of the tools introduced there, but there are some extremely important differences. More precisely, on one hand in~\cite{DP} and~\cite{DP2} one had to approximate at the same time $Df$ and $Df^{-1}$ while here we need only to approximate $Df$, and this is a deep simplification. But on the other hand, in this paper we look for a sharp estimate, that is, $f$ is only in $W^{1,1}$ and we want an approximation exactly in $W^{1,1}$, thus we have not so regular maps and we cannot lose sharpness of the power anywhere, while in~\cite{DP} and~\cite{DP2} the maps were much better, namely bi-Lipschitz, and in several steps the sharpness of the power was lost. Roughly speaking, we can say that the most difficult steps of~\cite{DP} and~\cite{DP2} correspond to much simpler steps here, and vice versa.

\subsection{Brief description of the proof\label{briefdesc}}

In this section we outline the basic plan of our proof, to underline the main steps and help the reading of the construction. We remind the reader that our aim is to find an approximation done with piecewise affine homeomorphisms, and then the existence of an approximation with smooth diffeomorphisms will eventually immediately follow applying the result of~\cite{MP}.\par

First of all, we will divide our domain into some locally finite grid of small squares, these squares becoming maybe smaller and smaller close to $\partial \Omega$. We will then consider separately the ``good'' squares, and the ``bad'' ones. More precisely, a square $\S(c,r)$ in the grid will be called good if $f$ can be well approximated by a linear mapping $f(c)+M(x-c)$ there, where $M$ coincides with $Df$ in some Lebesgue point close to $c$: in particular, we will need that $\klintmed_{\S(c,r)}|Df-Df(c)|$ is small enough. Since almost every point of $\Omega$ is a Lebesgue point for $Df$, we will be able to deduce that, up to consider a sufficiently fine grid, the area covered by the good squares is as close as we wish to the total area of $\Omega$.\par

Moreover, up to a slight modification of the value of $f$ on the boundary of the squares, we will reduce ourselves to the case that
\begin{equation}\label{aha}
\intmed_{\partial \S}|Df|\leq K \intmed_{\S}|Df|\,,
\end{equation}
where $K$ is a big, but fixed, constant.\par

We will then define an approximation of $f$ (which will eventually become $f_\eps$) on the grid: on the boundary of each square, we will find a piecewise linear approximation of $f$, very close to $f$, in such a way that these approximations on the whole grid remain one-to-one (we do not have just to take care of the approximation on a single square, but also check that the different approximations coincide on the common sides, and that they do not overlap with each other). Of course this will be much easier for the good squares, since on a whole good square $f$ is already almost affine, and more complicated for the bad squares. We will do our approximation $g$ in such a way that, for any bad square $\S$,
\begin{equation}\label{jyv}
\int_{\partial \S}|D g|\leq K\int_{\partial \S} |Df|\,.
\end{equation}
The next step is then to extend the piecewise linear maps to the interior of each square; a good thing is that, being $g$ already defined on the grid in a one-to-one way, the extension inside each square is completely independent with what happens on the other squares. The rough idea to do so is that on good squares we can obtain very good estimates, while in bad squares we can get only bad estimates; but since the total area of the bad squares is arbitrarily small, in the end everything will work.

The first tool which we will need, presented in Section~\ref{sect3}, says that any piecewise linear map $g$ defined on the boundary of a square $\S$ can be extended to a piecewise affine homeomorphism $h$ in the interior of $\S$ in such a way that
\begin{equation}\label{uuu}
\intmed_{\S}|Dh|\leq K\intmed_{\partial \S}|D g|\, .
\end{equation}
This construction is done first by choosing many points on the boundary of the square; then, for any two of these points, say $x$ and $y$, we will select the shortest path joining $g(x)$ with $g(y)$ remaining inside the portion of $\R^2$ having $\varphi(\S)$ as boundary. Using in a careful way these shortest paths, we will then eventually obtain the definition of $h$ such that~(\ref{uuu}) holds true. This estimate, together with~(\ref{aha}) and~(\ref{jyv}), readily implies that for every bad square $\S$ one has
\begin{equation}\label{estibad}
\int_\S |Dh| \leq K \int_{\S} |Df|\,;
\end{equation}
we have then just to take a very fine grid, so that a very small portion $\Omega_B$ of $\Omega$ is covered with bad squares, and hence we will have
\[
\int_{\Omega_B} |Df-Dh| \leq \int_{\Omega_B} |Df|+|Dh| \leq (K+1) \int_{\Omega_B} |Df| \leq \eps\,.
\]
It remains then to consider the good squares, and here we will have to be extremely precise. As already said, around every good square $\S$ the map $f$ is very close to being affine, hence the image of $\S$ is very close to a parallelogram; therefore, there is no problem unless this parallelogram degenerates. Let us be more precise: for all the squares corresponding to a matrix $M$ with strictly positive determinant (hence, the parallelogram does not degenerate), the extension inside $\S$ is trivial; it is enough to divide the square in two triangles and consider on each triangle the affine map which equals $f$ on the three vertices. By construction, we will easily see that this works perfectly.\par

The good squares corresponding to $M=0$ are a problem only in principle: indeed, we can treat them as bad squares. The estimate~(\ref{estibad}) says that this gives a cost of a big constant $K$ times the total integral of $Df$ on those squares; however, since they are good squares and the corresponding matrix is $M=0$, by definition the integral of $Df$ will be extremely small, and everything will work.\par

The hard problem, instead, is for good squares for which $M\neq 0$ but $\det M=0$: these correspond to degenerate parallelograms, and we have to treat them carefully because these squares can cover a large portion of $\Omega$: in fact, recall that the set $\{x:\ |Df(x)|\neq 0\text{ and } J_f(x)=0\}$ can have positive or even full measure for a Sobolev homeomorphism (see~\cite{H}). Section~\ref{sect4} is devoted to build the extension for this specific case, which is somehow similar to the one with the shortest paths described above. The big difference with that case is, on one hand, that this time we are in a good square, hence very close to a Lebesgue point for $Df$, and this helps even for the degenerate case. But on the other hand, this time we are not satisfied with an estimate like~(\ref{uuu}), where a big constant $K$ appears, and we need instead an approximation $h$ which is very close to the original $f$. This extension procedure will be the most delicate step of the construction.\par

The construction of the proof, divided in its several steps, is done in the last Section~\ref{sect5}. Basically, putting together all the ingredients described above, the proof will then be concluded for what concerns the existence of the piecewise affine approximation; the existence of the smooth approximation will then follow thanks to the result of~\cite{MP}, while the claim about the boundary values will be easily deduced by the whole construction.

\subsection{Preliminaries and notation}

In this section we shortly list the basic notation that will be used throughout the paper. By $\S(c,r)$ we denote the square centered at $c$, side length $2r$ and sides parallel to the coordinate axis, while $\S_0=\{(x,y)\in\R^2:|x|+|y|< 1\}$ is the ``rotated square'', that we use only in Section~\ref{sect3}. Similarly, $\B(c,r)$ is the ball centered at $c$ with radius $r$.\par

The points in the domain $\Omega$ will be usually denoted by capital letters, such as $A$, $B$ and so on, while points in the image $f(\Omega)$ will be always denoted by bold capital letters, such as $\u A$, $\u B$ and similar. To shorten the notation and help the reader, whenever we use the same letter $A$ for a point in the domain and $\u A$ (in bold) for a point in the target, this always means that $\u A$ is the image of $A$ under the mapping that we are considering in that moment. By $AB$ (resp., $\u A\u B$) we denote the segment between the points $A$ and $B$ (resp. $\u A$ and $\u B$). The length of this segment is denoted as $\H^1(AB)$, or as $\overline{AB}$, while $\H^1(\gamma)$ is the length of a curve $\gamma$. With the notation $\arc{AB}$ (or $\arc{\u{AB}}$) we will denote a particular path between $A$ and $B$ (or $\u{A}$ and $\u{B}$), whose length will be then $\H^1(\arc{AB})$, or $\H^1(\arc{\u{AB}}$); we will use this notation only when it is clear what is the path we are referring to (often this will be a shortest path between the points). Given three non-aligned points $A,\,B,\,C$ (or $\u{A},\,\u{B},\,\u{C}$), we will denote by $\angle ABC$ (or $\angle{\u{A}}{\u{B}}{\u{C}}$) the angle in $(0,\pi)$ between them, and by $ABC$ (or $\u{ABC}$) the triangle having them as vertices.\par
We will denote the (modulus of the) horizontal and vertical derivatives of any mapping $f=(f_1,f_2):\R^2\to\R^2$ as 
\begin{align*}
|D_1f|=\sqrt{\bigg(\frac{\partial f_1}{\partial x}\bigg)^2+\bigg(\frac{\partial f_2}{\partial x}\bigg)^2}\,, &&
|D_2f|=\sqrt{\bigg(\frac{\partial f_1}{\partial y}\bigg)^2+\bigg(\frac{\partial f_2}{\partial y}\bigg)^2}\,.
\end{align*}
Analogously, the derivatives of the components $f_1$ and $f_2$ are written as
\begin{align*}
D_1f_1=\frac{\partial f_1}{\partial x}\,, && D_2f_1=\frac{\partial f_1}{\partial y}\,, &&
D_1f_2=\frac{\partial f_2}{\partial x}\,, && D_2f_2=\frac{\partial f_2}{\partial y}\,.
\end{align*}
Whenever a continuous function $g$ is defined on some curve $\gamma$ (usually, $\gamma$ will simply be the boundary of a square) we will denote by $\tau(t)$ the tangent vector to $\gamma$ at $t\in\gamma$, and by $Dg(t)$ the derivative of $g$ at $t$ in the direction of $\tau(t)$. With a small abuse of notation, even if the derivative is not necessarily defined, we will write $\int_{\gamma} |Dg(t)|\,d\H^1(t)$ to denote the length of the curve $g(\gamma)$: notice that the latter length is always well-defined, possibly $+\infty$, and it actually coincides with $\int_{\partial\S} |Dg(t)|$ as soon as this is defined. Finally, notice that, if a function $f$ is affine on a square $\S$, being $Df \equiv M$ for some matrix $M$, and we call $g$ the restriction of $f$ to $\partial\S$, then $Dg(t) = M \cdot \tau(t)$ for any $t\in\partial\S$.\par
The letter $K$ will always be used to denote a large purely geometrical constant, not depending on anything; we will not modify the letter, even if the constant may always increase from line to line. For the sake of simplicity (and since the precise value of $K$ does not play any role) we do not explicitely calculate the value of this constant.

\section{Extension from the boundary of the square\label{sect3}}

This section is entirely devoted to show the result below about the extension of a map from the boundary of the square to the whole interior.

\begin{thm}\label{extension}
Let $g:\partial \S_0\to\R^2$ be a piecewise linear and one-to-one function. There is a finitely piecewise affine homeomorphism $h:\S_0\to \R^2$ such that $h=g$ on $\partial \S_0$, and
\begin{equation}\label{hope}
\int_{\S_0}|Dh(x)|\, dx\leq K \int_{\partial \S_0}|Dg(t)|\, d\H^1(t)\,.
\end{equation}
\end{thm}
\begin{proof}
The construction of the map $h$ is quite long and technical, and hence we subdivide it in several steps.

\step{1}{Choice of good corners, so that~(\ref{goodpoint}) holds.}
For our construction, we will need to assume that $\int_{\partial \S_0} |Dg|$ does not concentrate too much around the corners; more precisely, we will need that
\begin{equation}\label{goodpoint}
\int_{\B(V_i,r)\cap\partial \S_0}|Dg|\,d\H^1\leq  Kr \int_{\partial \S_0}|Dg|\, d\H^1\qquad\text{ for all }r\in(0,1),\ i\in\{1,2\}\,,
\end{equation}
being $V_1\equiv(0,-1)$ and $V_2\equiv(0,1)$. It is quite easy to achieve that: in fact, it is enough to find two opposite points $P_1,\, P_2\in\partial \S_0$ such that 
\begin{equation}\label{goodcorner}
\int_{\B(P_i,r)\cap\partial \S_0}|Dg|\leq 6r \int_{\partial \S_0}|Dg|\qquad\text{ for all }r\in(0,\sqrt 2),\ i\in\{1,2\}\,,
\end{equation}
because then we can apply a bi-Lipschitz transformation (with bi-Lipschitz constant independent from $P_1$ and $P_2$) which moves the points $P_1$ and $P_2$ on the vertices $V_1$ and $V_2$, and get~(\ref{goodpoint}). And in turn, to obtain~(\ref{goodcorner}), we notice that every point of $\partial \S_0$ is a possible choice for $P_1$ or $P_2$ unless it is, or its opposite point is, in the set
\[
\A:=\bigg\{P\in\partial \S_0:\, \exists\, r\in (0,1):\, \int_{\B(P,r)\cap\partial \S_0}|Dg|>6 r \int_{\partial \S_0}|Dg|\bigg\}\, .
\]
By a Vitali covering argument, we can cover $\A$ with countably many balls $\B(P_i,3r_i)$ such that every $P_i$ is in $\A$, and the corresponding sets $\B(P_i,r_i)\cap\partial \S_0$ are as in the definition of $\A$ and are pairwise disjoint. Therefore, we can calculate
\[
\H^1(\A)\leq \sum_i 6r_i\leq \sum_i \frac{\int_{\B(P_i,r)\cap\partial \S_0}|Dg|}{\int_{\partial \S_0}|Dg|}\leq 1\,,
\]
and since $\H^1(\partial \S_0)=4\sqrt{2}$ it clearly follows that two opposite points both in $\partial \S_0\setminus \A$ exist and then satisfy~(\ref{goodcorner}), as required. 

\step{2}{Definition of the grid on $\partial \S_0$, and of the paths $\gamma^i$.}
To define our map $h$, we will make use of a fine grid made by horizontal segments in $\S_0$. More precisely, we will take several (but finitely many) distinct points $A^0\equiv (0,-1),\, A^1,\, A^2,\, \dots \,,\, A^k\equiv (0,1)$ in $\partial\S_0$, all with non-positive first coordinate $A^i_1$ and with second coordinate $A^i_2$ increasing, with respect to $i$, from $-1$ to $1$; on the opposite side, we will take the corresponding points $B^i\equiv (-A^i_1,A^i_2)$, so that the segments $A^iB^i$ are horizontal.\par

The way to choose our points is simple: since $g$ is piecewise linear, we can take the points in such a way that $g$ is linear on every segment $A^iA^{i+1}$, as well as in every $B^iB^{i+1}$. Since this property is of course not destroyed if we \emph{add} more points $A^i$ (as long as we also add the corresponding points $B^i$, of course), we are allowed to add more points during the construction, of course taking care to add only finitely many: we will do this a first time in few lines, and then also later.\par

From now on, we will call $\u\S_0$ the bounded component of $\R^2\setminus g(\partial \S_0)$, which is a polygon because $g$ is piecewise linear; notice that the map $h$ that we want to construct must be a homeomorphism between $\S_0$ and $\u\S_0$. Then, for any $0<i<k$, we define $\gamma^i$ the shortest path which connects $\u A^i$ and $\u B^i$ inside the closure of $\u\S_0$ (this shortest path is unique, as we will show in Step~3). Notice that, since $\u\S_0$ is a polygon, every $\gamma^i$ is piecewise linear, and any junction between two consecutive linear pieces is in $\partial\u\S_0$.\par

Up to add one more point between $A^0$ and $A^1$ (plus the corresponding one on the right part), we can suppose that $\gamma^1$ is either a segment between $\u A^1$ and $\u B^1$, and this happens if and only if the angle between $\u{A}^1$, $\u{A}^0$ and $\u{B}^1$ which goes inside $\u\S_0$ is smaller than $\pi$, or it is done by the union of the two segments $\u{A}^1\u{A}^0$ and $\u{A}^0\u{B}^1$, thus it entirely lies on $\partial\u\S_0$. We do the same between $A^{k-1}$ and $A^k$.\par

\step{3}{Uniqueness of the shortest paths.}
Let $\Q$ be a simply connected closed planar domain with polygonal boundary. We briefly recall the proof of the well-known fact that, for any two points in $\Q$ --not necessarily on the boundary-- there is a unique shortest path inside $\Q$. Since the existence is obvious, we just have to check the uniqueness.\par

If the claim were not true, there would be two points $A,\,B\in\Q$ and two shortest paths $\tau_1$ and $\tau_2$ between $A$ and $B$ inside $\Q$, such that $\tau_1$ and $\tau_2$ meet only at $A$ and $B$. The union of the two paths is then a polygon, say with $n$ sides. The sum of the internal angles of this polygon is $\pi(n-2)$, and thus there must be a vertex of the polygon, different from $A$ and $B$ and thus inside one of the shortest paths, corresponding to an angle strictly less than $\pi$. Since the interior of the polygon is entirely in the interior of $\Q$, this is of course impossible, because cutting around that vertex would strictly shorten the length of the path, against the minimality.

\step{4}{The path $\gamma^{i+1}$ is above $\gamma^i$, and definition of $\gamma^i_1,\,\gamma^i_2,\, \gamma^i_3$.}
For two curves $\gamma$ and $\tilde\gamma$ inside $\u\S_0$ and with endpoints in $\partial\u\S_0$, we say that ``$\gamma$ is above $\tilde\gamma$'' if $\gamma$ does not intersect the interior of the (possibly disconnected) subset of $\u\S_0$ whose boundary is the union between $\tilde\gamma$ and the path on $\partial\u\S_0$ connecting the endpoints of $\tilde\gamma$ and containing $\u{A}^0=g(A^0)$. We want to show that, for any $0<i <k-1$, the path $\gamma^{i+1}$ is above $\gamma^i$.\par

To show that, assume that two points $\u{P}$ and $\u{Q}$ belong to both the paths $\gamma^i$ and $\gamma^{i+1}$. Then, the two restrictions of $\gamma^i$ and $\gamma^{i+1}$ from $P$ to $Q$ are two shortest paths, and by Step~3 we derive that $\gamma^i$ and $\gamma^{i+1}$ coincide between $\u P$ and $\u Q$. As an immediate consequence of this observation, we get that $\gamma^{i+1}$ is above $\gamma^i$ as claimed.\par
Another immediate consequence is the following: the intersection between $\gamma^i$ and $\gamma^{i+1}$ is always a connected subpath, possibly empty. If it is not empty, and then it is a path $\arc{\u{PQ}}$, we will subdivide both $\gamma^i$ and $\gamma^{i+1}$ in three parts, writing
\begin{align*}
\gamma^i = \gamma^i_1 \cup \gamma^i_2 \cup \gamma^i_3\,, &&
\gamma^{i+1} = \gamma^{i+1}_1 \cup \gamma^{i+1}_2 \cup \gamma^{i+1}_3\,,
\end{align*}
where $\gamma^i_1$ (resp. $\gamma^{i+1}_i$) is the first part, from $\u A^i$ to $\u P$ (resp., from $\u A^{i+1}$ to $\u P$); $\gamma^i_2$ (resp. $\gamma^{i+1}_2$) is the second part, from $\u P$ to $\u Q$ (thus the common part, and $\gamma^i_2=\gamma^{i+1}_2$); and $\gamma^i_3$ (resp., $\gamma^{i+1}_3$) is the third and last part, from $\u Q$ to $\u B^i$ (resp., from $\u Q$ to $\u B^{i+1}$). If $\gamma^i$ and $\gamma^{i+1}$ have empty intersection, then we simply set $\gamma^i_1=\gamma^i$ and $\gamma^{i+1}_1=\gamma^{i+1}$, letting $\gamma^i_2$, $\gamma^i_3$, $\gamma^{i+1}_2$ and $\gamma^{i+1}_3$ be empty paths. The situation is depicted in Figure~\ref{fig:paths}, where the common part $\gamma^i_2=\gamma^{i+1}_2$ is done by two segments, one on $\partial\u\S_0$ and the other one in the interior of $\u\S_0$.\par
\begin{figure}[thbp]
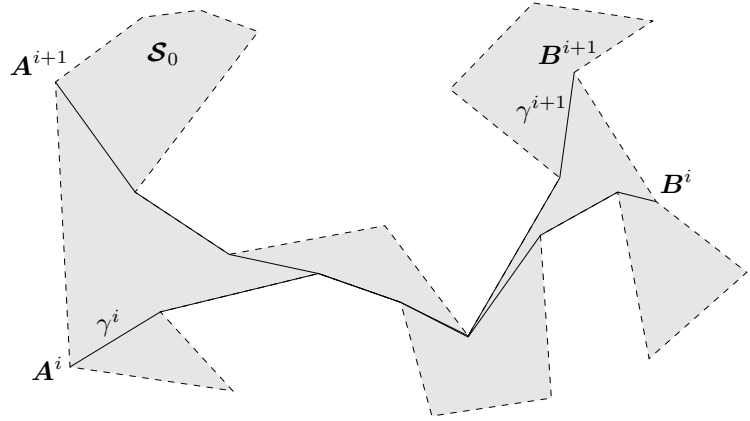
\caption{The paths $\gamma^i$ and $\gamma^{i+1}$ in Step~4.}\label{fig:paths}
\end{figure}
Notice that this subdivision of a path does not depend only on the path itself, but also on the other path that we are considering; in other words, the subdivision of the path $\gamma^j$ done when $i=j$, and then considering the possible common part between $\gamma^j$ and $\gamma^{j+1}$, does not need to coincide with the subdivision of the same path done when $i=j-1$, and then considering the possible common part between $\gamma^j$ and $\gamma^{j-1}$.

\step{5}{Convexity of the polygon having boundary $\gamma^{i+1}_1\cup \u A^{i+1}\u P$.}
Let us call $\u P$ the last point of the path $\gamma^{i+1}_1$; hence $\u P$ is the first common point with $\gamma^i$, if $\gamma^i$ and $\gamma^{i+1}$ have a non-empty intersection, while otherwise $\u P=\u B^{i+1}$. We claim that the polygon having $\gamma^{i+1}_1\cup \u A^{i+1}\u P$ as its boundary is convex (notice that in principle the curve $\gamma^{i+1}_1$ and the segment $\u A^{i+1}\u P$ could have other intersection points in addition to $\u A^{i+1}$ and $\u P$). We start assuming that $\gamma^{i+1}_2\neq \emptyset$, at the end of this step we will then consider the other case.\par

If $\gamma^{i+1}_1$ is a single point, or just a segment, then the claim is emptily true, and the convex polygon is degenerate. Let us assume then that $\gamma^{i+1}_1$ is done at least by two affine pieces, and assume also, just to fix the ideas, that the direction of the oriented segment $\u A^i \u A^{i+1}$ is $\pi/2$, as in Figure~\ref{fig:range}, left. Call then $\u\D\subseteq \u\S_0$ the polygon having, as boundary, the Jordan curve $\gamma^i_1\cup\gamma^{i+1}_1\cup \u A^{i+1} \u A^i$. The same argument as in Step~3 immediately ensures that, for any vertex of $\gamma^{i+1}$ (i.e., any junction point between two consecutive linear pieces of $\gamma^{i+1}_1$), the angle pointing inside $\u\S_0$ (hence in particular inside $\u\D$) is bigger than $\pi$. By construction, and recalling Step~4, we get also that none of these points can belong to the curve in $\partial\u\S_0$ connecting $\u A^i$ and $\u B^i$ and containing $\u A^0$, since such a point should necessarily belong also to $\gamma^i$, against the definition of $\u P$. Of course, this already ``suggests'' that our convexity claim is true, but observe that the proof is still not over, since in principle $\gamma^{i+1}_1$ could be some spiral-like curve connecting $\u A^{i+1}$ with $\u P$. To conclude the proof, for any vertex of $\gamma^{i+1}_1$ (except $\u P$) consider the range of directions pointing toward the interior of $\u\D$: for instance, the range associated to $\u A^{i+1}$ in the situation of Figure~\ref{fig:range}, left, is done by the angles between $-\pi/2$ and $-\pi/3$. We claim that, for each vertex of the curve $\gamma^{i+1}_1$, this range cannot contain the angle $+\pi/2$: observe that this will immediately imply the required convexity.\par
Assume then by contradiction that this claim is false, and let $\u Q$ be the first vertex of $\gamma^{i+1}_1$ having $\pi/2$ in its range of directions; by a trivial perturbation argument we can assume that $\pi/2$ is in the interior of this range, and then the vertical line passing through $\u Q$ is in the interior of $\u\D$ for a while, both above and below $\u Q$ itself. Call then, as in Figure~\ref{fig:range}, left, $\u Q^-$ and $\u Q^+$ the first points of this line, respectively below and above $\u Q$, which are on $\partial \u\D$. Since the segment $\u Q^-\u Q^+$ is parallel to $\u A^i\u A^{i+1}$, each of these points must belong either to $\gamma^i_1$ or to $\gamma^{i+1}_1$. Observe now that, if $\gamma$ is a shortest path in $\u S_0$ between its extremes, it is also a shortest path in $\u S_0$ between any pair of its points. In particular, if the line connecting two points of $\gamma$ is entirely in the closure of $\u S_0$, then $\gamma$ must be the segment between these two points. This immediately imply that none of the points $\u Q^-$ and $\u Q^+$ can belong to $\gamma^{i+1}_1$, because otherwise $\gamma^{i+1}_1$ should be a segment between that point and $\u Q$; as a consequence, both the points $\u Q^-$ and $\u Q^+$ must belong to $\gamma^i_1$, but this is also impossible because then $\gamma^i_1$ should be the segment between them. The contradiction shows the claim, and then we have obtained the required convexity. Of course, the very same argument works for the polygon having boundary $\gamma^{i+1}_3\cup \u{PB}^{i+1}$, being this time $\u{P}$ the first point of $\gamma^{i+1}_3$, and everything also works for the polygons around the path $\gamma^i$ instead of $\gamma^{i+1}$.\par
\begin{figure}[thbp]
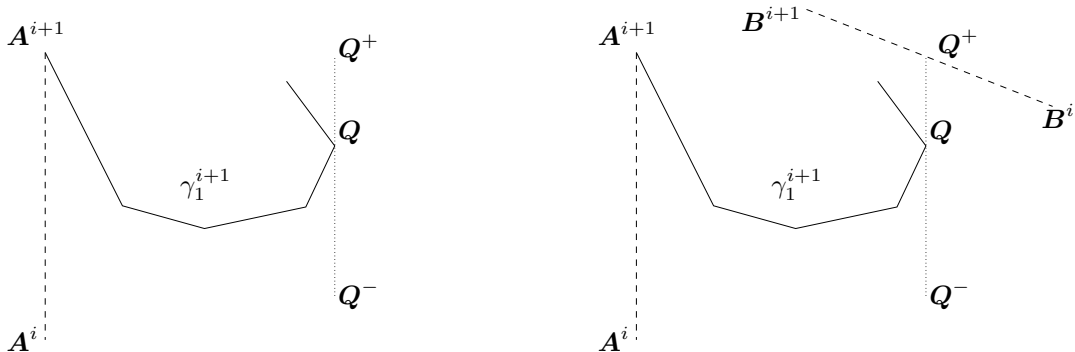
\caption{Construction in Step~5.}\label{fig:range}
\end{figure}

Let us then consider the case when $\gamma^{i+1}_2=\emptyset$, that is, the case when $\gamma^i$ and $\gamma^{i+1}$ are disjoint: this situation is depicted in Figure~\ref{fig:range}, right. This time, we define $\u\D$ the polygon having, as boundary, the Jordan curve $\gamma^i\cup \u B^i \u B^{i+1}\gamma^{i+1}\cup \u A^{i+1} \u A^i$. The very same argument as in the first case ensures again that every vertex of $\gamma^{i+1}$ has angle bigger than $\pi$ in the direction inside $\u\D$; as a consequence, the required convexity follows, as before, if the range of every vertex of $\gamma^{i+1}$ does not contain the direction $+\pi/2$.\par

However, this time it is not impossible that a vertex $\u Q$ of $\gamma^{i+1}$ has $\pi/2$ in its range. Let then, as before, $\u Q$ be the first vertex (if any) with this property, and let $\u Q^{\pm}\in\partial \u\D$ be as before. As already noticed, none of the points $\u Q^{\pm}$ can be in $\gamma^{i+1}$, and at most one in $\gamma^i$. Hence, the only possibility is that one point is in $\gamma^i$, and the other one in $\u B^i\u B^{i+1}$. A simple topological argument ensures that $\u Q^-$ must be in $\gamma^i$ and $\u Q^+$ in $\u B^i\u B^{i+1}$. Indeed, consider the path, contained in $\partial \u\D$ and not containing $\u A^i\u A^{i+1}$, which connects $\u Q$ and $\u Q^-$; together with the segment $\u Q^-\u Q$, this is a Jordan curve, and then it can not intersect the other path in $\partial \u\D$ which contains $\u A^i\u A^{i+1}$: in particular, it must contain $\u Q^+$, and it readily follows, as claimed, that $\u Q^-\in \gamma^i$, $\u Q^+ \in \u B_i\u B^{i+1}$. The very same topological argument ensures also that $\u B^{i+1}$ is the ``left'' vertex (that is, the one in the direction $\u A^i\u A^{i+1}$) of the segment $\u B^i\u B^{i+1}$, and $\u B^i$ is the ``right'' one, as in Figure~\ref{fig:range}, right.\par

Let us now restrict our attention to the subset $\u\D_0$ of $\u\D$ made by the polygon whose boundary is the part of $\gamma^{i+1}$ connecting $\u Q$ to $\u B^{i+1}$, plus the two segments $\u B^{i+1} \u Q^+$ and $\u Q^+ \u Q$. The same argument of the first half of this step ensures that the range of directions, toward the interior of $\u\D_0$, corresponding to any vertex of $\gamma^{i+1}$ in $\partial\u\D_0$, can never contain the direction of the segment $\u B^i\u B^{i+1}$, since otherwise a segment parallel to $\u B^i\u B^{i+1}$ and contained in $\u\D_0$ should have both the endpoints in the segment $\u Q\u Q^+$, which is impossible.\par

Finally, it is immediate to notice that this property of the directions, analogously as before, is enough to ensure the required convexity of the polygon having $\gamma^{i+1} \cup \u{A}^{i+1}\u P$ as boundary.

\step{6}{Definition of the ``vertical segments'' and their length.} 
In this step, we associate to any vertex $\u P$ of the curve $\gamma^{i+1}$ a point (or many points) $\u Q$ of the curve $\gamma^i$, and vice versa. Every such segment $\u{PQ}$, which we will call ``vertical'', will be contained in the closure of the polygon $\u\D\subseteq \u \S_0$ having boundary $\u A^i \u A^{i+1}\cup \gamma^{i+1}\cup \u B^{i+1}\u B^i \cup \gamma^i$, any two vertical segments will have empty intersection, except possibly at a common endpoint, and the following estimate for the length of the vertical segments will hold,
\begin{equation}\label{estilength}
\H^1(\u{PQ}) \leq \max \Big\{\H^1(\u A^i\u A^{i+1}),\,\H^1(\u B^i\u B^{i+1}) \Big\}\,.
\end{equation}
Let us give our definition distinguishing the possible cases, as in Step~5.\par

First of all, consider the situation, depicted in Figure~\ref{fig:vertical}, left, when $\gamma^i$ and $\gamma^{i+1}$ have a non-empty intersection. In the common part $\gamma^i\cap\gamma^{i+1}=\gamma^i_2=\gamma^{i+1}_2$, we will associate to any vertex $\u P$ of $\gamma^{i+1}$ the same point $\u Q\equiv \u P$, which is also in $\gamma^i$ by definition. The segment $\u{PQ}$ is just a point, which is of course in the closure of $\u\D$, and the length is $0$, so that~(\ref{estilength}) of course holds. In the ``left'' part of the paths, instead, we will give the following simple definition. To any vertex $\u{P}\in\gamma^{i+1}_1$, we associate the point $\u{Q}\in\gamma^i_1$ so that the segment $\u{PQ}$ is parallel to $\u A^{i+1}\u A^i$: the existence and uniqueness of such a point, the validity of~(\ref{estilength}), and the fact that $\u{PQ}$ is contained in the closure of $\u\D$, all come immediately from the convexity obtained in Step~5. We do the very same thing for the vertices of $\gamma^i_1$, and we argue completely similarly for the ``right'' part of the paths, of course taking segments parallel to $\u B^{i+1}\u B^i$, instead of $\u A^{i+1}\u A^i$. Then we have already completed our definition of the vertical segments, and the fact that any two such segments do not intersect is obvious from the construction.\par
\begin{figure}[thbp]
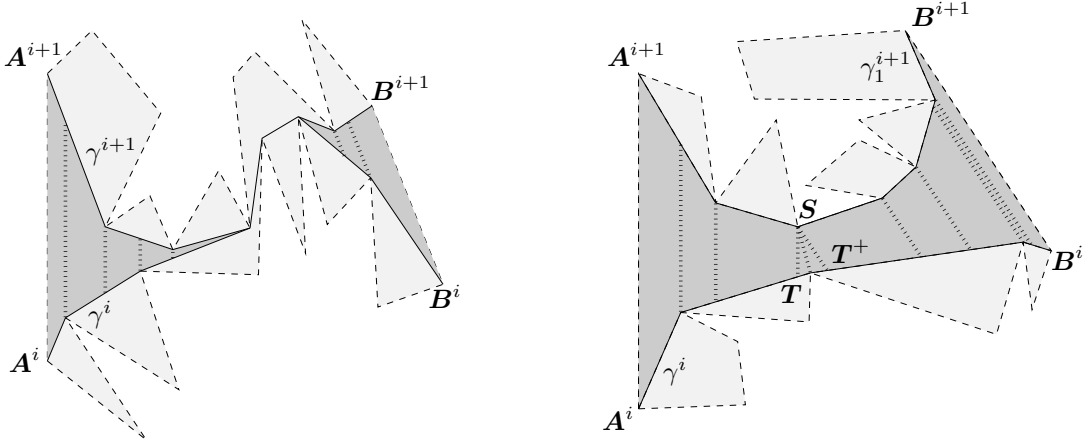
\caption{Construction in Step~6: the polygon $\u\S_0$ (resp., $\u\D$) is light (resp., dark) coloured, and the ``vertical segments'' are dotted.}\label{fig:vertical}
\end{figure}
Consider now the situation when $\gamma^i\cap\gamma^{i+1}=\emptyset$, see Figure~\ref{fig:vertical}, right. Without loss of generality we can think that, as in the Figure, the direction of $\u A^i \u A^{i+1}$ is vertical, while the segment $\u B^i \u B^{i+1}$ goes ``toward left''. Let then $\u S\in \gamma^{i+1}$ and $\u T \in \gamma^i$ be the two closest points such that the segment $\u{TS}$ is vertical; notice that it is possible that $\u S=\u A^{i+1}$ or that $\u S= \u B^{i+1}$, this makes no difference in our proof, even if the picture shows an example where $\u S$ is in the interior of the curve $\gamma^{i+1}$.\par

Let us now consider the subset $\u\D_0\subseteq \u\D$ whose boundary is given by the two segments $\u A^i\u A^{i+1}$ and $\u{ST}$, together with the parts of $\gamma^i$ (resp., $\gamma^{i+1}$), connecting $\u A^i$ and $\u T$ (resp., $\u{A}^{i+1}$ and $\u S$). Again by the convexity result of Step~5, it is clear that at any point of $\gamma^{i+1}$ between $\u A^{i+1}$ and $\u S$ starts a vertical segment, whose interior is entirely contained in $\u\D_0$, which ends in a point of $\gamma^i$ between $\u A^i$ and $\u T$. We define then in the obvious way the ``vertical segments'' inside $\u\D_0$, which are in fact vertical. The validity of~(\ref{estilength}) is as usual obvious from the convexity.\par

Consider now the half-line starting at $\u S$ and parallel to $\u B^{i+1} \u B^i$. The choice of the points $\u S$ and $\u T$, together with the convexity proved in Step~5, ensure that this half-line remains inside $\u \D$ for a while, after the point $\u S$; therefore, the intersection of this half-line with $\u\D$ is a segment $\u S \u T^+$, and the point $\u T^+$ is on $\gamma^i$ by construction. Observe that $\u T^+$ coincides with $\u T$ in the particular case when $\u B^{i+1}\u B^i$ is parallel to $\u A^{i+1}\u A^i$, but otherwise it stays, as in the figure, outside of $\u\D_0$. The construction implies that all the half-lines, parallel to $\u B^{i+1}\u B^i$ and starting at a point of $\gamma^{i+1}$ after $\u S$, remain in $\u\D$ for a while and then intersect $\gamma^i$ at some point after $\u T^+$. We use this observation to associate to any vertex of $\gamma^{i+1}$ after $\u S$ a point of $\gamma^i$ after $\u T^+$, and we call then ``vertical segments'' all the corresponding segments, which are actually not vertical but parallel to $\u B^{i+1}\u B^i$. Finally, to every vertex of $\gamma^i$ between $\u T$ and $\u T^+$, if any, we associate always the point $\u S$. The validity of~(\ref{estilength}) for all the vertical segments is then again clear by the construction and by Step~5, and any two vertical segments have always empty intersection, unless in the case when they meet at $\u S$. This concludes the step.

From now on, we will always consider as ``vertices'' the points $\u S$, $\u T$ and $\u T^+$, even if they were not vertices in the sense of the piecewise linear curves. Moreover, for every vertex of $\gamma^i$, or of $\gamma^{i+1}$, we will consider as ``vertex'' also the corresponding point in the other curve, which again could be or not be a vertex in the classical sense. Notice that in this way we are adding a \emph{finite number} of new vertices and, as already pointed out before, it is always admissible to regard as ``vertices'' also finitely many new points in our curves. Summarizing, on the piecewise linear curve $\gamma^i$ we are considering as ``vertices'' all the actual vertices, plus some other new points. However, these ``new points'' have been selected working on the region between $\gamma^i$ and $\gamma^{i+1}$, and then they do not need to coincide with the ``new points'' selected by working on the region between $\gamma^{i-1}$ and $\gamma^i$.

\step{7}{Definition of $\tilde h$ on $\S_0$.}
We are now ready to define a function on $\S_0$ which extends $g$; for simplicity, we start now with the definition of a ``temptative'' function $\tilde h$, without taking care of the injectivity. The definitive function $h$ will be obtained later.\par

Recall that we have selected several horizontal segments $A^iB^i$, $1\leq i \leq k-1$, in the square $\S_0$; the square is then divided in $k-2$ ``horizontal strips'', i.e. the regions between two consecutive horizontal segments, plus two triangles, the ``top one'' $A^{k-1}A^kB^{k-1}$ and the ``bottom one'' $A^1 A^0 B^1$.\par

We start defining the function $\tilde h$ on the ``$1$-dimensional skeleton'', that is, the union of $\partial \S_0$ with all the horizontal segments $A^iB^i$: more precisely, we set $\tilde h = g$ on the boundary $\partial \S_0$, while for every $1\leq i \leq k-1$ we define $\tilde h$ on the horizontal segment $A^iB^i$ as the piecewise linear function, parametrized at constant speed, whose image is the path $\gamma^i$. Notice that, with this definition, $\tilde h$ is continuous on the $1$-skeleton.\par

To extend $\tilde h$ to the whole $\S_0$, we can then argue separately on each of the horizontal strips of $\S_0$, as well as on the top and bottom triangle. First, let us consider the bottom triangle $A^1A^0B^1$: thanks to the construction of Step~2, we know that the path $\gamma^1$ is either the segment $\u A^1\u B^1$, or the union of the two segments $\u A^1\u A^0$ and $\u A^0\u B^1$. In the first case, we define $\tilde h$ on the bottom triangle as the affine function extending the values on the boundary; in the second case, let $P$ be the point of the segment $A^1B^1$ such that $\tilde h(P)=\u A^0$, let us extend $\tilde h$ as constantly $\u A^0$ on the segment $PA^0$, and let $\tilde h$ be the (degenerate) affine function extending the values on the boundary on each of the two triangles $A^1 P A^0$ and $A^0 P B^1$. In the top triangle, we give of course the very same definition of $\tilde h$.\par

Let us now consider the horizontal strip $\D_i$ between $A^iB^i$ and $A^{i+1}B^{i+1}$, and let us call $\u\D_i$ the bounded region in $\u\S_0$ having as boundary the closed curve $\gamma^{i+1}\cup \u B^{i+1}\u B^i \cup \gamma^i \cup \u A^i \u A^{i+1}$. In Step~6, we have selected a finite number of points on $\gamma^i$ and on $\gamma^{i+1}$, and we have called ``vertical segments'' the corresponding segments. More precisely, let us denote the points in $A^{i+1} B^{i+1}$ as $P_0=A^{i+1}$, $P_1$, \dots , $P_{M-1}$, $P_M = B^{i+1}$, and the points in $A^iB^i$ as $Q_0=A^i$, $Q_1$, \dots , $Q_{M-1}$, $Q_M = B^i$; as always, let us write $\u P_j=\tilde h(P_j)$, and $\u Q_j=\tilde h(Q_j)$. Keep in mind that each segment $\u P_j\u Q_j$, whose interior is entirely contained in $\u\D_i$, has been called a ``vertical segment'', and notice that the points $\u P_j$ and $\u Q_j$ are not necessarily all different: for instance, the point $\u S$ of Figure~\ref{fig:vertical}, right, is the point $\u P_j$ for three consecutive indices $0 < j< M$.\par

We are finally in position to give the definition of $\tilde h$ on the interior of each strip $\D_i$ (and, since $\tilde h$ has been already defined in the $1$-skeleton and on the top and bottom triangle, this will conclude the present step). The strip $\D_i$ is the essentially disjoint union of the triangles $P_j P_{j+1} Q_j$ and $P_{j+1} Q_j Q_{j+1}$ for all $0\leq j < M$, and $\u\D_i$ is the essentially disjoint union of the corresponding triangles $\u P_j \u P_{j+1} \u Q_j$ and $\u P_{j+1} \u Q_j \u Q_{j+1}$, where the triangles in $\u\D_i$ (but not those in $\D_i$) can be degenerate, in particular they are degenerate for the points in $\gamma^{i+1}_2=\gamma^i_2$. We define then $\tilde h$ on $\D_i$ as the function which is affine on each of the above-mentioned triangles. Notice that, by construction, $\tilde h$ is linear on each side $P_j P_{j+1}$ and $Q_j Q_{j+1}$, hence this definition on $\D_i$ is a continuous extension of the definition on the $1$-skeleton.

\step{8}{Estimate for $\int_{A^0A^1B^1} |D\tilde h|$.}
In this and in the following step, we aim to estimate the integral of $|D\tilde h|$ on $\S_0$; in particular, in this step we will consider the bottom triangle $A^0A^1B^1$ (by symmetry, we will get an estimate valid also for the top triangle, of course), while in the next step we will consider the situation of the horizontal strips $\D_i$. The aim of this step is to show the validity of the bound
\begin{equation}\label{step8}
\int_{A^0A^1B^1} |D\tilde h| \leq K\int_{\partial \S_0}|Dg|\, d\H^1\,,
\end{equation}
where as usual $K$ denotes a purely geometric constant. By simplicity, let us call
\[
\bar r:=\H^1(A^0A^1)=\H^1(A^0B^1)\,.
\]
Recall that, on the bottom triangle, the function $\tilde h$ has been defined as an affine function, if the angle $\angle{\u A^1}{\u A^0}{\u B^1}$, pointing inside $\u\S_0$, is smaller than $\pi$ --or, equivalently, if the curve $\gamma^1$ coincides with the segment $\u A^1 \u B^1$-- and as two degenerate affine functions on the two triangles $A^1 P A^0$ and $A^0 P B^1$ (being $P$ as in Step~7) otherwise. Let us then estimate the $L^1$ norm of $D\tilde h$ on the bottom triangle in both cases.\par

First of all, consider the non-degenerate case when $\tilde h$ is a single affine function on the bottom triangle. In particular, the image of the segment $A^0A^1$ is the segment $\u A^0\u A^1$, while the image of the segment $A^0B^1$ is the segment $\u A^0\u B^1$; this implies that, on the bottom triangle, one has
\begin{align*}
\frac{\sqrt 2}2 \,\big|D^b_1\tilde h + D^b_2 \tilde h\big| = \frac{\H^1\big(\u A^0\u B^1\big)}{\H^1\big(A^0B^1\big)}\,, &&
\frac{\sqrt 2}2 \,\big|-D^b_1\tilde h + D^b_2 \tilde h\big| = \frac{\H^1\big(\u A^0\u A^1\big)}{\H^1\big(A^0A^1\big)}\,,
\end{align*}
where by $D^b_1\tilde h$ and $D^b_2\tilde h$ we denote the constant value of $D_1\tilde h$ and $D_2 \tilde h$ on the bottom triangle. This readily implies
\begin{equation}\label{touse1}
| D^b \tilde h| \leq \frac{\H^1\big(\u A^0\u B^1\big)}{\H^1\big(A^0B^1\big)} + \frac{\H^1\big(\u A^0\u A^1\big)}{\H^1\big(A^0A^1\big)}
=\frac{\H^1\big(\u A^0\u B^1\big)+\H^1\big(\u A^0\u A^1\big)}{\bar r}\,.
\end{equation}
On the other hand,
\begin{align*}
\H^1\big(\u A^0 \u B^1) = \int_{A^0}^{B^1} |Dg|\, d\H^1\,, &&
\H^1\big(\u A^0 \u A^1) = \int_{A^0}^{A^1} |Dg|\, d\H^1\,,
\end{align*}
which inserted in~(\ref{touse1}), and using~(\ref{goodpoint}) from Step~1, gives
\[
| D^b \tilde h| \leq \frac 1{\bar r} \int_{\B(V_1,\bar r)\cap\partial \S_0} |Dg|\, d\H^1
\leq K \int_{\partial \S_0}|Dg|\, d\H^1\,.
\]
Hence, we deduce that
\[
\int_{A^0A^1B^1} |D\tilde h| = \frac{\bar r^2}2\, |D^b\tilde h| 
\leq \frac{K\bar r^2}2\,\int_{\partial \S_0}|Dg|\, d\H^1
\leq K\int_{\partial \S_0}|Dg|\, d\H^1\,,
\]
thus the validity of~(\ref{step8}) follows.\par
Let us now consider the degenerate case, where in the bottom triangle the function $\tilde h$ is made by two degenerate affine pieces, one on the left triangle $A^1PA^0$ and the other on the right triangle $A^0PB^1$; we call $D^l\tilde h$ and $D^r \tilde h$ the constant values of $D\tilde h$ respectively on the left and on the right triangle. Since the image of the segment $A^1B^1$ through the map $\tilde h$ is the path $\gamma^1$ (that is, the union of the two segments $\u A^1\u A^0$ and $\u A^0 \u B^1$), parametrized at constant speed, we get that $|D_1^l \tilde h| = |D_1^r \tilde h|$ (while in general $D_1^l \tilde h\neq D_1^r \tilde h$); more precisely,
\begin{equation}\label{touse2}
|D_1^l \tilde h| = |D_1^r \tilde h|= \frac{\H^1\big(\u A^1\u A^0\big)+\H^1\big(\u A^0\u B^1\big)}{\H^1\big(A^1B^1\big)} \,.
\end{equation}
Moreover, the affine map in the right triangle transforms the segment $A^0B^1$ in the segment $\u A^0\u B^1$, while the affine map in the left triangle moves $A^0A^1$ onto $\u A^0\u A^1$; this implies
\begin{align*}
\frac{\sqrt 2}2 \,\big|D^r_1\tilde h + D^r_2 \tilde h\big| = \frac{\H^1\big(\u A^0\u B^1\big)}{\H^1\big(A^0B^1\big)}\,, &&
\frac{\sqrt 2}2 \,\big|-D^l_1\tilde h + D^l_2 \tilde h\big| = \frac{\H^1\big(\u A^0\u A^1\big)}{\H^1\big(A^0A^1\big)}\,,
\end{align*}
which together with~(\ref{touse2}) gives
\begin{align*}
|D^l \tilde h| \leq \frac 3{\bar r}\,\int_{\B(V_1,\bar r)\cap\partial \S_0} |Dg|\, d\H^1\,, &&
|D^r \tilde h| \leq \frac 3{\bar r}\,\int_{\B(V_1,\bar r)\cap\partial \S_0} |Dg|\, d\H^1\,.
\end{align*}
Arguing exactly as before, again thanks to~(\ref{goodpoint}) of Step~1, we obtain again the validity of~(\ref{step8}), possibly with a slightly larger, but still purely geometric, constant $K$.

\step{9}{Estimate for $\int_{\D_i} |D\tilde h|$.}
In this step, we want again to find an estimate for the integral of $|D\tilde h|$, but this time on the generic horizontal strip $\D_i$, $1\leq i \leq k-2$. Our goal is to obtain the estimate
\begin{equation}\label{step9}
\int_{\D_i} |D\tilde h| \leq K |\D_i| \int_{\S_0} |Dg|\, d\H^1 + K \int_{A^iA^{i+1}\cup B^iB^{i+1}} |Dg|\, d\H^1\,,
\end{equation}
where by $|\D_i|$ we denote the area of the horizontal strip $\D_i$. Consider the horizontal segment $A^{i+1}B^{i+1}$: by symmetry, it is not restrictive to assume that this segment is below the $x$-axis, precisely at a distance $0<r\leq 1$ from the ``south pole'' $V_1\equiv (0,-1)$; in other words, we have that $A^{i+1}\equiv (-r,r-1)$ and $B^{i+1}\equiv (r,r-1)$. Moreover, let us call $\sigma$ the distance between the segment $A^{i+1}B^{i+1}$ and the segment $A^iB^i$, and
\begin{equation}\label{obvious}
\ell:= \max \Big\{\H^1(\u A^i\u A^{i+1}),\,\H^1(\u B^i\u B^{i+1}) \Big\}
\leq \int_{A^iA^{i+1}\cup B^iB^{i+1}} |Dg|\, d\H^1\,.
\end{equation}
Remember now that, in Step~7, we defined $\tilde h$ as the function which is affine on each of the triangles $P_jP_{j+1}Q_j$, and $P_{j+1}Q_jQ_{j+1}$, sending each of the points point $P_m$ (resp., $Q_m$) in $\S_0$ onto $\u P_m$ (resp., $\u Q_m$) in $\u \S_0$. Let us then concentrate ourselves on the generic triangle $\T=P_jP_{j+1}Q_j$ (for the triangles of the form $P_{j+1}Q_jQ_{j+1}$ the very same argument will work); since $\tilde h$ is affine on $\T$, let us for simplicity denote by $D^\tau \tilde h$ the constant value of $D\tilde h$ on $\T$.\par

First of all let us recall that, on the segment $A^{i+1}B^{i+1}$, the function $\tilde h$ has been defined as the piecewise linear function whose image is $\gamma^{i+1}$, parametrized at constant speed; this ensures that
\begin{equation}\label{late1}
\big|D^\tau_1 \tilde h\big| = \frac{\H^1(\gamma^{i+1})}{\H^1(A^{i+1}B^{i+1})}\,.
\end{equation}
We observe now that, by definition, $\gamma^{i+1}$ is the shortest path in the closure of $\u\S_0$ connecting $\u A^{i+1}$ with $\u B^{i+1}$; in particular, $\gamma^{i+1}$ is shorter than the image, through $g$, of the curve connecting $A^{i+1}$ with $B^{i+1}$ on $\partial \S_0$ passing through the south pole. This implies in particular that
\[
\H^1(\gamma^{i+1})\leq \int_{\B(V_1,\sqrt 2 r)\cap\partial\S_0} |Dg|\, d\H^1\,,
\]
which inserted in~(\ref{late1}) and recalling~(\ref{goodpoint}) gives
\begin{equation}\label{tq}
\big|D^\tau_1 \tilde h\big|\leq \frac 1{2r} \int_{\B(V_1,\sqrt 2 r)\cap\partial\S_0} |Dg|\, d\H^1\leq K\int_{\S_0} |Dg|\, d\H^1\,.
\end{equation}
Let us now use the fact that the affine map $\tilde h$ on $\T$ sends the segment $P_jQ_j$ onto the segment $\u P_j\u Q_j$. Calling, as in Figure~\ref{fig:deriv}, $d$ and $d'$ the distances between $A^{i+1}$ and $P_j$, and between $A^i$ and $Q_j$, we derive that
\begin{equation}\label{usenow}
\Big|\big(d'-d+\sigma\big) D_1^\tau\tilde h + \sigma D_2^\tau\tilde h \Big| = \H^1\big(\u P_j\u Q_j\big) \leq \ell\,,
\end{equation}
where in the last equality we have used the estimate~(\ref{estilength}) from Step~6, which is valid since $\u P_j\u Q_j$ is a vertical segment in the sense of Step~6.\par
\begin{figure}[thb]
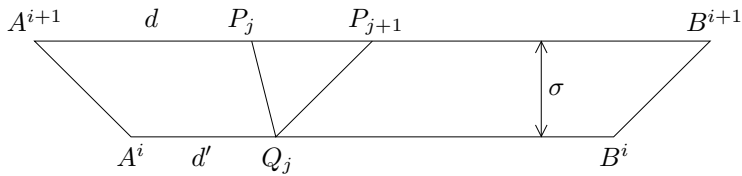
\caption{Position of points and lengths in Step~9.}\label{fig:deriv}
\end{figure}
Let us now use once again the fact that $\gamma^{i+1}$ is the shortest path between $\u A^{i+1}$ and $\u B^{i+1}$ on the closure of $\u \S_0$: in particular, $\gamma^{i+1}$ is shorter than the path obtained as the union of $\u A^{i+1}\u A^i$, the part of $\gamma^i$ between $\u A^i$ and $\u Q_j$, the segment $\u Q_j\u P_j$, and the part of $\gamma^{i+1}$ between $\u P_j$ and $\u B^{i+1}$; namely,
\[\begin{split}
\H^1(\gamma^{i+1})&\leq \H^1(\u A^{i+1}\u A^i) + \frac{d'\H^1(\gamma^i)}{\H^1(A^iB^i)} + \H^1(\u Q_j\u P_j) + \H^1(\gamma^{i+1}\big)\bigg(1-\frac{d}{\H^1(A^{i+1}B^{i+1})}\bigg)\\
&\leq 2\ell + \frac{d'\H^1(\gamma^i)}{\H^1(A^iB^i)} + \H^1(\gamma^{i+1})\bigg(1-\frac{d}{\H^1(A^{i+1}B^{i+1})}\bigg)\,,
\end{split}\]
which implies
\[
d\,\frac{\H^1(\gamma^{i+1})}{\H^1(A^{i+1}B^{i+1})}-d'\,\frac{\H^1(\gamma^i)}{\H^1(A^iB^i)} \leq 2\ell\,.
\]
The completely symmetric argument, using that $\gamma^i$ is the shortest path between $\u A^i$ and $\u B^i$, thus shorter than the union of $\u A^i\u A^{i+1}$, the part of $\gamma^{i+1}$ between $\u A^{i+1}$ and $\u P_j$, the segment $\u P_j\u Q_j$, and the part of $\gamma^i$ between $\u Q_j$ and $\u B^i$, gives the opposite inequality, hence we get
\[
\bigg|d\,\frac{\H^1(\gamma^{i+1})}{\H^1(A^{i+1}B^{i+1})}-d'\,\frac{\H^1(\gamma^i)}{\H^1(A^iB^i)}\bigg| \leq 2\ell\,,
\]
which further implies
\[\begin{split}
|d-d'|\,&\frac{\H^1(\gamma^{i+1})}{\H^1(A^{i+1}B^{i+1})}
\leq 2\ell+d'\bigg|\frac{\H^1(\gamma^{i+1})}{\H^1(A^{i+1}B^{i+1})}-\frac{\H^1(\gamma^i)}{\H^1(A^iB^i)}\bigg|\\
&\leq 2\ell+\bigg|\frac{{\H^1(A^iB^i)}}{\H^1(A^{i+1}B^{i+1})}\,\H^1(\gamma^{i+1})-\H^1(\gamma^i)\bigg|\\
&\leq 2\ell+\Big|\H^1(\gamma^{i+1})-\H^1(\gamma^i)\Big| + 2\sigma\,\frac{\H^1(\gamma^{i+1})}{\H^1(A^{i+1}B^{i+1})}
\leq 4\ell+2\sigma\,\frac{\H^1(\gamma^{i+1})}{\H^1(A^{i+1}B^{i+1})}\,.
\end{split}\]
Using now~(\ref{late1}), we can rewrite the above estimate as
\[
|d-d'| \big|D^\tau_1 \tilde h\big| \leq 4\ell + 2\sigma \big|D^\tau_1 \tilde h\big|\,,
\]
which recalling also~(\ref{usenow}) finally gives
\[
\sigma |D^\tau_2 \tilde h| \leq 5 \ell + 3\sigma |D^\tau_1 \tilde h|\,.
\]
We can then easily evaluate the integral of $|D \tilde h|$ on $\T$, also by~(\ref{tq}), as
\[\begin{split}
\int_\T |D\tilde h| &= \int_\T |D^\tau\tilde h|
\leq \int_\T |D_1^\tau\tilde h|+|D_2^\tau\tilde h|
\leq | \T| \,\bigg( 4K\int_{\S_0} |Dg|\, d\H^1+ 5\,\frac\ell\sigma\bigg)\,.
\end{split}\]
Adding now the above estimate over all the triangles $\T$ forming $\D_i$, and recalling~(\ref{obvious}) and the fact that $|\D_i|\leq \sigma$, we directly obtain~(\ref{step9}).

\step{10}{Definition of the modified function $h$ and conclusion of the proof.}
We start observing that, adding the estimates~(\ref{step8}) for the top and for the bottom triangle together with the estimates~(\ref{step9}) for all the horizontal strips, we directly obtain the validity of~(\ref{hope}) for the function $\tilde h$. However, the proof is still not over, because $\tilde h$ satisfies~(\ref{hope}), coincides with $g$ on $\partial\S_0$, and it is finitely piecewise affine, but it is not a homeomorphism (unless all the paths $\gamma^i$ lie in the interior of $\u\S_0$). However, we can easily obtain this with a simple modification of $\tilde h$: more precisely, let us slightly modify all the paths $\gamma^i$, so that they remain piecewise linear but they live in the interior of $\u\S_0$ and they do not intersect each other.
\begin{figure}[thbp]
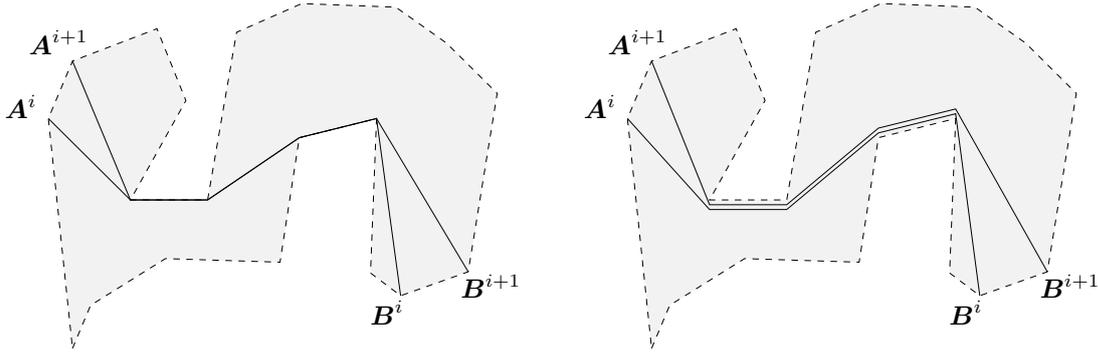
\caption{Modification of the paths $\gamma^i$ in Step~10.}\label{fig:step10}
\end{figure}
The idea, depicted in Figure~\ref{fig:step10}, is obvious. Notice that, since there are only finitely many paths $\gamma^i$, and each of them has only finitely many vertices, it is clear that we can ``separate'' as desired all the paths, and we can also move each of them of a distance which is arbitrarily smaller than all the other distances between extreme points. Then, we define the function $h$ exactly in the same way as we defined $\tilde h$, except that we use not the original paths $\gamma^i$ but the modified ones; therefore, the function $h$ is now not only finitely piecewise affine and coinciding with $g$ on $\partial\S_0$, but it is also a homeomorphism. Moreover, the estimate~(\ref{hope}) is still valid, with a geometric constant $K$ which is as close as we wish to the one found above. The proof of the theorem is then now concluded.
\end{proof}

\begin{remark}\label{ext1rem}
A trivial rotation and dilation argument proves the following generalization of Theorem~\ref{extension}. If $\S$ is a square of side $2r$ and $g:\partial\S\to\R^2$ is a piecewise linear and one-to-one function, there exists a piecewise affine extension $h:\S\to\R^2$ of $g$ such that
\begin{equation}\label{hopeext}
\int_{\S}|Dh(x)|\, dx\leq Kr \int_{\partial \S}|Dg(t)|\, d\H^1(t)\,.
\end{equation}
\end{remark}

\section{Extension in the degenerate case $|Df(c)|=0$ but $J_f(c)\neq 0$\label{sect4}}

As already explained in the description of Section~\ref{briefdesc}, a crucial difficulty in our proof will be the case when a square $\S$ is ``good'' (this means that $Df$ is almost constantly equal to some matrix $M$ within $\S$), but $\det M=0$, while being $M\neq 0$. It will be important to handle this case with care, because the map $f$ on $\S$ is then very close to an affine map, but this affine map is degenerate. The goal of this section is to prove a single result, which will solve this difficulty. Recall that, whenever a map $g$ is defined on $\partial \S$, for any $t\in\partial \S$ we denote by $\tau(t)$ the tangent vector to $\partial\S$ at $t$, by $Dg(t)$ the derivative of $g$ in the direction $\tau(t)$ (whenever it exists), and by $\int_{\partial\S} |Dg|\,d\H^1$ the length of the curve $g(\partial\S)$.

\begin{thm}\label{extension2}
Let $\S$ be a square of unit side and $g:\partial\S\to\R^2$ a piecewise linear and one-to-one function such that
\begin{equation}\label{bound}
\int_{\partial \S}\Big|D g(t)-\matx\cdot\tau(t)\Big|\, d\H^1(t)<\delta
\end{equation}
for $\delta\leq \delta_{\rm MAX}$, where $\delta_{\rm MAX}\ll 1$ is a geometric quantity. Then there is a finitely piecewise affine homeomorphism $h:\S\to \R^2$ such that $h=g$ on $\partial \S$ and
\begin{equation}\label{thesisext2}
\int_\S \Big|Dh(x)-\matx\Big|\, dx\leq K \delta\,.
\end{equation}
\end{thm}
\begin{proof}
We divide this proof in several steps, to make it as clear as possible.
\step{1}{Definition of good and bad intervals.}
First of all we notice that, since $g$ is a one-to-one piecewise linear function, then its image $g(\partial\S)$ is the boundary of a nondegenerate polygon, that we call $\u\S$. Moreover, thanks to~(\ref{bound}), we know that this polygon is very close to a horizontal segment in $\R^2$. Up to a translation, we can assume that the first coordinates of the points in $\u\S$ are between $0$ and $L$.\par
Fix now any $0<\sigma<L$: it is reasonable to expect that there are exactly two points in $g(\partial\S)$ having first coordinate $\sigma$, and that the two counterimages in $\partial\S$ are more or less one above the other (this means, with the same first coordinate). In the situation of Figure~\ref{fig:3.2st1}, this happens with $\sigma$, but not with $\sigma'$, since the points in $\partial\u \S$ with first coordinate $\sigma'$ are four. We define then ``good'' any $\sigma\in (0,L)$ with the property that exactly two points in $\partial\u\S$ have first coordinate $\sigma$, and with the additional requirement that $\sigma>2\delta$ and $\sigma<L-2\delta$.
\begin{figure}[thbp]
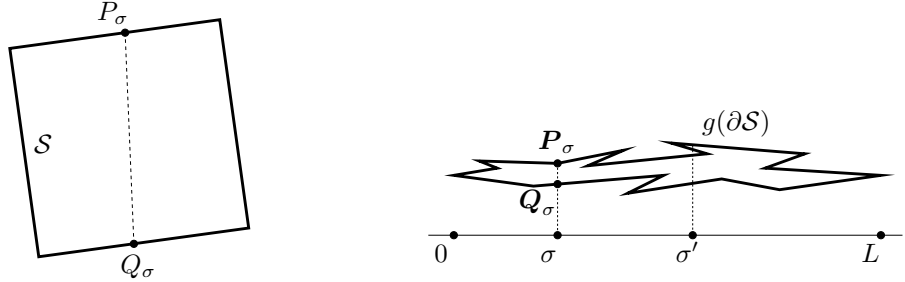~
\caption{A good $\sigma$ and a bad $\sigma'$ in Step~1.}\label{fig:3.2st1}
\end{figure}
For any such $\sigma$, we call $\u P_\sigma$ and $\u Q_\sigma$ the two above-mentioned points, being $\u P_\sigma$ above $\u Q_\sigma$, and we call $P_\sigma=g^{-1}(\u P_\sigma)$ and $Q_\sigma= g^{-1}(\u Q_\sigma)$. We can immediately show that a big percentage of the points are good, more precisely
\begin{equation}\label{lengthbad}
\H^1\Big(\Big\{\sigma\in (0,L):\, \sigma \hbox{ is not good}\Big\}\Big) \leq 5\delta\,.
\end{equation}
Indeed, take any segment $RS$ in $\partial\S$ on which $g$ is linear, and call as usual $\u R=g(R)$ and $\u S=g(S)$: by definition, we have that
\begin{equation}\label{alwaysthis}\begin{split}
\int_{RS}\Big|D g(t)-\matx\cdot\tau(t)\Big|\, d\H^1(t) &\geq 
\bigg|\int_{RS} Dg(t)\,d\H^1(t) - \int_{RS} \bigg(
\begin{matrix}
\tau_1(t)\\
0
\end{matrix}\bigg)\, d\H^1(t) \bigg|\\
&\geq \big|(\u S_1 -\u R_1)- (S_1 - R_1)\big|\,.
\end{split}\end{equation}
As a consequence, if the segment $\u R \u S$ is going backward (that is, $\u S_1 - \u R_1$ and $S_1-R_1$ have opposite sign), then its horizontal spread is bounded by the above integral on the interval $RS$. Recalling~(\ref{bound}), this means that all the backward segments have a projection on $(0,L)$ with total lenght less than $\delta$. Since of course any $\sigma\in (2\delta,L-2\delta)$ which is not good must belong to this projection, the validity of~(\ref{lengthbad}) follows.\par
By adding the inequality~(\ref{alwaysthis}) for all the segments of $\partial\S$, we find also that $L$ equals the horizontal width of $\S$ up to an error $\delta/2$, thus in particular
\[
1-\frac\delta 2 \leq L \leq \sqrt 2+\frac\delta 2\,.
\]
Moreover, take any good $\sigma$ and consider all the segments on $\partial\S$ connecting $P_\sigma$ and $Q_\sigma$: again adding~(\ref{alwaysthis}) on all these segments, and recalling that $\u P_\sigma$ and $\u Q_\sigma$ have the same first projection, we derive that
\begin{equation}\label{smallspread}
\big|(P_\sigma)_1 - (Q_\sigma)_1 \big| \leq \frac\delta 2\,,
\end{equation}
that is, the points $P_\sigma$ and $Q_\sigma$ are always exactly one above the other up to an error $\delta/2$: the factor $1/2$ comes by the possibility of choosing either of the two paths in $\S$ from $P_\sigma$ to $Q_\sigma$.\par
Finally, assume that $P_\sigma$ and $Q_\sigma$ lie on a same side of $\S$ for some good $\sigma$. Adding once again~(\ref{alwaysthis}) among all the segments where $g$ is linear connecting $P_\sigma$ and $Q_\sigma$, we derive that, up to an error $\delta$, the sum of all the horizontal spreads of these segments coincides with the corresponding sum of the horizontal spreads in $\partial\u\S$; however, the first sum is smaller than $\delta/2$ by~(\ref{smallspread}), while the second is at least the minimum between $2\sigma$ and $2(L-\sigma)$, which is impossible by the definition of good $\sigma$. In other words, we have proved that $P_\sigma$ and $Q_\sigma$ never lie on a same side of $\S$ if $\sigma$ is good.\par

Observe now that, since $g$ is piecewise linear, then by definition $(0,L)$ is a finite union of intervals, alternately done entirely by bad $\sigma$ and entirely by good ones. However, the endpoints of all these intervals are always bad. Therefore, we slightly shrink the intervals made by good points and we call \emph{good intervals} these shrinked intervals. Thanks to~(\ref{lengthbad}), we can do this in such a way that the union of the good intervals covers the whole $(0,L)$ up to a length of $6\delta$: notice that all the points of any good interval are good points, also the endpoints, while the bad intervals may also contain good points. Finally, it is convenient to make the following further slight modification: up to replace a good interval with a finite union of good intervals, we can also assume that whenever $(\sigma,\sigma')$ is a good interval, the map $g$ is linear in the segments $P_\sigma P_{\sigma'}$ and $Q_\sigma Q_{\sigma'}$.

\step{2}{The extension on the segments $P_\sigma Q_\sigma$, definition of good and bad quadrilaterals.}
In this step, we extend $g$ --which is defined on $\partial\S$-- to a union of segments in $\S$. More precisely, recall that $(0,L)$ has been divided in intervals, which can be either bad or good. Moreover, the extremes of these intervals are all good points except $0$ and $L$. Take then any other of these extremes, say $\sigma$, and consider the points $P_\sigma$ and $Q_\sigma$ in $\partial \S$. We define $g$ on the segment $P_\sigma Q_\sigma$ as the linear function such that $g(P_\sigma)=\u P_\sigma$ and $g(Q_\sigma)=\u Q_\sigma$; notice that all the different open segments $P_\sigma Q_\sigma$ are contained in the interior of $\S$ by Step~1, and they do not intersect with each other by construction. We have then many segments $P_\sigma Q_\sigma$ inside $\S$, almost vertical by~(\ref{smallspread}), on each of which a linear function $g$ is defined. Observe that, as a consequence, $\S$ has been divided in several quadrilaterals (actually, the first and the last one are generally triangles), and $g$ is defined in the whole corresponding $1$-dimensional grid; also $\u\S$ has then been subdivided by the images of $g$ in a union of several polygons. A positive consequence of this fact is that we can now define the extension $h$ of $g$ in an independent way from each quadrilateral in $\S$ to the corresponding polygon in $\u\S$ --respecting of course the boundary data, and being a piecewise affine homeomorphism: then, the resulting function $h$ will automatically be a piecewise affine homeomorphism.\par

Let us conclude this short step with another piece of notation: any quadrilateral in $\S$ will be called a \emph{good quadrilateral} if it corresponds to a good interval in $(0,L)$, and a \emph{bad quadrilateral} otherwise. In the remaining of the proof, we will first give an estimate for the good quadrilaterals; then, we will give one for the first and the last quadrilateral, that is, the one starting at $0$ and the one ending at $L$: notice that these quadrilaterals are always bad by definition, and actually they are usually triangles. Finally, we will give an estimate for the ``internal'' bad quadrilaterals, which we obtain by considering two subcases.

\step{3}{The extension in good quadrilaterals.}
Let us first start by considering a good quadrilateral, corresponding to the good interval $(\sigma,\sigma')$ in $(0,L)$; for brevity, we will write $P,\, P',\, \u P,\, \u P'$ in place of $P_\sigma,\, P_{\sigma'},\, \u P_\sigma,\,\u P_{\sigma'}$. Recall that the map $g$ is linear between $P$ and $P'$, as well as between $Q$ and $Q'$, thanks to the construction in Step~1. As a consequence, the image under $g$ of the boundary of the quadrilateral $PP'Q'Q$ is the boundary of the quadrilateral $\u P \u P' \u Q' \u Q$, and we have to define the extension $h$ by sending the interior of $PP'Q'Q$ onto the interior of $\u{PP}'\u Q'\u Q$. Let $h$ simply be the piecewise affine map sending $PP'Q$ onto $\u P \u P' \u Q$ and $QP' Q'$ onto $\u Q \u P' \u Q'$.\par
\begin{figure}[thbp]
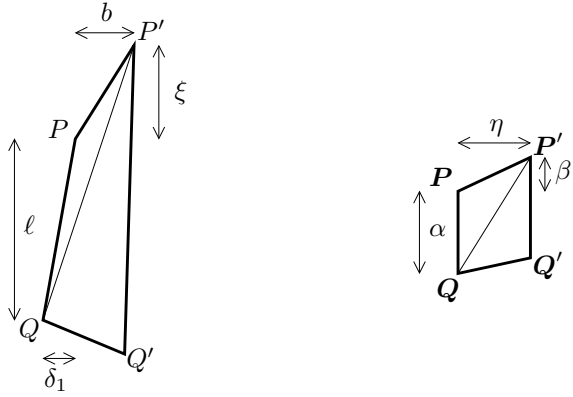
\caption{The approximation in a good quadrilateral, Step~3.}\label{fig:3.2st3}
\end{figure}
We need to estimate
\begin{equation}\label{testm}
\int_{PP'Q}\bigg|Dh-\matx\bigg| \, dx\,,
\end{equation}
the estimate in the triangle $QP'Q'$ being then of course identical. Let us define for shortness
\begin{align*}
\ell&= P_2 - Q_2\,, & b&=P'_1 - P_1\,, & \xi&=P'_2-P_2\,, & \theta&=\arctan(\xi/b)\,,\\
\delta_1 &= P_1 - Q_1 \,, &\alpha&=\u P_2 -\u Q_2\,, & \eta&= \u P'_1-\u P_1\,, & \beta &= \u P'_2 - \u P_2\,, 
\end{align*}
we refer to Figure~\ref{fig:3.2st3} for help with the notation. By definition, the constant value of $Dh$ in $PP'Q$ satisfies
\begin{equation}\label{affineDh}\begin{aligned}
b D_1 h_1 + \xi D_2 h_1 &= \eta \,,\hspace{35pt} &b D_1 h_2 + \xi D_2 h_2 = \beta\,, \\
\delta_1 D_1 h_1 + \ell D_2 h_1 &= 0 \,,  &\delta_1 D_1 h_2 + \ell D_2 h_2 = \alpha\,.
\end{aligned}\end{equation}
Let us start by defining
\begin{equation}\label{defeps}
\eps= \int_{PP'} \Big|D g(t)-\matx\cdot\tau(t)\Big|\, d\H^1(t)\,,
\end{equation}
so that adding the values of $\eps$ on the different segments we will get less than $\delta$ by~(\ref{bound}). We claim now the validity of the following estimates, all obtained again arguing as in~(\ref{alwaysthis}):
\begin{align}\label{various}
|\eta-b| \leq \eps\,, && |\beta|\leq \eps\,, && \alpha \leq \delta\,, && |\delta_1| \leq \frac\delta 2 \,, && \ell > \delta\max\{\tan\theta,\,1\}\,.
\end{align}
The first two estimates can be found just integrating on the segment $PP'$, so they are valid with the small constant $\eps$; instead, to get the third estimate we have to integrate on all the segments connecting $P$ and $Q$ on $\partial\S$, so we can only estimate with $\delta$; the fourth estimate is given by~(\ref{smallspread}). Finally, the evaluation of $\ell$ follows by a simple geometric argument, just recalling that $\sigma>2\delta$, (\ref{smallspread}) and that we have defined $\theta$ as the direction of the side containing $PP'$.\par

Let us now start by evaluating $D_1 h_1$: inserting the third equation of~(\ref{affineDh}) into the first one, we get
\[
D_1 h_1 \bigg( b - \frac{\xi\delta_1}\ell \bigg) = \eta\,,
\]
from which we readily obtain, by using the estimates~(\ref{various}) and recalling that $\xi/b=\tan\theta$,
\begin{equation}\label{est11}
\big| D_1 h_1 -1 \big| \leq 2 \bigg( \frac \eps b + \frac{\xi\delta}{b\ell}\bigg)\,.
\end{equation}
Substituting the value of $D_1 h_1$ again in the third equation of~(\ref{affineDh}), we get then
\begin{equation}\label{est21}
\big| D_2 h_1 \big| = \frac{\delta_1}\ell \big|D_1 h_1\big|
\leq \frac \delta{2\ell} + \frac \eps b  + \frac{\xi\delta}{b\ell}\,.
\end{equation}
We control now the derivatives of $h_2$: inserting the second equation of~(\ref{affineDh}) into the fourth, we get
\[
D_2 h_2  \bigg(\ell - \delta_1 \,\frac \xi b \bigg) = \alpha - \frac{\delta_1 \beta}b
\]
so that, again using~(\ref{various}) and again recalling that $\xi/b=\tan\theta$, we deduce
\begin{align}\label{est2212}
\big|D_2 h_2 \big|  \leq 2\,\frac{\delta}\ell + \frac{\delta \eps}{b\ell}\,, && 
\big|D_1 h_2 \big|  \leq 2\,\frac \eps b + 2 \,\frac{\xi\delta}{b\ell} \,.
\end{align}
Estimating the integral in~(\ref{testm}) is then straightforward. Notice that $\delta$ is a fixed constant, not depending on the subdivision in intervals: as a consequence, we can assume without loss of generality that $\xi\leq \delta<\ell$, otherwise it is enough to subdivide a good interval in a finite union of good intervals; the area of the triangle $PP'Q$ is then less than $b\ell$, and so from~(\ref{est11}), (\ref{est21}) and~(\ref{est2212}) we obtain
\[
\int_{PP'Q}\bigg|Dh-\matx\bigg| \, dx
\leq \bigg(5 \, \frac \eps b + 5\,\frac{\xi\delta}{b\ell}+3\,\frac\delta\ell + \frac{\delta \eps}{b\ell}\bigg)\cdot b\ell
\leq 8\eps + \delta\big(5\xi + 3 b\big)\,,
\]
where we have also used that $\delta\leq 1/2$ and $\ell\leq 3/2$ (the latter follows by straightforward geometrical arguments). Of course, the fully analogous estimate holds for the integral in the triangle $QP'Q'$, up to replace the segment $PP'$ by $QQ'$ in the definition~(\ref{defeps}) of $\eps$.\par

To conclude, we need to evaluate the total integral in the union of the good quadrilaterals; this is simply achieved by summing the above estimates over all the different quadrilaterals. Notice that the constant $\delta$ is fixed and does not depend on the quadrilateral, while the constants $\eps,\, b$ and $\xi$ are specific of each quadrilateral. By definition~(\ref{defeps}) of $\eps$, it appears clear that the sum of all the different $\eps$'s is less than $\delta$, while by definition of the lengths in the square it is clear that the sum of the different $\xi$, as well as of the different $b$, is bounded by $4$. As a consequence, we deduce that
\begin{equation}\label{intgood}
\int_G\bigg|Dh(x)-\matx\bigg| \, dx \leq K \delta\,,
\end{equation}
where $G$ denotes the union of all the good quadrilaterals in $\S$, while $K$ is as usual a purely geometric constant.

\step{4}{The extension in the first and last bad quadrilaterals.}
In this and in the next step we are going to consider the bad quadrilaterals. Notice that, since almost the whole square is done by good quadrilaterals thanks to~(\ref{lengthbad}), we can even be satisfied by a rough estimate here, while we needed a precise one in the preceding step: what is important, is that we can define a piecewise affine homeormophism $h$ on each of the bad quadrilaterals.

 Here we begin with the ``first'' and the ``last'' quadrilateral, that is, with the quadrilaterals which correspond to the two intervals having $0$ or $L$ as one endpoint. Notice that these ``quadrilaterals'' are actually triangles, unless some side of the square is very close to being vertical. More precisely, let us consider just the first bad quadrilateral $\C$, by symmetry: as Figure~\ref{fig:3.2st4} depicts, it can be either a triangle $VPQ$, being $V$ the left vertex of the square, or a quadrilateral $VV'PQ$, if $V$ and $V'$ are the two left vertices of the square, being the side $VV'$ almost vertical, and then the sides $V'P$ and $VQ$ almost horizontal. Notice that all the sides of $\C$ belong to $\partial \S$ except $PQ$.
\begin{figure}[thbp]
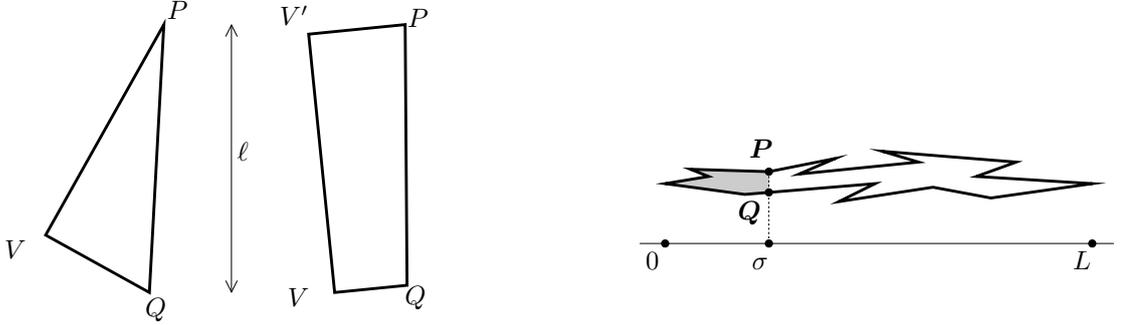
\caption{The approximation in the first bad quadrilateral, Step~4.}\label{fig:3.2st4}
\end{figure}
We need to send $\C$ on the polygon $\u\C$ inside $\u\S$ made by the points which have first coordinate less than $\sigma=\u P_1 =\u Q_1$, shaded in the right of the figure. Keep in mind that by construction (recall Step~1) the coordinate $\sigma$ is good, and the bad intervals cover only a portion less than $2\delta$ of $(2\delta, L-2\delta)$: this means that $2\delta \leq \sigma \leq 4\delta$. As a consequence, again by using several times~(\ref{alwaysthis}) and~(\ref{smallspread}), we know that
\begin{align*}
\u V_1 \leq \frac \delta 2 \,, &&
\frac \delta 2\leq P_1- V_1\leq 5\delta\,, &&
\frac \delta 2\leq Q_1- V_1\leq 5\delta\,, &&
|Q_1 - P_1 | \leq \frac \delta 2\,,
\end{align*}
and the estimates on $V$ are valid also for $V'$ in the case when the bad quadrilateral $\C$ is actually a quadrilateral. We would like to infer that $\C$ is the biLipschitz image of a square with side $\delta$, with uniformly bounded biLipschitz constant; however, this is true only if $\ell$ is comparable to $\delta$, while we only know by Step~3 that $\ell \geq \delta$ --this was established in~(\ref{various}). Let us then consider the affine map $\Phi(x_1,x_2)=(x_1,x_2\delta/\ell)$, and let $\widetilde\C=\Phi(\C)$; define also $\tilde g=g \circ\Phi^{-1}$ on $\partial\widetilde\C$, which is admissible since $g$ is defined on the whole $\partial\C$. By construction, the set $\widetilde\C$ is the biLipschitz image of a square of side $\delta$, with biLipschitz constant less than a geometrical constant $K$. We can then apply Theorem~\ref{extension} to the map $\tilde g$, and we find an extension $\tilde h$ of $\tilde g$ inside $\widetilde\C$ such that~(\ref{hope}) holds, that is,
\[
\int_{\widetilde\C} |D\tilde h(y)|\, dy\leq K \delta\int_{\partial \widetilde\C}|D\tilde g(t)|\, d\H^1(t)
\]
(notice that the multiplication by $\delta$ comes from the argument of Remark~\ref{ext1rem}). Observe now that the integral in the right side of the above inequality is simply the perimeter of $\u\C$, which is less than $K\delta$ by~(\ref{bound}) and again by~(\ref{alwaysthis}). Thus, we infer
\begin{equation}\label{almostdone}
\int_{\widetilde\C} |D\tilde h(y)|\, dy\leq K \delta^2\,.
\end{equation}
Finally, we define $h=\tilde h \circ\Phi$ on $\C$: this is a piecewise affine homeomorphism from $\C$ to $\u\C$, and by definition it extends the map $h$ already defined on $\partial\C$. We have then to show that $Dh$ is not too big on $\C$, and to do so it is enough to observe that
\[
\big|Dh\big(\Phi^{-1}(y)\big)\big| \leq \big| D\tilde h(y)\big|\,,
\]
which by~(\ref{almostdone}) finally implies
\begin{equation}\label{priultqua}
\int_\C\ | Dh(x) |\, dx 
\leq \int_{\widetilde\C} \big|Dh\big(\Phi^{-1}(y)\big)\big| \,\frac \ell \delta\, dy
\leq \frac 2 \delta\, \int_{\widetilde\C} \big| D\tilde h(y)\big| \,dy\leq K \delta\,.
\end{equation}
We have then found the estimate we were looking for related to the first bad quadrilateral, and by symmetry the same holds also in the last bad quadrilateral.

\step{5}{The extension in the internal bad quadrilaterals.}
To conclude our analysis, we need to concentrate in the internal bad quadrilaterals. Let $\C$ be a bad quadrilateral, and let us call its vertices, as usual, $P,\, Q,\, P'$ and $Q'$; the image of $\partial\C$ under $g$ is then the boundary of a polygon $\u\C$. Notice that $\u\C$ needs not to be a quadrilateral, since it has two vertical sides, $\u{PQ}$ and $\u P' \u Q'$, but $\arc{\u{PP}'}$ and $\arc{\u{QQ}'}$ are piecewise linear paths, not necessarily segments. Keeping a notation similar to that of Step~3, we set
\begin{align*}
\ell&= P_2 - Q_2\,, & b&=P'_1 - P_1\,, & \alpha&=\max\{\u P_2,\, \u P'_2\} -\min\{\u Q_2,\, \u Q'_2\}\,,  \\
\xi&=P'_2-P_2\,, & \theta&=\arctan(\xi/b)\,, & \eta&= \H^1\big(\arc{\u P\u P'}\big)+\H^1\big(\arc{\u Q\u Q'}\big)\,,
\end{align*}
see Figures~\ref{fig:3.2st5} and~\ref{fig:3.2st6}. By a simple symmetry argument, we can assume without loss of generality that
\begin{align}\label{posstheta}
\hbox{$\theta\geq 0$, and} && \hbox{either $PP'$ and $QQ'$ are parallel}\,, && \hbox{or $\theta\geq \pi/4$}\,.
\end{align}
Observe that this is possible because, if $PP'$ and $QQ'$ are not parallel, then they belong to two consecutive sides of the square, hence if $\theta\leq \pi/4$ we just have to exchange $P$ with $Q$. Notice that, by definition,
\begin{equation}\label{propeta}
\eta= \int_{PP'\cup QQ'} \big|D g(t)\big|\, d\H^1(t)\,.
\end{equation}
We need now to further subdivide our analysis in two subcases, depending whether $\alpha$ is bigger or smaller than $10\eta$. Notice that $\alpha$ is bounded by $\delta$, while $\eta$ could be even much smaller than $\delta$, since the sum of all the different $\eta$'s corresponding to bad intervals is smaller than $3\delta$: indeed, the total length of the internal bad intervals is less than $2\delta$, so we do not even need to subtract the matrix $\Big(\,\begin{matrix}1\ 0\\[-3pt] 0\ 0 \end{matrix}\,\Big)$ as in~(\ref{defeps}). As a consequence, either of the two cases may actually hold.

\step{5a}{The case $\alpha\leq 10\eta$.}
Let us start with the case when $\alpha\leq 10\eta$. We let $H$ be the point in the segment $P'Q'$ satisfying $P_2=H_2$ (such a point exists by~(\ref{posstheta})), and $\u H=g(H)$, which is well defined since $g$ has been defined on the good segment $P'Q'$.
\begin{figure}[thbp]
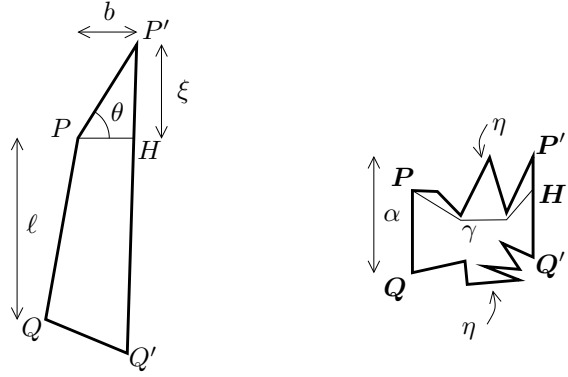
\caption{The approximation in an internal bad quadrilateral: case~1, Step~5a.}\label{fig:3.2st5}
\end{figure}
We subdivide the quadrilateral $\C$ into the union of the triangle $PP'H$ and the quadrilateral $PHQ'Q$, and we aim to define the function $h$ separately one these two pieces. First of all, similarly as in the proof of Theorem~\ref{extension}, we consider the shortest path between $\u P$ and $\u H$ in $\u\C$, which is a piecewise affine path, possibly intersecting $\partial\u\C$ in other points than $\u P$ and $\u H$, and we call $\gamma$ a slight modification of this path, which is still piecewise affine, but which is entirely in the interior of $\u\C$ except for the two extremes $\u P$ and $\u H$. By minimality, we can of course take the modified $\gamma$ satisfying
\begin{equation}\label{estigamma}
\H^1(\gamma) < \H^1\big(\arc{\u{PP}'}\big) + \H^1(\u P'\u H)\,.
\end{equation}
We extend then $g$ to the segment $PH$ as the piecewise affine function sending the segment $PH$ onto the path $\gamma$ at constant speed.\par

Let us now point our attention on the triangle $PP'H$: the segment $PH$ is horizontal by definition, while the segment $P'H$ is ``quite vertical''; more precisely, it is contained in the segment $P'Q'$ and by definition we have
\begin{align*}
|P'_1-Q'_1|\leq \frac \delta 2\,, && P'_2 - Q'_2 \geq \ell \geq \delta\,.
\end{align*}
The triangle would then be a biLipschitz image of a square with side $b$, with uniformly bounded constant, if $\xi$ were not too much bigger than $b$, or, in other words, if $\theta$ is not too big. Since we cannot be sure that this is the case, exactly as in Step~4 we define $\Phi$ the affine map which does not modify the horizontal segments, and which shrinks of a ratio $\xi/b$ the segments parallel to $P'H$. Then, $\Phi(PP'H)$ is a triangle which is uniformly biLipschitz with a square of side $b$, so exactly as in Step~4 we apply Theorem~\ref{extension} to the map $\tilde g=g\circ \Phi^{-1}$ finding an extension $\tilde h$ on $\Phi(PP'H)$, and we finally obtain the extension to $g$ in $PP'H$ as $h=\tilde h\circ\Phi$. Estimating the derivatives of $h$, $\tilde h$, $g$ and $\tilde g$ exactly as in Step~4, we get then the estimate
\[\begin{split}
\int_{PP'H} |Dh(&x)|\,dx \leq
K\, \frac\xi b \int_{\Phi(PP'H)} |D\tilde h(y)|\, dy
\leq K \xi \int_{\partial(\Phi(PP'H))} |D\tilde g(t)|\, dt\\
&= K \xi\, \H^1\Big(\partial \big(g(PP'H)\big)\Big)
= K \xi\Big( \H^1\big(\arc{\u{PP}'}\big) + \H^1(\gamma) + \H^1(\u P'\u H)\Big)\\
&\leq K \Big( \H^1\big(\arc{\u{PP}'}\big) +\H^1(\u P'\u H)\Big)
\leq K\big(\eta + \alpha\big)
\leq K\eta\,,
\end{split}\]
where we have also used~(\ref{estigamma}) and the assumption $\alpha\leq 10\eta$.\par

Let us now consider the quadrilateral $PHQ'Q$. Since we have already seen that $PQ$ and $HQ'$ are ``quite vertical'', while $PH$ is exactly horizontal and $QQ'$ is ``quite horizontal'' since it makes with the horizontal direction an angle equal to $\pi/2 - \theta \leq \pi/4$, this quadrilateral is uniformly biLipschitz with a rectangle. Up to shrinking vertically of a ratio $b/\ell$ as before, it becomes uniformly biLipschitz with a square of side $b$, hence by arguing as before by shrinking, applying Theorem~\ref{extension} and then stretching back, we define an extension $h$ of $g$ inside the quadrilateral $PHQ'Q$ which satisfies
\[
\int_{PHQ'Q} |Dh(x)|\, dx \leq K \ell \Big(\H^1(\u{PQ}) + \H^1(\u{HQ}')+\H^1(\gamma) + \H^1\big(\arc{\u{QQ}'}\big)\Big)\leq K \eta\,,
\]
which put together with the estimate above for the triangle $PP'H$ gives
\begin{equation}\label{intbad5a}
\int_\C |Dh(x)|\, dx \leq K \eta\,.
\end{equation}

\step{5b}{The case $\alpha>10\eta$.}
The last case that we have to consider is that of an internal bad quadrilateral corresponding to $\alpha>10\eta$: in this case, if we argued as in Step~5a, then we would find the same estimates as in~(\ref{intbad5a}), but with $K\delta$ in place of $K\eta$; and in turn, this would not be acceptable, because adding all the different $\eta$'s for the bad quadrilaterals we get something smaller than $\delta$, while adding a term $\delta$ in each of the bad quadrilaterals we could get any big constant in the end, since bad intervals could be many more than $1/\delta$.\par

As a consequence, in this substep we present a different definition of the extension $h$ for the case $\alpha>10\eta$.
\begin{figure}[thbp]
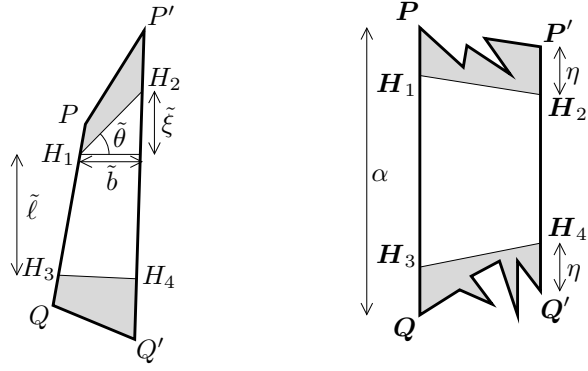
\caption{The approximation in an internal bad quadrilateral: case~2, Step~5b.}\label{fig:3.2st6}
\end{figure}
More precisely, take four points $\u H_1,\, \u H_2,\, \u H_3,\, \u H_4$, as in Figure~\ref{fig:3.2st6}, in the segments $\u{PQ}$ and $\u P'\u Q'$, at a distance $\eta$ from the four vertices; let also $H_i = g^{-1}(\u H_i)$ for $i=1,\,2,\,3,\,4$. By definition of $\eta$, the open segments $\u H_1\u H_2$ and $\u H_3\u H_4$ are entirely contained in the interior of $\u\C$, and by construction the same happens for the segments $H_1H_2$ and $H_3H_4$ in $\C$. We regard then both $\C$ and $\u\C$ as the union of three pieces: the internal quadrilaterals $H_1H_2H_4H_3$ and $\u H_1\u H_2 \u H_4\u H_3$, and the ``top'' and ``bottom'' remaining parts, shaded in Figure~\ref{fig:3.2st6}. We aim to define the piecewise affine function $h$ so to send each part of $\C$ onto the corresponding one in $\u\C$.\par

For the ``top'' and ``bottom'' part, we can argue more or less exactly as in the last steps: each of the quadrilateral $PP'H_2H_1$ and $H_3H_4Q'Q$ can be transformed into a square, then one applies Theorem~\ref{extension} and then goes back to the quadrilateral; since the perimeter of each of the shaded regions in $\u\C$ is now at most $4\eta$, the same estimates as in Step~5a can be repeated, so that similarly to~(\ref{intbad5a}) we get now
\begin{equation}\label{stibr1}
\int_{PP'H_2H_1\cup H_3H_4Q'Q} |Dh(x)|\,dx \leq K\eta\,.
\end{equation}
To conclude, we have to define the extension $h$ so to send the internal quadrilateral of $\C$ onto the internal quadrilateral of $\u\C$, and we will do that again by sending in an affine way the triangle $H_1H_2H_3$ (resp. $H_3H_2H_4$) onto the triangle $\u H_1\u H_2\u H_3$ (resp. $\u H_3\u H_2\u H_4$). We need thus only to check the value of $|Dh|$ on the triangle $H_1H_2H_3$, being the estimate for the triangle $H_3H_2H_4$ completely similar. To get the estimate, as in Figure~\ref{fig:3.2st6} we set
\begin{align*}
H_1-H_3 = (\tilde\delta_1,\tilde\ell)\,, && H_2-H_1=(\tilde b,\tilde\xi)\,, &&\tilde\theta=\arctan\, \frac{\tilde\xi}{\tilde b}\,;
\end{align*}
notice that $\tilde\delta_1,\,\tilde\ell,\,\tilde b,\,\tilde \xi$ and $\tilde\theta$ are very similar to $\delta_1,\,\ell,\, b,\, \xi$ and $\theta$, since the assumption $\alpha>10\eta$ implies that the points $H_i$ are very close to the vertices of $\C$. Hence, the constant matrix $Dh$ in $H_1H_2H_3$ satisfies
\begin{align}\label{estigene}
\begin{array}{l}
\big|Dh(\tilde\delta_1,\tilde\ell) \big|= \big|(\alpha-2\eta,0)\big|\leq \delta\,, \\[7pt]
\big|Dh(\tilde b,\tilde b\tan\theta)\big|=\big|Dh(\tilde b,\tilde\xi)\big| = \big|\u H_2 - \u H_1\big|\leq \eta\,.
\end{array}
\end{align}
As a consequence, we get first that
\[
\big|Dh(\tilde \delta_1,\tilde\delta_1\tan\theta)\big| \leq \frac{\tilde\delta_1}{\tilde b}\, \eta
\leq \delta\, \frac \eta b\,,
\]
and then that
\begin{equation}\label{fafe}
\big|Dh(0,\tilde\ell-\tilde\delta_1\tan\theta)\big| \leq \delta \bigg(1+ \frac \eta b\bigg)\,.
\end{equation}
Recall now that the estimates~(\ref{various}) ensure
\[
\frac{\tilde\ell}{\tilde\delta_1}=\frac\ell{\delta_1} \geq 2\tan\theta \geq 2\tan\tilde\theta\,;
\]
notice carefully that the estimates~(\ref{various}) were obtained in a good quadrilateral, so they are not valid now, but since the segment $PQ$ corresponds to a good $\sigma$, and in particular it is in the boundary of the good quadrilateral immediately preceding $\C$, the estimates about $\ell$ and $\delta_1$ are still valid and we can use them now. As a consequence, by~(\ref{fafe}) we get
\begin{equation}\label{fin1}
| D_2 h| \leq 2\, \frac \delta{\tilde\ell}\bigg(1+ \frac \eta b\bigg)
\leq 3\,\frac\delta\ell + 3\, \frac{\delta\eta}{b\ell}\,,
\end{equation}
and substituting this in~(\ref{estigene}) we have also
\[
\big| D h(\tilde b,0)\big| \leq \eta + \tilde b \tan \theta |D_2 h| 
\leq \eta + 2 \tilde b \tan \theta\, \frac \delta{\tilde \ell} \bigg(1+ \frac \eta b\bigg)
\leq \eta + 3 \tilde b \bigg(1+ \frac \eta b\bigg)
\leq 4\eta+3 \tilde b\,,
\]
from which we derive
\begin{equation}\label{fin2}
\big| D_1 h\big|  \leq 5\,\frac\eta b+3\,.
\end{equation}
Since the area of the triangle $H_1H_2H_3$ is bounded by $b(\ell + b\tan\theta)$, by~(\ref{fin1}) and~(\ref{fin2}), and recalling again $\delta\tan\theta\leq \ell$, we get
\[\begin{split}
\int_{H_1H_2H_3} |Dh(x)|\,dx &\leq \big(b \ell + b^2\tan\theta \big)\bigg(3\,\frac\delta\ell + 3\, \frac{\delta\eta}{b\ell} + 5\,\frac\eta b+3\bigg)\\
&\leq 3 \delta b + 3 b^2 + 3 \delta \eta + 3 \eta b +5 \eta \ell + 5 \eta \xi + 3 b \ell + 3 b \xi
\leq K (b + \eta )\,,
\end{split}\]
so that repeating the same estimate in the triangle $H_3H_2H_4$, and adding~(\ref{stibr1}), we obtain that in a bad quadrilateral $\C$ where $\alpha>10\eta$ the estimate
\begin{equation}\label{intbad5b}
\int_\C |Dh(x)|\,dx \leq Kb + K\eta
\end{equation}
holds.

\step{6}{Conclusion.}
We can now put together all the estimates of the last steps to conclude the proof. Let us start by considering the bad quadrilaterals, whose union is $\S\setminus G$, since in Step~3 we have defined $G$ as the union of the good quadrilaterals. Thanks to~(\ref{intbad5a}) and~(\ref{intbad5b}), we know that the integral of $|Dh|$ in any internal bad quadrilateral can always be estimated by $b+\eta$. If we add the different $b$'s corresponding to the bad quadrilaterals, up to an error $\delta$ we find the sum of the lengths of the internal bad intervals, which is at most $2\delta$ by construction. On the other hand, adding the different $\eta$'s and recalling~(\ref{propeta}), we get something smaller than (twice) the sum of the lengths of the bad intervals. As a consequence, putting together the estimates for all the internal bad quadrilaterals, and also adding the estimate~(\ref{priultqua}) for the first and the last bad quadrilateral, we obtain
\[
\int_{\S\setminus G} |Dh(x)|\,dx \leq K \delta\,.
\]
Since the total area of the bad quadrilaterals can be estimated by (twice) the total length of their horizontal projections, which in turn corresponds with the total length of the bad intervals up to an error $\delta$, and so it is less than $7\delta$, we can now insert~(\ref{intgood}) to get
\[\begin{split}
\int_\S\bigg|Dh(x)-\matx\bigg| \, dx &=
\int_G\bigg|Dh(x)-\matx\bigg| \, dx+\int_{\S\setminus G}\bigg|Dh(x)-\matx\bigg| \, dx\\
&\leq K \delta + \int_{\S\setminus G}|Dh(x)| \, dx + \big| \S\setminus G\big|\leq K\delta\,,
\end{split}\]
which is~(\ref{thesisext2}), and then the proof is concluded.
\end{proof}

\begin{remark}\label{ext2rem}
A trivial rotation and dilation argument proves the following generalization of Theorem~\ref{extension2}: whenever $\S$ is a square of side $2r$, $g:\partial\S\to\R^2$ is a piecewise linear and one-to-one function, and $M$ is a matrix with $\det M=0$, there is a finitely piecewise affine extension $h:\S\to\R^2$ of $g$ such that
\[
\int_\S \big|Dh(x)-M\big|\, dx\leq Kr \int_{\partial \S}\big|D g(t)-M\cdot\tau(t)\big|\, d\H^1(t)\,,
\]
as soon as
\[
\int_{\partial \S}\big|D g(t)-M\cdot\tau(t)\big|\, d\H^1(t)< r\delta_{\rm MAX} \|M\|\,.
\]
\end{remark}

\section{Proof of Theorem~\ref{main}\label{sect5}}

This last section is devoted to give the proof of Theorem~\ref{main}. This proof is still quite involved, but the overall idea is simple after the preceding sections. As already explained in the introduction, the idea is to divide the whole $\Omega$ in squares, and then treat them in three different ways: roughly speaking, the ``good'' squares, where the function is very close to an affine map (and this group will be further divided in two subgroups), and the ``bad'' ones, where this is not true. Moreover, we will have to slightly change the value of $f$ on the boundaries of all these squares, in order to become piecewise linear. Then, in the bad squares we will simply use Theorem~\ref{extension} to get an extension, and the constant $K$ in~(\ref{hope}) will not be a problem because the bad squares will cover only a small portion of $\Omega$. Instead, we have to perform a very precise approximation of $f$ in the good squares; to do so, we will treat differently the squares where the affine map close to $f$ has zero determinant, and those where the determinant is strictly positive. For the first ones, we will use Theorem~\ref{extension2}, while for the second ones it will be enough to interpolate the values of $f$ on the boundary, as we show in Section~\ref{secgoodsquares}.\par

Before starting with the proof, let us give a couple of definitions.
\begin{definition}\label{Lebsquare}
We say that $\S(c,r)$ is a \emph{Lebesgue square with matrix $M\in \R^{2\times 2}$ and constant $\delta>0$} if $\S(c,3r)\subseteq\Omega$ and
\[
\intmed_{\S(c,3r)} |Df(z)-M| \, dz \leq \delta\,.
\]
\end{definition}
\begin{definition}\label{defphi_}
Let $\S(x,r)\subseteq \Omega$ be a square, and denote by $T_1$ and $T_2$ the two triangles on which $\S$ is subdivided by the diagonal connecting $(x_1-r,x_2+r)$ and $(x_1+r,x_2-r)$. We call $\varphi_{\S(x,r)}$ the piecewise affine function which is affine on the two triangles $T_1$ and $T_2$, and which coincides with $f$ on the four vertices of $\S(x,r)$.
\end{definition}

\subsection{The Lebesgue squares\label{secgoodsquares}}
In this first subsection, we consider the situation in the best possible case, namely, a Lebesgue square. It is rather easy to show the following uniform estimate.
\begin{lemma}\label{genlemma}
For every $\eps>0$ and every matrix $M$, there exists $\bar\delta=\bar\delta(M,\eps)\ll \eps$ such that the following holds: if $\S(c,r)$ is a Lebesgue square with matrix $M$ and constant $\delta\leq \bar\delta$,
\begin{align}\label{lemmatoprove}
\| f - \varphi\|_{L^\infty(\S(c,r))} \leq r\eps \,, && \|Df - D\varphi\|_{L^1(\S(c,r))} \leq r^2 \eps  \,, &&
\| f - \psi\|_{L^\infty(\S(c,2r))} \leq \frac{r\eps}{10}\,,
\end{align}
where $\varphi=\varphi_{\S(c,r)}$ is the piecewise affine map of Definition~\ref{defphi_} and $\psi:\R^2\to\R^2$ is an affine function satisfyiny $D\psi=M$. If in addition $\det M>0$, then $\varphi$ is injective and
\begin{equation}\label{stp}
f\big(\S(c,(1-\eps)r)\big)\subseteq \varphi(\S(c,r))\subseteq f\big(\S(c,(1+\eps)r)\big)\,.
\end{equation}
\end{lemma}
\begin{proof}
We assume for simplicity of notation that the point $c$ is the origin of $\R^2$, and we write $\S(r)$ instead of $\S(0,r)$. Let $R$ be a big constant, depending only on $M$ and $\eps$, to be specified later, and let us define the two sets $A$ and $B$ as
\begin{align*}
A&=\Big\{x\in (-3r,3r):\, \intmed_{-3r}^{3r} |Df(x,t)-M|\, dt \geq R\delta \Big\}\,, \\
B&=\Big\{y\in (-3r,3r):\, \intmed_{-3r}^{3r} |Df(t,y)-M|\, dt \geq R\delta \Big\}\,.
\end{align*}
By definition of Lebesgue square, we immediately get that
\begin{align}\label{smallAB}
|A| \leq \frac {6r}R\,, && |B| \leq \frac {6r}R\,.
\end{align}
Let us now arbitrarily fix a point $z=(\bar x,\bar y)\in \S(r)$ with $\bar x\notin A,\, \bar y\notin B$, and let us define $\psi:\S(c,3r)\to\R^2$ as $\psi(w)= f(z) + M(w-z)$: it is clear that $\psi$ is an affine map with $D\psi\equiv M$, and calling $g=f-\psi$ we have by definition $g(z)=0$. We claim that
\begin{equation}\label{stim1}
|g(w)| \leq 12 r R \delta \qquad \forall\, w=(x,y)\in \S(3r)\setminus (A\times B)\,.
\end{equation}
Indeed, assume that $x\notin A$ (if $y\notin B$ the obvious modification of the argument works). Recalling that $g(\bar x,\bar y)=0$ and that $x\notin A$ and $\bar y\notin B$, we can evaluate
\[\begin{split}
|g(w)|&= |g(x,y)-g(\bar x,\bar y)| \leq |g(x,y)-g(x,\bar y)| + |g(x,\bar y)-g(\bar x,\bar y)|\\
&\leq \int_{\bar y}^y |Df(x,t)-M|\, dt + \int_{\bar x}^x |Df(t,\bar y)-M|\, dt \leq 12 rR\delta\,,
\end{split}\]
and~(\ref{stim1}) is proved.\par
Let now $w=(x,y)\in \S(2r)$ be a generic point. By~(\ref{smallAB}), we can find $x_1<x<x_2$ and $y_1<y<y_2$ such that for $i=1,\,2$ we have $(x_i,y_i)\in \S(3r)$, $x_i\notin A$, $y_i\notin B$ and $x_2-x_1\leq 6r/R,\, y_2-y_1\leq 6r/R$. Hence, $w$ is inside the small rectangle $\RR$ having sides with coordinate $x_i$ and $y_i$, and by~(\ref{stim1}) we know that
\begin{equation}\label{boundR}
|g(P)|\leq 12 rR\delta \qquad \forall\, P\in\partial \RR\,.
\end{equation}
By definition $\psi(\partial \RR)$ is a small parallelogram around $\psi(w)$, with
\[
\psi(\partial \RR)\subseteq \B\bigg(\psi(w),\frac{6r\sqrt 2}R\,\|M\|\bigg)\,,
\]
having defined $\|M\|=\max |M(v)|/|v|$. This estimate, together with~(\ref{boundR}), ensures that the whole curve $f(\partial\RR)$ is done by points with distance less than $12rR\delta + 6r\sqrt 2 \|M\|/R$ from $\psi(w)$. Since $f$ is a homeomorphism, the point $f(w)$ is inside this curve, hence we finally deduce
\begin{equation}\label{firstinfty}
|f(w)-\psi(w)|=|g(w)| \leq 12rR\delta + \frac{6r\sqrt 2}R\, \|M\| < r  \,\frac \eps {10}\,,
\end{equation}
where the last inequality holds as soon as $R$ has been chosen big enough, depending on $M$ and on $\eps$, and then $\bar\delta$ has been chosen small enough, depending on $R$ and $\eps$, and thus ultimately only on $M$ and $\eps$. Hence, we have obtained the third estimate in~(\ref{lemmatoprove}).\par
Let us now consider the function $\varphi=\varphi_{\S(r)}$, and let us call
\begin{align*}
V_1 = (c_1-r,c_2+r)\,, && V_2 = (c_1+r,c_2+r) \,, && V_3 = (c_1+r,c_2-r)\,, && V_4 = (c_1-r,c_2-r)
\end{align*}
the four vertices of the square $\S(r)$. By definition, $\varphi=f$ at every vertex $V_i$: then, in the triangle $T_1$ (and the same holds true in the triangle $T_2$) we have that the two affine functions $\varphi$ and $\psi$ satisfy $|\varphi-\psi|=|f-\psi|$ at every vertex, thus~(\ref{firstinfty}) gives 
\begin{equation}\label{estivarphipsi}
\| \varphi-\psi\|_{L^\infty(\S(r))} \leq \| f-\psi\|_{L^\infty(\S(r))}\leq
\| f-\psi\|_{L^\infty(\S(2r))}\leq
12rR\delta + \frac{6r\sqrt 2}R\, \|M\|< r\,\frac \eps {10}\,,
\end{equation}
and this, together with~(\ref{firstinfty}), implies the first estimate in~(\ref{lemmatoprove}).\par

Concerning the second one, let us call $M_1$ the constant value of $D\varphi$ in $T_1$, and notice that by~(\ref{estivarphipsi}) we get
\[
\big|(M_1 - M)  ({\rm e}_1)\big| = \bigg| \frac{\big(\varphi(V_2) - \varphi(V_1)\big) - \big(\psi(V_2) - \psi(V_1)\big)}{2r} \bigg| < \frac \eps {10}\,,
\]
and the same estimate holds true for $\big|(M_1 - M)  ({\rm e}_2)\big|$ simply by checking the vertices $V_1$ and $V_4$. As a consequence, we have that the constant value of $|D\varphi- D\psi|$ in $T_1$ is less than $\eps/5$, and the same holds true of course also in $T_2$. In other words, $\|D\varphi- D\psi\|_{L^\infty(\S(r))}\leq \eps/5$. Since $D\psi$ is constantly equal to $M$ in the square $\S(r)$, by the definition of Lebesgue square we get
\[\begin{split}
\|Df - D\varphi\|_{L^1(\S(r))} &= \int_{\S(r)} |Df(z)-D\varphi(z)|\, dz
\leq \int_{\S(r)} |Df-M| + \int_{\S(r)} |M-D\varphi|\\
&\leq 36 r^2 \delta + \frac 45\, r^2\eps <r^2\eps\,,
\end{split}\]
where the last inequality is true up to possibly decrease the value of $\bar\delta$: this gives the second estimate in~(\ref{lemmatoprove}).\par
Let us now suppose that $\det M>0$. As a consequence, the image of $\S(r)$ under $\psi$ is a non-degenerate parallelogram, and then the image of $\S(r)$ under $\varphi$ is the disjoint union of two non-degenerate triangles, as soon as $\|\psi -\varphi\|_{L^\infty(\S(r))} / r$ is small enough, depending on $M$; moreover, also the validity of~(\ref{stp}) is obvious as soon as $\|\psi -f\|_{L^\infty(\S(2r))} / r$ is small enough, depending on $M$ and on $\eps$. Since~(\ref{firstinfty}) is valid for every $w\in\S(2r)$, we get~(\ref{stp}) up to further increase the value of $R$ and decrease the value of $\bar\delta$, again depending only on $M$ and on $\eps$; thus, the proof is concluded.
\end{proof}

\begin{remark}
Notice that, in the last estimate of the above proof, the final values of $1/R$ and of $\delta$ behave more or less as $\eps$ multiplied by $\min |M(v)|/|v|$, and the latter number is strictly positive exactly when $\det M>0$. This clarifies the need of the assumption $\det M>0$ in order to get~(\ref{stp}). We can come to the same conclusion also directly by considering the claim of~(\ref{stp}): we cannot hope it to be valid for the case when $\det M=0$; indeed, it is true that $f$ is as close as we wish to an affine function, but this affine function is degenerate, hence the image of a small square around $c$ is close to a degenerate parallelogram, which is a segment (or even a point if $M=0$). And of course, knowing that the four vertices of a small square are sent very close to the vertices of a parallelogram does not even imply that the piecewise affine function $\varphi_{\S(x,r)}$ is injective, if this parallelogram is degenerate!
\end{remark}

\begin{remark}\label{genlemmarem}
A quick look at the proof of the above lemma ensures that the constant $\bar\delta(M,\eps)$ actually depends only on $\eps$, $\|M\|$, and $\det M$: more precisely, $\bar\delta(M,\eps)$ is the minimum between a constant which continuously depend on $\eps$, $\|M\|$, and $\det M$ for $\det M\geq 0$ (found in the first part of the proof), and another constant which also depends continuously on $\eps$, $\|M\|$, and $\det M$, but in the range $\det M>0$ (found at the end of the proof): this second constant explodes when $\det M\to 0$. As a consequence, $\bar\delta$ is bounded if $\eps$ is bounded from below, $\|M\|$ from above, and if $\det M$ is either $0$ or it is bounded both from above and below (with a strictly positive constant).
\end{remark}

The crucial importance of the above general lemma comes from the fact that we can always apply it for small squares ``almost centered'' at Lebesgue points $\bar x$.
\begin{lemma}\label{4.6}
Let $\delta>0$ be given and let $\bar x$ be a Lebesgue point for $Df$. Then, there exists $\bar r=\bar r (\bar x,\delta)$ such that, for any $r<\bar r$ and any $x\in\S(\bar x,r/2)$, the square $\S(x,r)$ is a Lebesgue square with matrix $M=Df(\bar x)$ and constant $\delta$.
\end{lemma}
\begin{proof}
Let $\bar x\in\Omega$, $M\in \R^{2\times 2}$ and $\delta>0$ be as in the claim. Since $\bar x$ is a Lebesgue point, there exists $\bar r=\bar r(\bar x,\delta)$ such that, for any $r<\bar r$, one has $\B(\bar x,5r)\subseteq \Omega$ and
\begin{equation}\label{thisest}
\intmed_{\B(\bar x, 5r)} \big|Df(z) - M \big| \, dz \leq \frac \delta 3\,.
\end{equation}
Let now $x\in\S(\bar x,r/2)$. We have $\S(x,3r)\subseteq\B(\bar x,5r) \subseteq \Omega$, and moreover~(\ref{thisest}) gives
\[\begin{split}
\intmed_{\S(x,3r)} |Df(z)-M| \, dz &=
\frac{1}{36r^2}\int_{\S(x,3r)} |Df-M|
\leq \frac{1}{36r^2}\int_{\B(\bar x,5r)} |Df-M|\\
&= \frac{25 \pi}{36} \intmed_{\B(\bar x,5r)} |Df-M| \leq \delta\,,
\end{split}\]
and by Definition~\ref{Lebsquare} this means that $\S(x,r)$ is a Lebesgue square with matrix $M$ and constant $\delta$. Thus, the proof is concluded.
\end{proof}

\subsection{How to ``move the vertices'' of a grid}

In this section we describe how to ``move the vertices'' of a grid in order to be able to control the average of $|Df|$ inside a square with the average of $|Df|$ on the boundary of the same square. To do so, we first introduce the following notation.
\begin{definition}
We say that the domain $\Omega$ is an \emph{$r$-set} if it is the finite union of essentially disjoint squares, all having side $2r$ and sides parallel to the coordinate axes. Any side of one of these squares will be called \emph{side of type A} if both the endpoints are in the interior of $\Omega$, \emph{side of type B} if at least one endpoint is in $\partial\Omega$, but the interior of the side is inside $\Omega$, and \emph{side of type C} if the whole side is in $\partial \Omega$. Any vertex of one of the squares will be then called \emph{vertex of type A} if it belongs to the interior of $\Omega$, \emph{vertex of type B} if it belongs to $\partial\Omega$ but it is endpoint of at least one side of type B, and \emph{vertex of type C} otherwise.
\end{definition}
For any small constant $\eps>0$, we will define a short segment or curve around each vertex of type A and B. We start with the vertices of type A.
\begin{definition}
Let $\Omega$ be an $r$-set, $\eps\ll 1$ and let $V=(V_1,V_2)$ be a vertex of type A. We call $I_\eps(V)$ the segment of length $2\sqrt 2 \eps r$ centered at $V$ and with direction $\pi/4$, namely
\begin{equation}\label{defI_eps}
I_\eps(V) = \big\{(V_1+t,V_2+t):\, |t| \leq \eps r \big\}\,.
\end{equation}
Notice that all the segments $I_\eps(V)$ lie inside $\Omega$ and they do not intersect with each other.
\end{definition}
The main result that we prove ensures that the average of $|Df|$ in a side of a square can be always estimated with the average of $|Df|$ in the whole square, up to move the two vertices in the corresponding segments $I_\eps$. Actually, in order to be able also to treat Lebesgue squares with $\det M=0$, we will in fact estimate $|Df-M|$ instead of $|Df|$ for some matrix $M$: then, we will apply this result with $M=0$ for all the non-Lebesgue squares, and with $M=Df(\bar x)$ for Lebesgue squares ``almost centered'' at a Lebesgue point $\bar x$. Unfortunately (but the reason is quite evident) the estimate must explode as $1/\eps$; however, this will not be a problem for our construction.

\begin{lemma}\label{movegrid}
Let $\Omega$ be an $r$-set, let $AB$ be a side of type A, and let $M\in\R^{2\times 2}$ be a matrix. Calling $\RR\subseteq \Omega$ the union of the $6$ squares of the grid having either $A$ or $B$ (or both) as one vertex and
\[
\Gamma(A,B,M)=\Big\{(x,y)\in I_\eps(A)\times I_\eps(B):\, \int_{xy} |Df-M|\, d\H^1>\frac{25}{\eps r} \int_\RR |Df-M|\, d\H^2\Big\}\,,
\]
one has
\begin{equation}\label{moveclaim}
\H^1\bigg(\bigg\{x \in I_\eps(A):\, \H^1\Big(\Big\{y\in I_\eps(B):\, (x,y)\in \Gamma\big\}\Big)> \frac {\H^1(I_\eps(B))}5 \bigg\}\bigg) < \frac{\H^1(I_\eps(A))}5\,.
\end{equation}
\end{lemma}
\begin{proof}
Let us suppose, just to fix the ideas, that the side $AB$ is horizontal, as in Figure~\ref{Fig:move}; suppose also, for the moment, that $M=0$. Let then $\RR_0\subseteq\RR$ be the small parallelogram --dark shaded in the figure, while $\RR$ is light shaded-- whose four vertices are the endpoints of the segments $I_\eps(A)$ and $I_\eps(B)$.
\begin{figure}[thbp]
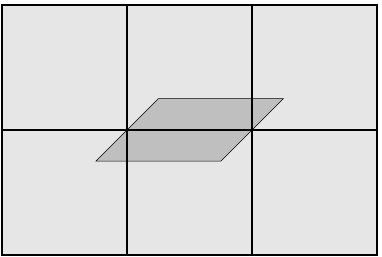
\caption{The rectangle $\RR$, the side $AB$, and the two segments $I_\eps(A)$ and $I_\eps(B)$.}\label{Fig:move}
\end{figure}
A simple change of variable argument, together with the fact that $\RR_0\subseteq \RR$, ensures that
\[
\int_{x\in I_\eps(A)}\int_{y\in I_\eps(B)} \int_{xy} |Df|\, d\H^1\, dy\,dx \leq 8 \eps r\int_{\RR_0} |Df(z)|\, d\H^2(z)
\leq 8 \eps r\int_{\RR} |Df(z)|\, d\H^2(z)\,.
\]
On the other hand, writing $\Gamma=\Gamma(A,B,0)$ for shortness, we also have
\[
\int_{x\in I_\eps(A)}\int_{y\in I_\eps(B)} \int_{xy} |Df|\, d\H^1\, dy\,dx \geq \int_{(x,y)\in\Gamma} \int_{xy} |Df|\,d\H^1\,dy\,dx
\geq \H^2(\Gamma) \, \frac{25}{\eps r} \int_\RR |Df|\,d\H^2\,,
\]
hence we deduce
\[
\H^2(\Gamma) \leq \frac 8{25}\,\eps^2 r^2 = \frac 1{25}\, \H^1(I_\eps(A))\cdot \H^1(I_\eps(B))\,,
\]
from which~(\ref{moveclaim}) immediately follows.\par
To show the general case with $M\neq 0$, it is enough to apply the above argument to the function $\tilde f(x) = f(x) - Mx$: of course the set $\Gamma(A,B,M)$ coincides with the set $\Gamma(A,B,0)$ corresponding to the function $\tilde f$, hence~(\ref{moveclaim}) follows also in the general case.
\end{proof}

The above result will be useful in order to treat the internal squares, but we have to take care also of the boundary squares. Let us then extend the definition of the segment $I_\eps(V)$ to the vertices of type B.
\begin{definition}
Let $\Omega$ be an $r$-set and let $V\in\partial\Omega$ be a vertex of type B. If $V$ is vertex of exactly two squares of the decomposition, and these two squares are adjacent (as for $V_1$ in Figure~\ref{fig:gener}, left), then we call $I_\eps(V)$ the segment of length $2\eps r$ on $\partial \Omega$ centered at $V$. If $V$ is vertex of three squares of the decomposition (as for $V_2$ in the figure), then we call $I_\eps(V)$ the union of the two segments contained in $\partial\Omega$, of length $\eps r$, having $V$ as one endpoint.
\end{definition}
\begin{figure}[thbp]
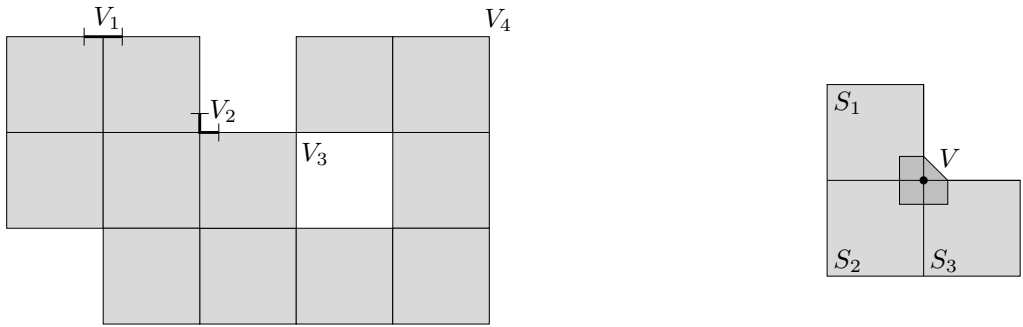
\caption{Left: some squares and vertices $V_i$ near the boundary of an $r$-set $\Omega$, and the corresponding $I_\eps(V_i)$. Right: definition of $\T_V$ and $\Psi_V$.}\label{fig:gener}
\end{figure}
Notice that we have defined $I_\eps(V)$ only for the vertices of type A and B, thus for instance not for points as $V_3$ of $V_4$ in Figure~\ref{fig:gener}. We want now to extend the validity of Lemma~\ref{movegrid} for sides of type B. To do so near points like $V_2$ in Figure~\ref{fig:gener} left, we need a last simple definition.
\begin{definition}\label{def:tilde}
Let $\Omega$ be an $r$-set and let $V\in\partial\Omega$ be a vertex of three squares of the decomposition, say $\S_1,\,\S_2,\,\S_3$. We call $\T_V$ the right triangle having right angle at $V$, two sides of length $r/2$, one horizontal and one vertical, and being not contained in $\Omega$, and for $i=1,\,2,\,3$ we call $\S_i^-$ the square contained in $\S_i$, having one vertex at $V$, and side $r/2$. Then, we let $\Psi_V$ be the obvious piecewise affine homeomorphism between $\S_1^-\cup\S_2^-\cup\S_3^-$ and $\S_1^-\cup\S_2^-\cup\S_3^-\cup\T_V$, which is bi-Lipschitz with constant $2$. Finally, we call $\Omega^+$ the union of $\Omega$ with all the triangles $\T_V$ for vertices $V$ as above, and we call $\Psi:\Omega\to\Omega^+$ the piecewise affine homeomorphism which coincides with $\Psi_V$ around every vertex $V$, and which is the identity outside (see Figure~\ref{fig:gener}, right). Notice that $\Psi$ is the identity in a $r$-neighborhood of all the internal squares of the decomposition, and it is globally $2$-biLipschitz. Finally, for every $x,\,y\in\overline\Omega$ such that the segment $\Psi(x)\Psi(y)$ is contained in $\overline{\Omega^+}$, we call $\widetilde{xy}$ the counterimage, under $\Psi$, of this segment (which is of course a piecewise linear path).
\end{definition}
We can finally generalize Lemma~\ref{movegrid} for all the sides of type B; it will be enough to limit ourselves to consider the simpler case $M=0$.

\begin{lemma}\label{movegrid2}
Let $\Omega$ be an $r$-set and let $AB\subseteq \Omega$ be a side of type B. Calling $\RR\subseteq \Omega$ the union of the squares of the grid having either $A$ or $B$ (or both) as one vertex and defining
\[
\Gamma(A,B)=\Big\{(x,y)\in I_\eps(A)\times I_\eps(B):\, \int_{\widetilde{xy}} |Df|\, d\H^1>\frac{100}{\eps r} \int_\RR |Df|\, d\H^2\Big\}\,,
\]
one has
\begin{equation}\label{moveclaim2}
\H^1\bigg(\bigg\{x \in I_\eps(A):\, \H^1\Big(\Big\{y\in I_\eps(B):\, (x,y)\in \Gamma\big\}\Big)> \frac {\H^1(I_\eps(B))}5 \bigg\}\bigg) < \frac{\H^1(I_\eps(A))}5\,.
\end{equation}
\end{lemma}
\begin{proof}
This is a very simple generalization of Lemma~\ref{movegrid}, we just have to consider a few possible cases, all depicted in Figure~\ref{fig:geneps}. Without loss of generality, we can assume that $B\in\partial\Omega$ and that the segment $AB$ is horizontal. If $B$ belongs to three squares of the decomposition, then there are three possible subcases. First of all, $A$ can be inside $\Omega$ (this is the first case depicted in the figure); second, $A$ can also belong to three squares of the decomposition (depending on how these squares are, this is the second or the third case in the figure); last, $A$ can belong to two squares, which are then necessarily the two squares having $AB$ as a side (this is the fourth case in the figure). Otherwise, $B$ can belong to exactly two adjacent squares of the decomposition, and then again $A$ can either be inside $\Omega$ (fifth case in the figure) or in the boundary of $\Omega$: in this latter case (the sixth and last one in the figure) $A$ must also belong only to the same two squares to which $B$ belong, since otherwise we fall back into an already considered case. To clarify the situation, Figure~\ref{fig:geneps} does not show the situation in $\Omega$, but directly in $\Omega^+$.\par
\begin{figure}[thbp]
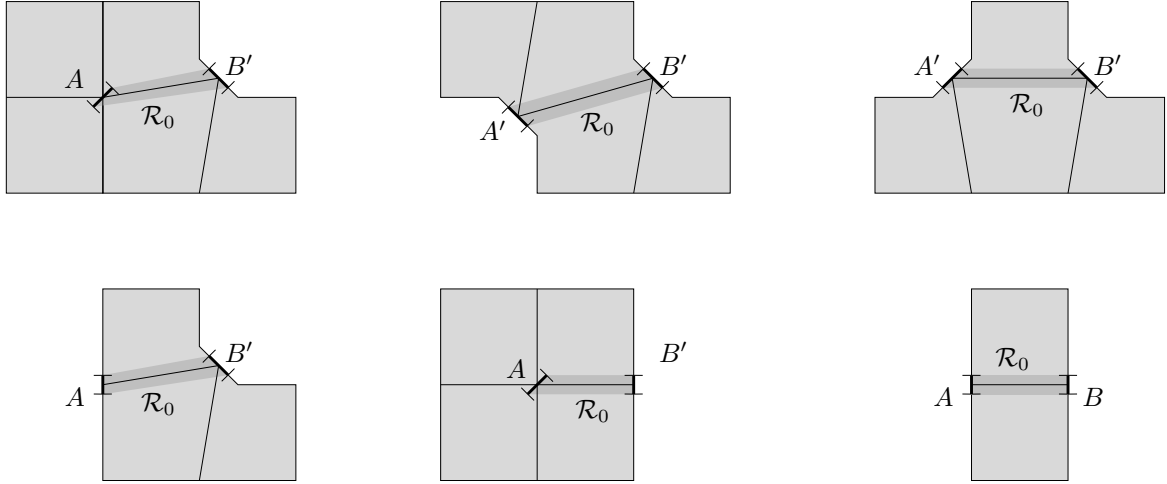
\caption{The six possibilities in Lemma~\ref{moveclaim2}.}\label{fig:geneps}
\end{figure}
The proof is now almost identical to the proof already done in Lemma~\ref{moveclaim}. Let us call again $\RR_0$ the quadrilateral having as vertices the endpoints of $\Psi(I_\eps(A))$ and $\Psi(I_\eps(B))$; this quadrilateral, depicted in the figure for all the possible cases, belongs to $\Omega^+$. Notice that, depending on the case, $\RR$ can be done by $2$, or $3$, or $4$, or $5$ squares, and the figure always shows only (the image under $\Psi$ of) these squares. To conclude the proof we only have to keep in mind that we are interested in what happens in the real domain $\Omega$, not in the simplified domain $\Omega^+$. However, we can use the map $\Psi$ to move the situation from $\Omega$ to $\Omega^+$; then, we notice that the very simple calculation done in Lemma~\ref{moveclaim} works perfectly in the new situation; and finally, we use $\Psi^{-1}$ to get back to the case of $\Omega$. The only detail which changes, since $\Psi$ is $2$-biLipschitz, is that the constant $25$ in the old definition of $\Gamma$ for internal sides becomes $100$ for the new definition of $\Gamma$ for sides touching the boundary.
\end{proof}

Let us now fix a matrix $M=M(A,B)$ for any side $AB$ of type A, and write for brevity $\Gamma(A,B)=\Gamma(A,B,M(A,B))$. Thanks to the above results, we have defined a set of ``bad pairs'' $(x,y)\in\Gamma(A,B)$, where ``bad'' means that in the segment $xy$ (or in the curve $\widetilde{xy}$) too much derivative is concentrated. We can now find a selection of points $x_V\in I_\eps(V)$ for any vertex $V$, so that for any side $AB$ the pair $(x_A,x_B)$ does not belong to $\Gamma(A,B)$.
\begin{lemma}\label{selection}
Let $\Omega$ be an $r$-set, and let us fix a matrix $M(A,B)$ for any side of type A. It is possible to select a point $x_V\in I_\eps(V)$ for any vertex $V$ of type A or B in such a way that, for every side $AB$ of type A or B, one has $(x_A,x_B)\notin \Gamma(A,B)$.
\end{lemma}
\begin{proof}
We will argue recursively. First of all, let us enumerate all the vertices of type A or B as $V_1,\, V_2,\, \dots V_N$, for some $N\in\N$. Then, we aim to show that it is possible to select by recursion points $x_i=x_{V_i}$ in every $I_\eps(V_i)$ in such a way that, whenever $V_iV_m$ is a side of type A or B, the point $x_i$ is chosen in such a way that
\begin{equation}\label{howtopick}
\left\{\begin{array}{ll}
\H^1\Big(\Big\{y\in I_\eps(V_m):\, (x_i,y)\in \Gamma(V_i,V_m\big\}\Big)\leq \bal\frac {\H^1(I_\eps(V_m))}5\eal \qquad &\hbox{if $m>i$}\,,\\
(x_i,x_m)\notin \Gamma(V_i,V_m) & \hbox{if $m<i$}\,.
\end{array}\right.
\end{equation}
Notice that, since we will define the points recursively, then the above requests make sense: in particular, if $m<i$ then the point $x_m$ has been already chosen when we have to choose $x_i$. Of course, if we can select all the points $x_i$ according to~(\ref{howtopick}), then we are done: the thesis will be simply given by the second property in~(\ref{howtopick}), but the first one is essential to let the recursion work.\par

Let then $1\leq i \leq N$, and let us suppose that the points $x_j$ for $j<i$ have been already chosen according to~(\ref{howtopick}); let $n^-$ (resp., $n^+$) be the number of the indices $j<i$ (resp., $j>i$) such that $V_iV_j$ is a side of type A or B. By~(\ref{howtopick}) applied to the indices $j<i$, we know that the points $x\in I_\eps(V_i)$ such that $(x,x_j)\in \Gamma(V_i,V_j)$ for some $j<i$ corresponding to a side $V_iV_j$ cover a portion at most $n^-/5$ of $I_\eps(V_i)$. On the other hand, by Lemma~\ref{movegrid2}, the points $x\in I_\eps(V_i)$ such that
\[
\H^1\Big(\Big\{y\in I_\eps(V_j):\, (x_i,y)\in \Gamma(V_i,V_j\big)\Big\}\Big)> \frac {\H^1(I_\eps(V_j))}5
\]
for some $j>i$ for which $V_iV_j$ is a side of type A or B cover a portion at most $n^+/5$ of $I_\eps(V_i)$. Since of course $n^-+n^+\leq 4$, we can pick a point $x_i\in I_\eps(V_i)$ for which none of the above problems occur, hence by definition this choice fulfills~(\ref{howtopick}). The recursion argument is then proved, and the proof is concluded.
\end{proof}

The last goal of this section is to define an approximating function $\tilde f$ of $f$ on the grid given by the boundaries of the squares. First of all, let us give the definition of ``grid'' and ``modified grid''.

\begin{definition}\label{defGrid}
Let $\Omega$ be an $r$-set, and for any vertex $V$ of type B let $V'$ be a given point $x_V$ in $I_\eps(V)$; instead, let $V'=V$ for any vertex $V$ of type A or C. We call \emph{grid} the union $\G$ of all the sides $AB$ of the squares of the decomposition, while the \emph{modified grid} is the union $\widetilde\G$ of all the ``modified sides'', that is, the piecewise linear curves $\widetilde{A'B'}$. Notice that, if $AB\subseteq \partial\Omega$, it might happen that the curve $\widetilde{A'B'}$ has not been defined in Definition~\ref{def:tilde}; if this is the case, we simply denote by $\widetilde{A'B'}$ the shortest curve in $\overline\Omega$ connecting $A'$ and $B'$: notice that this shortest curve lies entirely inside $\partial\Omega$, and that actually this minimizing property for $\widetilde{A'B'}$ is true also for the sides $AB\subseteq\partial\Omega$ where $\widetilde{A'B'}$ was already defined in Definition~\ref{def:tilde}. For every square $\S$ of the grid, we call $\widetilde \S$ the union of its modified sides.
\end{definition}
Observe that the grid $\widetilde\G$ coincides with the grid $\G$, except near the boundary of $\Omega$; analogously, the piecewise linear curve $\widetilde{A'B'}$ is nothing else than the segment $AB$, if it is a side of type A. Notice that both the grid and the modified grid contain the boundary of $\Omega$. We give now our last definition of a map $\tilde f$ on $\widetilde \G$.

\begin{definition}\label{defabo}
Let $\Omega$ be an $r$-set, and for any side $AB$ of type A fix a matrix $M=M(A,B)$. Let the points $x_V\in I_\eps(V)$ for the different vertices $V$ of type A or B be as in Lemma~\ref{selection}. We define the function $g:\widetilde \G\to\R^2$ as follows. For any side $AB$ of type A or B, we define $g$ on the curve $\widetilde{A'B'}$ as the reparameterization, at constant speed, of the function $f$ on the curve $\widetilde{x_Ax_B}$; moreover, we let $g=f$ on $\partial\Omega\subseteq \widetilde\G$.
\end{definition}
In the above definition, it is important not to confuse the points $x_V$ with the points $V'$: according with Definition~\ref{defGrid}, $V'=x_V$ if $V$ is a vertex of type B, while $V'=V$ if $V$ is a vertex of type A or C. In particular, if $AB$ is a side of type A, then $\widetilde{A'B'}$ is simply the segment $AB$, hence $g$ on the segment $AB$ is the reparameterized copy of $f$ on the segment $x_Ax_B=\widetilde{x_Ax_B}$. We conclude this section with an estimate for the function $g$.
\begin{lemma}\label{4.16}
Let $\Omega$ be an $r$-set, and let the matrices $M=M(A,B)$, the points $x_V$ and the function $g:\widetilde\G\to\R^2$ be as in Definition~\ref{defabo}. Then, for any side $AB$ of type A, calling $\nu$ the unit vector with direction $AB$ and $\RR$ again the union of the squares of the grid having either $A$ or $B$ as one vertex, we have
\begin{equation}\label{firstthesis}
\int_{AB} |Dg(t) - M\cdot \nu|\,dt \leq \frac {25}{\eps r}\int_\RR |Df-M|\, d\H^2+ 11 \|M\| \eps r\,.
\end{equation}
Instead, for any side $AB$ of type B, we have
\begin{equation}\label{secondthesis}
\int_{\widetilde{A'B'}} |Dg(t)|\,dt \leq \frac {100}{\eps r}\int_\RR |Df|\, d\H^2\,.
\end{equation}
\end{lemma}
\begin{proof}
Let us start with a side $AB$ of type A, and call for brevity $x=x_A$ and $y=x_B$. By Lemma~\ref{movegrid}, we already know that
\begin{equation}\label{alreadyknow}
\int_{xy} |Df(s) - M| \, ds \leq  \frac {25}{\eps r} \int_\RR |Df-M|\,d\H^2\,.
\end{equation}
Recall now that, by definition, the function $g$ on the segment $AB$ is simply the reparameterization of the function $f$ on the segment $xy$. Define then $\lambda=\overline{AB}/\overline{xy}$, and call $\tilde\nu$ the unit vector with direction $xy$: notice that by construction
\begin{align*}
1-2\eps \leq \lambda \leq 1+3\eps\,, && |\tilde\nu - \nu | \leq 2\eps\,.
\end{align*}
As a consequence, by a change of variable we obtain
\[\begin{split}
\int_{AB} |Dg(t) - M\cdot \nu | \,dt &=
\lambda \int_{xy} \bigg|\frac{Df(s)\cdot \tilde \nu}\lambda- M\cdot \nu \bigg|\, ds
=\int_{xy} \big|Df(s)\cdot \tilde \nu- M\cdot \lambda\nu \big|\, ds\\
&\leq\int_{xy} \big|(Df(s)-M)\cdot \tilde \nu\big|\, ds+\int_{xy} \|M\| \,|\lambda \nu - \tilde\nu|\, ds\\
&\leq\int_{xy} |Df(s)-M|\, ds+5\|M\| \eps\,\overline{xy}\,,
\end{split}\]
thus recalling~(\ref{alreadyknow}) we get~(\ref{firstthesis}).\par

Let now $AB$ be a side of type B, and call again for brevity $x=x_A$ and $y=y_B$. In this case, by Lemma~\ref{movegrid2} we already know that
\begin{equation}\label{hereabove}
\int_{\widetilde{xy}} |Df|\, d\H^1\leq \frac{100}{\eps r} \int_\RR |Df|\, d\H^2\,.
\end{equation}
Now, by definition of $g$ we have that the image of $\widetilde{A'B'}$ under $g$ coincides with the image of $\widetilde{xy}$ under $f$, hence the lengths of the two curves coincide, which means
\[
\int_{\widetilde{xy}} |Df|\, d\H^1 = \int_{\widetilde{A'B'}} |Dg(t)|\,dt\,.
\]
Hence, (\ref{secondthesis}) directly follows from~(\ref{hereabove}), and the proof is concluded.
\end{proof}

\subsection{How to modify $f$ on a grid}

In this section, we show how to modify a function on a one-dimensional grid; more precisely, we take a generic function defined on a grid, and we modify it in order to become piecewise linear. We have to do so because both our big results, namely Theorems~\ref{extension} and~\ref{extension2}, need a function which is piecewise linear on the boundary of a square. We start with a rather simple modification, which we will eventually apply to the ``bad'' squares and to the ``good'' squares corresponding to a matrix with $\det M=0$; this construction is reminiscent to the one made in~\cite{DP}, where the situation was more complicated because also the inverse should be approximated.

\begin{prop}\label{notgoodforgood}
Let $\Omega$ be an $r$-set, let $\G$ and $\widetilde \G$ be a grid and a modified grid according to Definition~\ref{defGrid}, and let $g:\widetilde \G\to \R^2$ be a continuous, injective function which is piecewise linear on $\partial\Omega$. Then, there exists a piecewise linear and injective function $\hat g:\widetilde\G\to\R^2$ such that $\hat g=g$ on $\partial\Omega$, $\hat g(V')=g(V')$ for every vertex $V$ of the grid, and for every side $AB$ of type A and every matrix $M$ one has
\begin{equation}\label{check1}
\int_{AB} |D\hat g(t) - M\cdot \nu|\,dt \leq  \int_{AB} |Dg(t) - M\cdot \nu|\,dt  \,,
\end{equation}
while for every side $AB$ of type B one has
\begin{equation}\label{check2}
\int_{\widetilde{A'B'}} |D\hat g(t)|\,dt \leq \int_{\widetilde{A'B'}} |Dg(t)|\,dt\,.
\end{equation}
Moreover, on each curve $\widetilde{A'B'}$ the function $\hat g$ is an interpolation of finitely many points of the curve $g(\widetilde{A'B'})$.
\end{prop}
\begin{proof}
We define the map $\hat g$ in two steps, namely, we work first around the vertices, and then in the interior of the sides. Figure~\ref{Fig:Interpol} depicts how the construction works.
\step{I}{Definition of $\hat g$ around the vertices.}
Let us start by selecting a vertex $V$ of type A, that is, $V$ belongs to the interior of $\Omega$. There are then four sides of the grid for which $V$ is an endpoint, and we can call $V_i,\, 1\leq i\leq 4$ the four other endpoints of these sides. Since $g$ is continuous and injective, the quantity
\[
\inf \Big\{ \overline{g(x)g(V)}:\, x\in \widetilde\G \setminus \bigcup_{i=1}^4 VV_i\Big\}
\]
is strictly positive: let us then select a small radius $\rho=\rho(V)>0$, much smaller than this quantity. Hence, by definition, the ball $\B(g(V),\rho)$ intersects the image of the grid $\widetilde \G$ under $g$ only in points of the form $g(x)$ for $x$ belonging to one of the four sides $VV_i$. On the other hand, for each $i=1,\,\dots\,,\, 4$ there is at least a point $x\in VV_i$ such that $g(x)$ belongs to $\partial \B(g(V),\rho)$. Let us then call $V^+_i$ the last point $x$ of the segment $VV_i$ for which $g(x)\in\partial \B(g(V),\rho)$, where ``last'' means ``the farthest from $V$''. We define then the function $\hat g$, on each of the four segments $VV^+_i$, simply as the linear function connecting $g(V)$ with $g(V^+_i)$.\par

Let us now consider a vertex $V$ of type B, that is, $V$ belongs to $\partial \Omega$ but there is some side of the grid, contained in the interior of $\Omega$, of which $V$ is an endpoint: let us call $j$ the number of such sides, and notice that by construction $j$ is either one or two (keep in mind Figure~\ref{fig:geneps}). We argue then similarly as before: we call $V_i$, for $1\leq i\leq j$ the other endpoints of these internal sides, and we consider the strictly positive quantity
\[
\inf \Big\{ \overline{g(x)g(V)}:\, x\in \widetilde\G \setminus \Big(\bigcup\nolimits_{i=1}^j \widetilde{V'V'_i}\cup\partial\Omega\Big)\Big\}\,.
\]
This time, we will define $\rho=\rho(V)$ not only much smaller than the above quantity, but also so small that $g(\partial\Omega)\cap \B(g(V'),\rho)$ is the union of two segments (this is surely true as soon as $\rho$ is small enough, since $g$ is piecewise linear on $\partial\Omega$). Exactly as before, for $1\leq i\leq j$ we call $V_i^+$ the last point $x\in \widetilde{V'V_i'}$ such that $g(x)\in \partial\B(g(V'),\rho)$: up to further decrease $\rho(V)$, we can also assume that the portion of the piecewise linear curve $\widetilde{V'V'_i}$ connecting $V'$ and $V_i^+$ is simply a segment. Then, as before we define $\hat g$ on each of the $j$ segments $V'V_i^+$ as the linear function connecting $g(V')$ with $g(V_i^+)$.
\begin{figure}[thbp]
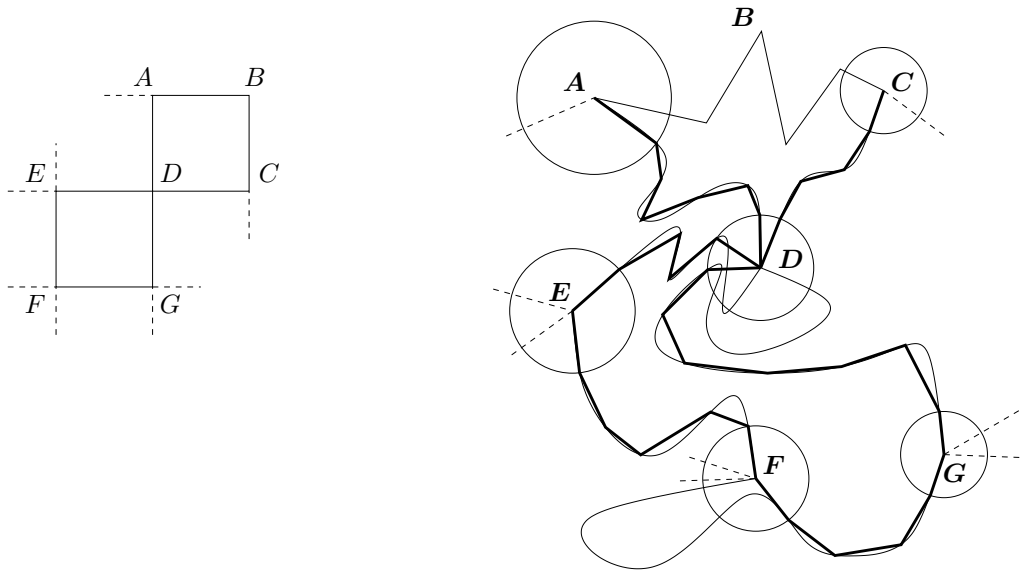
\caption{The construction in Proposition~\ref{notgoodforgood}: the points $A,\,B$ and $C$ are in $\partial\Omega$, while $D,\,E,\,F$ and $G$ are inside $\Omega$; the image of $\hat g$ inside $\Omega$ is thicker.}\label{Fig:Interpol}
\end{figure}
\step{II}{Definition of $\hat g$ inside the sides.}
Up to now, we have defined $\hat g$ on a neighborhood of each vertex of type A and B; moreover, $\hat g$ is already automatically defined on $\partial\Omega$, since it must be $\hat g=g$ on $\partial\Omega$. Therefore, to conclude we have to define $\hat g$ on the remaining portion of $\widetilde \G$. By construction, this remaining portion is a finite and disjoint union of internal parts of sides of type A or B: more precisely, for every side $AB$ of type A there is a segment $A^+B^-\subset\subset AB$ where $\hat g$ has to be defined, while for every side $AB$ of type B $\hat g$ has still to be defined in some piecewise linear curve $\widetilde{A^+B^-}\subset\subset\widetilde{A'B'}$.\par
Let us first consider the case of a side $AB$ of type A: the function $\hat g$ has been already defined in the segment $AA^+$ (resp., $B^-B$) as the linear function connecting $g(A)$ and $g(A^+)$ (resp., $g(B^-)$ and $g(B)$), and moreover the points $\hat g(A^+)=g(A^+)$ and $\hat g(B^-)=g(B^-)$ are in the boundary of the disks $\B(g(A),r(A))$ and $\B(g(B),r(B))$ respectively. We have to define $\hat g$ in the segment $A^+B^-$, and this must be a piecewise linear curve connecting $g(A^+)$ and $g(B^-)$. Observe that $g$, in the segment $A^+B^-$, is already a curve connecting $g(A^+)$ and $g(B^-)$, the only problem being that it is not necessarily piecewise linear. We can then select many points $P_0=A^+,\, P_1,\, P_2\,, \dots\,,\, P_N=B^-$ in the segment $A^+B^-$, and define $\hat g$ in $A^+B^-$ as the piecewise affine interpolation of these values (that is, $\hat g(P_i)=g(P_i)$ and $\hat g$ is linear on each $P_iP_{i+1}$). A very simple geometric argument (similar to~\cite[Lemma~5.5]{DP}, but much easier) shows that, by choosing carefully many points, the map $\hat g$ in $A^+B^-$ is injective, it never crosses the two disks $\B(g(A),\rho(A))$ and $\B(g(B),\rho(B))$, and its $L^\infty$ distance with $g$ is much smaller than
\[
\inf \Big\{ \overline{g(x)g(y)}:\, x\in A^+B^-\,,\  y\in \widetilde \G\setminus AB\Big\}\,.
\]
Let us now consider a side $AB$ of type B: in this case, $\hat g$ is piecewise linear in the two segments $A'A^+$ and $B^-B'$, and we have to extend the definition to the piecewise linear curve $\widetilde{A^+B^-}$. This can be done exactly as we just did for a side of type A, the only difference is that, by doing the interpolation, the points $P_i$ must include all the extremes of the segments forming the curve $\widetilde{A^+B^-}$. Apart from that, nothing else changes, thus the definition of $\hat g$ is finally concluded.
\step{III}{The properties of $\hat g$.}
To conclude the proof, we just need to check that $\hat g$ fulfills all the required properties. The fact that $\hat g=g$ on $\partial \Omega$ and at every vertex is true by construction, as well as the fact that $\hat g$ is an interpolation of finitely many points of the curve $g(\widetilde{A'B'})$ on every side $\widetilde{A'B'}$ of $\widetilde\G$. To check~(\ref{check1}) and~(\ref{check2}), we just have to keep in mind that $\int_{AB} |Dg|$ is the length of the curve $g$ on the segment $AB$, while $\int_{AB} |D\hat g|$ is the length of the (piecewise affine) curve $\hat g$ on $AB$. But any interpolation of points of a curve is shorter than the curve itself, thus~(\ref{check1}) follows immediately for the case $M=0$, and the very same argument with $\widetilde{A'B'}$ in place of $AB$ ensures also~(\ref{check2}).\par
To show~(\ref{check1}) when $M\neq 0$, let $CD\subseteq AB$ be a segment where $\hat g$ is linear and satisfies $\hat g(C)=g(C)$, $\hat g(D)=g(D)$: then
\[\begin{split}
\int_C^D |D\hat g(t)-M\cdot \nu| \,dt&=
\bigg|\int_C^D D\hat g(t)-M\cdot \nu \,dt\bigg|
=\,\bigg| \int_C^D D g(t) - M\cdot \nu\,dt \bigg|\\
&\leq \int_C^D |D g(t) -M\cdot \nu|\,dt \,,
\end{split}\]
thus adding over the segments where $\hat g$ is linear we get~(\ref{check1}) and the proof is finished.
\end{proof}

We conclude this section with the following generalization of the above result, rather technical but very useful to obtain Theorem~\ref{main}, and whose proof is actually nothing else than a straightforward modification of the previous one.

\begin{prop}\label{goodforstarting}
Let $\C=\cup_{i=1}^N A_iB_i$ be a finite union of closed segments in $\R^2$, and let $\C_0=\cup_{i=1}^{N_0} A_iB_i$, with $N_0\leq N$, be a selection of some of them. Let $g:\C\to\R^2$ be a continuous one-to-one function, and let $\eta>0$ be given. Then, there exists another continuous and one-to-one function $\hat g:\C\to\R^2$ such that $\hat g=g$ at every endpoint of each segment, $\{\hat g\neq g\}$ is contained in the $\eta$-neighborhood of $\C_0$, $\hat g$ is piecewise linear on $\C_0$, and for every side of $\C$ the estimates~(\ref{check1}) and~(\ref{check2}) hold.
\end{prop}
\begin{proof}
The proof can be done almost exactly as in Proposition~\ref{notgoodforgood}. First of all, up to a subdivision of some of the segments, we can assume that every two segments are either disjoint, or they meet at a common endpoint. Then, we start defining $\hat g=g$ on all the sides both whose endpoints are not contained in $\C_0$. Further, we consider any of the remaining vertices, say $V$. We can select a small $\rho$ such that the ball $\B(g(V),\rho)$ contains only points of the form $g(x)$ for $x$ belonging to one of the (finitely many) segments $VV_i$ having $V$ as an endpoint; up to decrease $\rho$, we can also ensure that $|x-V|<\eta$ for any such $x$. We define then the points $V_i^+$ exactly as in Step~I of Proposition~\ref{notgoodforgood}, and we let $\hat g$ be linear on each segment $VV_i^+$: the continuity and injectivity of $\hat g$ up to now is then clear. In the portions of the segments where $\hat g$ has not yet been defined, we can then define it in two different ways: inside the segments which belong to $\C_0$, we define a piecewise linear $\hat g$ exactly as in Step~II of Proposition~\ref{notgoodforgood}; inside the other segments, we simply let $\hat g=g$.\par
By construction and arguing as in Step~III of Proposition~\ref{notgoodforgood}, we can then immediately observe that the function $\hat g$ is as required.
\end{proof}

\subsection{The proof of Theorem~\ref{main}}

We are finally in position to give the proof of Theorem~\ref{main}, which will come by putting together all the results that we got up to now. For the reader's convenience, we split the proof in some parts. The first one is a very peculiar case, namely, when $\Omega$ is an $r$-set and the function $f$ is already piecewise linear on the boundary; nevertheless, most of the difficulties are contained in this part.

\begin{prop}\label{allhere}
Under the assumption of Theorem~\ref{main}, assume that in addition $\Omega$ is an $r$-set, and that $f$ is continuous up to $\partial\Omega$ and piecewise linear there. Then, for every $\eps>0$ there exists a finitely piecewise affine homeomorphism $f_\eps:\Omega\to\R^2$ such that
\begin{align}\label{thesis1}
\|f_\eps-f\|_{W^{1,1}}+\|f_\eps-f\|_{L^\infty} < \eps\,, && f_\eps=f  \hbox{ on } \partial\Omega\,.
\end{align}
\end{prop}
\begin{proof}
The idea of the construction is rather simple: we divide the squares in four groups, namely, the Lebesgue squares with positive determinant, the Lebesgue squares with $M\neq 0$ but zero determinant, the Lebesgue squares with $M=0$, and the other ones. Inside the first squares we can substitute $f$ with $\varphi_{\S(c,r)}$ and rely on Lemma~\ref{genlemma}, for the second ones we will use Theorem~\ref{extension2}, and for the third and fourth ones Theorem~\ref{extension}. However, to treat each square separately, we need to take care of the values on the boundaries of the squares: on one hand, they must be piecewise linear, in order to allow us to apply Theorems~\ref{extension2} and~\ref{extension}, and this will be obtained thanks to Proposition~\ref{notgoodforgood}; but on the other hand, any two adjacent squares must have the same boundary values on the common side, and this will require same care. Let us then start with the proof, dividing it in several steps.
\step{I}{Definition of the constants $\eps_i$ and of the sets $\A_1,\,\A_2$ and $\A_3$.}
First of all, we have to take five small constant $\eps_i$ for $1\leq i \leq 5$. More precisely, $\eps_1$ is a small geometric constant (for instance, $\eps_1=1/10$ is enough); instead, the constants $\eps_5\ll \eps_4\ll \eps_3\ll \eps_2 \ll \eps$ will depend on the data, that is, $\Omega,\, f$ and $\eps$. More precisely, since $f\in W^{1,1}(\Omega)$, we can select $\eps_2\ll 1$ so small that
\begin{equation}\label{defeps2}
\int_A |Df| \leq \frac{\eps_1 \eps}{54K} \qquad \forall\, A\subseteq\Omega:\, |A| \leq \eps_2\,,
\end{equation}
where $K$ is a purely geometric constant, which we will make explicit during the proof. Then, we define $\eps_3\ll \eps_2$ in such a way that
\begin{equation}\label{firstbad}
\bigg|\bigg\{ x\in\Omega:\, 0<|Df(x)|< \eps_3 \hbox{ or } |Df(x)|> \frac 1{\eps_3} \hbox{ or } 0<\det(Df(x)) < \eps_3 \bigg\}\bigg|< \frac {\eps_2}{45}\,.
\end{equation}
This estimate is surely true as soon as $\eps_3$ is small enough, depending on $\eps_2$, on $\Omega$, and on $f$. Now, we define the following two subsets of the matrices $M\in\R^{2\times 2}$,
\begin{align*}
\M^+ = \Big\{\eps_3 < \|M\| < \frac 1{\eps_3}\,,\, \det M > \eps_3 \Big\}\,, &&
\M^0 = \Big\{\eps_3 < \|M\| < \frac 1{\eps_3}\,,\, \det M =0\Big\}\,,
\end{align*}
which of course depend only on $\eps_3$. Finally, we let $\eps_5\ll \eps_4\ll \eps_3$ be so that
\begin{align}\label{here345}
\frac{\eps_5}{\eps_4}+\frac{\eps_4}{\eps_3}+\frac{\eps_5}{\eps_1} \leq \frac \eps {6K|\Omega|}\,, &&
\eps_5\ll \eps_4\eps_3\,, && \eps_4\ll \eps_3^2\,,
\end{align}
and we set $\hat\delta=\hat\delta(\eps_3,\eps_5)$ as
\begin{equation}\label{defhatdelta}
\hat\delta = \min \big\{ \bar\delta(M,\eps_5),\, M\in \M^+\cup\M^0\big\}\,,
\end{equation}
where $\bar\delta$ are the constants of Lemma~\ref{genlemma}. Observe that $\hat\delta$ really depends only on $\eps_3$ and $\eps_5$ by construction, as observed in Remark~\ref{genlemmarem}. The last constant to select is $r$: indeed, $\Omega$ is an $r$-set, but then we can regard it as an $r/H$-set for every $H\in\N$; as a consequence, we can now change the value of $r$, making it as small as we need: in particular, we let $r$ be so small that
\begin{align}\label{secondbad}
r P(\Omega) + \Big| \big\{x\in \Omega:\,  \bar r(x,\hat\delta) \leq r\big\}\big| < \frac {\eps_2}{180}\,, &&
r P\big(f(\Omega)\big) \leq \frac \eps{66 K}\,,
\end{align}
where the constants $\bar r(x,\hat\delta)$ have been defined in Lemma~\ref{4.6} for every $x\in\Omega$ which is a Lebesgue points for $Df$ (so, for almost every point of $\Omega$), and where $P(A)$ is as usual the perimeter of $A$, that is, $\H^1(\partial A)$.\par
Having fixed all the constants $\eps_i$, and having also chosen the final value of $r$, we can now enumerate the squares of the grid as $\S_i,\, 1\leq i \leq N$, and we subdivide these squares in four groups, namely,
\begin{align*}
\A_1 &= \Big\{\S_i \subset\subset \Omega:\, \hbox{$\S_i$ is a Lebesgue square with matrix $M_i\in \M^+$ and constant $\hat\delta$}\Big\}\,,\\
\A_2 &= \Big\{\S_i \subset\subset \Omega:\, \hbox{$\S_i$ is a Lebesgue square with matrix $M_i\in \M^0$ and constant $\hat\delta$}\Big\}\,,\\
\A_3 &= \Big\{\S_i \subset\subset \Omega:\, \hbox{$\S_i$ is a Lebesgue square with matrix $M_i=0$ and constant $\hat\delta$}\Big\}\,,\\
\A_4 &= \Big\{\S_i:\, \S_i \notin \A_1\cup\A_2\cup\A_3\Big\}\,.
\end{align*}
We aim now to show that most of the squares belong to the first three groups. In fact, let us consider the total area of the squares containd in $\A_4$. The union of those which touch the boundary of $\Omega$ has of course an area smaller than $r P(\Omega)$. Let us instead consider a square $\S(c,r)\in\A_4$ compactly contained in $\Omega$: this means that, for every point $x\in \S(c,r/2)$, either we cannot apply Lemma~\ref{4.6} with constant $\hat\delta$ to $x$ (thus, $\bar r(x,\hat \delta)\leq r$), or $x$ belongs to the set in~(\ref{firstbad}). Since the area of $\S(c,r)$ is four times greater than the area of $\S(c,r/2)$, by~(\ref{firstbad}) and~(\ref{secondbad}) we deduce
\begin{equation}\label{smalleps2}
\big|\bigcup \{\S_i \in\A_4\}\big| 
\leq r P(\Omega) +4\bigg(\Big|\big\{x\in\Omega:\, \bar r(x,\hat \delta)\leq r\big\} \Big|+\frac {\eps_2}{45}\bigg)
< \frac{\eps_2} 9\,.
\end{equation}

\step{II}{Squares in $\A_1$ and $\A_2$ never meet.}
Let us now show that squares in $\A_1$ and $\A_2$ never meet, that is, no vertex of a square in $\A_1$ can be vertex also of a square in $\A_2$: this will come as a simple consequence of the $L^\infty$ estimate~(\ref{lemmatoprove}) in Lemma~\ref{genlemma}. Indeed, assume for simplicity of notations that $\S_1=\S(c_1,r)$ and $\S_2=\S(c_2,r)$ have some common vertex, and that $\S_1\in \A_1,\, \S_2\in\A_2$. Then, Lemma~\ref{genlemma} provides us with two affine functions $\psi_1,\,\psi_2$ satisfying $D\psi_1=M_1,\, D\psi_2=M_2$, each $\S_i$ being  a Lebesgue square with matrix $M_i$. Let us now call
\[
\S_{3/2}:=\S\bigg(\frac{c_1+c_2}2,r\bigg)\subseteq \S(c_1,2r)\cap \S(c_2,2r)
\]
so that by~(\ref{lemmatoprove}) and recalling~(\ref{defhatdelta}), we have
\[
\|\psi_1-\psi_2\|_{L^\infty(\S_{3/2})} \leq \|\psi_1-f\|_{L^\infty(\S(c_1,2r))}+\|\psi_2-f\|_{L^\infty(\S(c_2,2r))} < \frac{r\eps_5}5\,.
\]
Since $\det M_2=0$, by construction we can find two points $x,\,y \in \S_{3/2}$ such that $|y-x|=r$ and $\psi_2(y-x)=0$. The last inequality, keeping in mind the definition of $\A_1$, yields then
\[
\frac{r\eps_5}5 > |\psi_1(y-x)| \geq \frac{\det M_1}{\|M_1\|}\, |y-x| > \eps_3^2 r\,.
\]
Since this is in contradiction with~(\ref{here345}), we have concluded the proof of this step. For later use we underline that, more in general, we have proved what follows:
\begin{equation}\label{adjasimil}
\forall\, \S_a,\,\S_b \in \A_1\cup\A_2 \hbox{ adjacent, one has } \|M_a-M_b\| <\frac {\eps_5}5\,.
\end{equation}

\step{III}{A temptative modified grid $\widetilde \G_1$ and of a temptative function $g_1:\widetilde\G_1\to \R^2$.}
In this step, we define a modified grid $\widetilde \G_1$ and a function $g_1$ on it. To do so, we simply have to choose a matrix $M(A,B)$ for every side $AB$ of type A of the grid $\G$; once done this, we get first the points $x_V\in I_{\eps_4}(V)$ from Lemma~\ref{selection} applied with $\eps_4$ in place of $\eps$, then the modified grid $\widetilde \G_1$ from Definition~\ref{defGrid}, and finally the function $g_1:\widetilde\G_1\to\R^2$ from Definition~\ref{defabo}.\par
The matrices $M(A,B)$ will be defined as follows. For any side $AB$ of type A, let for a moment $\S_a$ and $\S_b$ be the two squares of the grid having $AB$ as a side: then, if neither $\S_a$ nor $\S_b$ belong to $\A_2$, we let $M(A,B)=0$; if $\S_a\in\A_2$ but $\S_b\notin \A_2$, then we let $M(A,B)=M_a$, and analogously if $\S_a\notin \A_2$ and $\S_b\in \A_2$ we let $M(A,B)=M_b$; finally, if both $\S_a$ and $\S_b$ belong to $\A_2$, then we set $M(A,B)$ to be arbitrarily one between $M_a$ and $M_b$: it makes no difference which one we choose, since by~(\ref{adjasimil}) we know that $M_a\approx M_b$.\par

We want now to evaluate $\int_{\partial\widetilde\S} |Dg_1|$ for some of the modified squares $\widetilde\S$ (recall Definition~\ref{defGrid}). Let us start by taking a square $\S=\S(c,r)\in\A_3\cup\A_4$, and let us call $\S^+=\S(c,3r)\cap\Omega$ the union of the nine squares around it (more precisely, of those which are inside $\Omega$). Take any side $AB\subseteq \partial\S$ of type A or B, and observe that the union $\RR_{AB}$ of the squares touching $A$ or $B$ is contained in $\S^+$. If $AB$ is of type A but $M(A,B)=0$, or if $AB$ is of type B, we can apply Lemma~\ref{4.16} (using~(\ref{firstthesis}) or~(\ref{secondthesis}) if $AB$ is of type A or B respectively) and get
\[
\int_{\widetilde{A'B'}} |Dg_1|\,d\H^1\leq \frac{100}{\eps_4 r} \int_{\RR_{AB}} |Df|\, d\H^2
\leq \frac{100}{\eps_4 r} \int_{\S^+} |Df|\, d\H^2\,.
\]
Instead, if $AB\subseteq \partial\S\cap\partial\Omega$, then of course
\[
\int_{\widetilde{A'B'}} |Dg_1|\,d\H^1=\int_{\widetilde{x_Ax_B}} |Df|\, d\H^1
=\int_{\widetilde{x_Ax_B}\cap\partial\Omega} |Df|\, d\H^1\,.
\]
Since the boundary of $\widetilde\S$ is nothing else than the union of its four modified sides $\widetilde{A'B'}$, adding the last two estimates for the four sides of $\partial\S$ we get
\begin{equation}\label{tul}
\int_{\partial\widetilde\S} |Dg_1|\, d\H^1 \leq \frac{400}{\eps_4 r} \int_{\S^+} |Df|\, d\H^2 + \int_{\partial\S^+\cap\partial\Omega} |Df|\, d\H^1 \qquad \forall\, \S\in\A_{3,4}^-\,,
\end{equation}
being
\[
\A_{3,4}^- = \Big\{ \S\in\A_3\cup\A_4:\, M(A,B)=0\hbox{ for each side $AB$ of $\S$}\Big\}\,.
\]
Let us now consider a square $\S=\S(c,r)\in\A_2$, and notice that by definition it is compactly contained in $\Omega$, so $\widetilde\S=\S$ and $\widetilde{A'B'}=AB$ for any of its sides. Let $AB$ be one of those sides, and notice that $\|M(A,B)-M\|\leq \eps_5/5$, being $\S$ a Lebesgue square with matrix $M$ and constant $\hat\delta$: indeed if the other square having $AB$ as a side is not in $\A_2$, then $M(A,B)=M$, and otherwise the inequality is given by~(\ref{adjasimil}). As a consequence, again~(\ref{firstthesis}) of Lemma~\ref{4.16} and the definition of $\M^0$ give
\begin{equation}\label{benecosi}\begin{split}
\int_{AB} |Dg_1-M\cdot \nu| \,d\H^1 &\leq
\int_{AB} |Dg_1-M(A,B)\cdot \nu| \,d\H^1 + \frac 25\,r\eps_5\\
&\leq \frac{25}{\eps_4 r} \int_\RR |Df-M(A,B)|\, d\H^2+ 11 \|M\| \eps_4 r+ \frac 25\,r\eps_5\\
&\leq \frac{25}{\eps_4 r} \int_\RR |Df-M|\, d\H^2+ 120 r\,\frac{\eps_5}{\eps_4}+ 11 r \,\frac{\eps_4}{\eps_3}+ \frac 25\,r\eps_5\\
&\leq \frac{25}{\eps_4 r} \int_{\S^+} |Df-M|\, d\H^2+ 120 r\,\frac{\eps_5}{\eps_4}+ 11 r \,\frac{\eps_4}{\eps_3}+ \frac 25\,r\eps_5\\
&\leq 1020 r\,\frac{\eps_5}{\eps_4}+ 11 r \,\frac{\eps_4}{\eps_3}+ \frac 25\,r\eps_5\,,
\end{split}\end{equation}
where in the last inequality we have used Definition~\ref{Lebsquare} together with the fact that $\hat\delta\leq \bar\delta(M,\eps_5)\leq \eps_5$.

\step{IV}{The ``correct'' modified grid $\widetilde \G$ and a second temptative function $g_2:\widetilde\G\to \R^2$.}
In this step, we define a second modified grid and a second temptative function; the idea is to repeat almost exactly the procedure of Step~III, but using the neighborhoods $I_{\eps_1}(V)$ instead of those $I_{\eps_4}(V)$. In fact, the presence of $\eps_4$ is perfect for the squares in $\A_2$, since in~(\ref{benecosi}) we only have small terms like $\eps_4/\eps_3$ or $\eps_5/\eps_4$; instead, for the squares in $\A_3\cup\A_4$, the constant $\eps_4$ in~(\ref{tul}) is too small, and we would need something much larger than $\eps_2$. Since there is no constant which is at the same time much larger than $\eps_2$ and much smaller than $\eps_3$, we are forced to repeat the procedure.\par

This time, let us define the matrices $M'(A,B)=0$ for all the sides $AB$ of type A, and let us consider a slightly modified version of the intervals $I_{\eps_1}(V)$. More precisely, we let $I_{\eps_1}'(V)=I_{\eps_1}(V)$ for all the vertices $V$ which are not in the boundary of some square of $\A_1$ or $\A_2$. Instead, if a vertex $V=(V_1,V_2)$ belongs to the boundary of at least a square in $\A_1$ or $\A_2$ (these two things cannot happen simultaneously, thanks to Step~II), let us define $I_{\eps_1}'(V)$ as a translation of $I_{\eps_1}(V)$ of $(\pm 2\eps_1r, \pm 2\eps_1 r)$, where the two choices of the sign $\pm$ are done in such a way that the whole interval is inside a square of $\A_1$ or $\A_2$; for instance, if $V$ is the lower-left corner of a square in $\A_1$ (or $\A_2$), we can set
\[
I_{\eps_1}'(V) = \big\{(V_1+t,V_2+t):\, \eps_1 r \leq t \leq 3\eps_1 r\big\}\,,
\]
compare with~(\ref{defI_eps}). If $V$ is corner of more than one square in $\A_1$ or $\A_2$, then we let $I_{\eps_1}'(V)$ be inside one of them arbitrarily, this will not make any difference. Figure~\ref{fig:Ieps} shows an example of a portion of an $r$-set $\Omega$, where eight intervals $I_{\eps_1}'(V)$ are depicted and the shaded squares are those in $\A_1$ (or $\A_2$). Notice that the intervals $I_{\eps_1}'(B)$ and $I_{\eps_1}'(C)$ could be inside each of the two shaded squares, in this example we have put the first interval inside the above square and the second interval in the below one.\par
\begin{figure}[thbp]
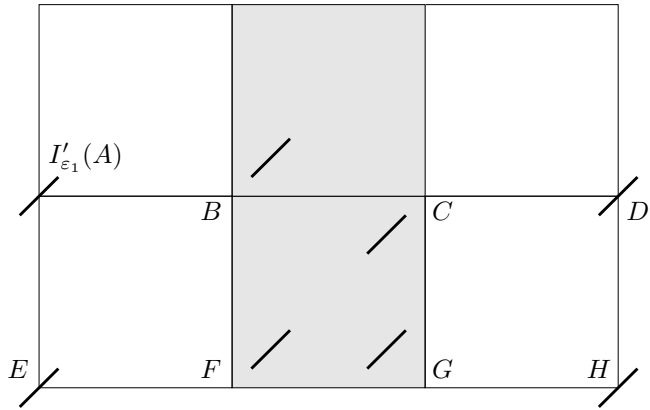
\caption{The intervals $I_{\eps_1}'(V)$ in Step~IV.}\label{fig:Ieps}
\end{figure}
After a quick look at the proof of Lemma~\ref{movegrid}, it is evident that it works perfectly even with the intervals $I_{\eps_1}'(V)$ in place of the $I_{\eps_1}(V)$: indeed, in that simple proof we just used that the internal intervals are all of length $2\sqrt{2}\eps r$, with direction at $45^\circ$, and placed very close to the vertices. As a consequence, we obtain the points $x_V'\in I_{\eps_1}'(V)$ from Lemma~\ref{selection}, the modified grid $\widetilde \G$ from Definition~\ref{defGrid}, and the function $g_2:\widetilde\G\to\R^2$ from Definition~\ref{defabo}. The modified grid that we obtained now is the ``correct'' one, and we will use the function $g_1$ (resp., $g_2$) around squares in $\A_2$ (resp., $\A_3$ and $\A_4$). The very same calculations of last step work also for this new case, just substituting the constant $\eps_4$ with $\eps_1$. In particular, since this time $M'(A,B)=0$ for all the sides, the estimate~(\ref{tul}) is true for every square in $\A_3\cup\A_4$, so we can rewrite it (with $\eps_1$ in place of $\eps_4$) as
\begin{equation}\label{tul2}
\int_{\partial\widetilde\S} |Dg_2|\, d\H^1 \leq \frac{400}{\eps_1 r} \int_{\S^+} |Df|\, d\H^2 + \int_{\partial\S^+\cap\partial\Omega} |Df|\, d\H^1 \qquad \forall\, \S\in\A_3\cup\A_4\,.
\end{equation}

\step{V}{Definition of the function $g_3:\widetilde \G\to\R^2$.}
We are now in position to define a function $g_3$ on the grid $\widetilde \G$ introduced in Step~IV. This function will behave almost correctly on the whole grid, its only fault (which will be solved in next step) being not to be piecewise linear. As we already observed, we would like to set $g_3=\varphi_\S$ on the boundary of any square $\S\in\A_1$, $g_3=g_1$ on the boundary of the squares in $\A_2$, and $g_3=g_2$ on the boundary of squares in $\A_3$ or $\A_4$; of course, this is impossible because the function $g_3$ would then not be continuous and injective. As a consequence, we use the above overall strategy to define $g_3$, but with some \emph{ad hoc} modification where squares of different types meet, so to get continuity and injectivity.\par

Let us start with the easy part of this definition. For every side $AB$ which is in the boundary of some square of $\A_1$, we define $g_3$ on $AB$ as the linear interpolation which satisfies $g_3(A)=f(A)$ and $g_3(B)=f(B)$: as a consequence, $g_3=\varphi_\S$ on $\partial \S$ for every $\S\in \A_1$, where $\varphi_\S$ is given by Definition~\ref{defphi_}. Second, for every side $AB$ which is in the boundary of some square in $\A_2$, we let $g_3=g_1$ on $AB$: recall that vertices of squares in $\A_1$ and vertices of squares in $\A_2$ are disctinct by Step~II. Finally, for every side $AB$ such that neither $A$ nor $B$ are vertices of squares of $\A_1$ or of $\A_2$, we let $g_3=g_2$ on $\widetilde {A'B'}$, where the points $A'$ and $B'$ are those given by Step~IV. Notice that, up to now, the function $g_3$ is continuous and injective: this comes as a straightforward consequence of the $L^\infty$ estimate around squares in $\A_1$, and of the fact that $g_3$ is a reparameterization of $f$ on different segments around squares in $\A_2$ or $\A_3\cup\A_4$. Nevertheless, $g_3$ has still not been defined in the whole $\widetilde \G$.\par

Let us then consider a side $AB$ such that $g_3$ has not yet been defined on $\widetilde{A'B'}$: by construction, this means that at least one between $A$ and $B$ is vertex of a square in $\A_1$ or $\A_2$, thus in particular $AB$ is not in the boundary of $\Omega$, and both the squares of the grid having $AB$ in the boundary belong to $\A_3\cup\A_4$. We aim now to define $g_3$ on $\widetilde{A'B'}$. To do so, let us keep in mind that we would like to set $g_3=g_2$, and observe also that $g_2$, on $\widetilde{A'B'}$, is nothing else than the reparametrization, at constant speed, of the image $\gamma_0$ of some piecewise linear curve $\widetilde{x_A'x_B'}$ under $f$. Our idea is to define $g_3$, on $\widetilde{A'B'}$, again as the reparametrization at constant speed of some modification $\gamma$ of $\gamma_0$. In particular, $\gamma$ and $\gamma_0$ will coincide in their big ``internal'' parts, the difference being only near the endpoints of these curves.\par
\begin{figure}[thbp]
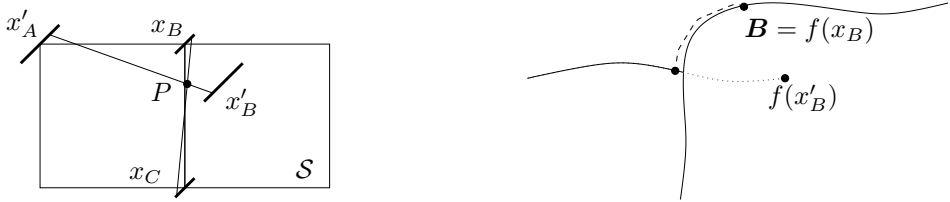
\caption{The definition of $g_3$ in Step~V if $B$ is in some square of $\A_2$.}\label{fig:hlp}
\end{figure}
For simplicity, let us start assuming that $A$ is not vertex of squares in $\A_1$ or $\A_2$, and $B$ is vertex of at least a square in $\A_2$ (by construction and recalling Step~II, $B$ is then vertex of either one or two squares in $\A_2$, and of no square in $\A_1$). By Step~IV, we know that the interval $I_{\eps_1}'(B)$ is entirely inside a square $\S$ of $\A_2$ which has $B$ as a vertex. As in Figure~\ref{fig:hlp} (left), let us call $C\neq B$ the vertex of $\S$ such that there is a square having both $A$ and $C$ as vertices. The function $g_3$ has already been defined, on $BC$, as the reparametrization of the image of a segment $x_Bx_C$ under $f$; by definition and by construction, the segment $x_Bx_C$ and the curve $\widetilde{x_A'x_B'}$ meet in some point $P$ near $B$; notice that $B\in \Omega$, and then the curve $\widetilde{x_A'x_B'}$ is actually a segment, except possibly in a small neighborhood of $x_A'$. We are now ready to define the curve $\gamma$: first, we define $\tilde\gamma$ as the image, under $f$, of the union between the part of $\widetilde{x_A'x_B'}$ from $x_A$ to $P$ and the segment $Px_B$. Then, since $\tilde\gamma$ and $g_3(x_Bx_C)$ have of course a part in common, we let $\gamma$ be a slight modification of $\tilde\gamma$ which intersects $g_3(x_Bx_C)$ only at $\u B=f(x_B)$. Of course, we need to modify $\tilde\gamma$ only between $f(P)$ and $\u B$: this modification can be done as the enlargement of Figure~\ref{fig:hlp} (right) shows, and it works exactly as in Step~10 of the proof of Theorem~\ref{extension}; in particular, the length of $\gamma$ is as close as we wish to the length of $\tilde\gamma$. By definition, we have
\begin{equation}\label{hund0}
\H^1(\tilde\gamma)\leq \int_{\widetilde{x_A'x_B'}} |Df| + \int_{x_Bx_C} |Df|
=\int_{\widetilde{A'B'}} |Dg_2| + \int_{BC} |Dg_1|\,.
\end{equation}
Thanks to~(\ref{benecosi}), calling for a moment $M$ the matrix associated to $\S$, we know that
\[
\int_{BC}  |Dg_1| \leq \int_{BC}|Dg_1-M\cdot \nu| \,d\H^1+ 2r|M|
\leq \bigg(1020\,\frac{\eps_5}{\eps_4}+ 11 \,\frac{\eps_4}{\eps_3}+ \frac 25\,\eps_5\bigg)r + \frac 2 r \int_\S |Df|
\,,
\]
and then~(\ref{hund0}) becomes
\begin{equation}\label{estihund}
\H^1(\tilde\gamma)\leq 
\int_{\widetilde{A'B'}} |Dg_2| + \bigg(1020\,\frac{\eps_5}{\eps_4}+ 11 \,\frac{\eps_4}{\eps_3}+ \frac 25\,\eps_5\bigg)r + \frac 2 r \int_\S |Df|\,.
\end{equation}\par
Let us now assume, instead, that $A$ is still not vertex of squares in $\A_1$ or $\A_2$, and that $B$ is vertex of some square in $\A_1$. We can then define $\S$ and $C$ as before; this time, $g_3$ in the segment $BC$ is not defined as the reparametrization of the image under $f$ of some segment $x_Bx_C$, but as the affine interpolation satisfying $g_3(B)=f(B)$ and $g_3(C)=f(C)$. However, the $L^\infty$ estimate~(\ref{lemmatoprove}) and the property~(\ref{stp}) immediately imply that, exactly as before, the curve $\gamma_0$ intersects the image of $BC$ under $g_3$ (which is the segment $f(B)f(C)$). If we call $\u P$ the first point of intersection (starting from $f(x_A')$), we can argue exactly as before: we define $\gamma$ as a slight modification of $\tilde\gamma$, which is this time the union of the curve $\gamma_0$ from $f(x_A')$ to $\u P$ with the segment $\u{PB}=\u P f(B)$. In this case, instead of~(\ref{hund0}) we get the estimate
\[
\H^1(\tilde\gamma)\leq \int_{\widetilde{x_A'x_B'}} |Df| + \overline{f(B)f(C)}
=\int_{\widetilde{A'B'}} |Dg_2| + \overline{f(B)f(C)}\,,
\]
and since by the $L^\infty$ estimate of Lemma~\ref{genlemma} we have of course
\[
\overline{f(B)f(C)} \leq \frac 2 r \int_\S |Df|\,,
\]
then also in this case we get the validity of~(\ref{estihund}): of course even the better estimate without the big term in parentheses is true, but it is simpler to consider the same estimate~(\ref{estihund}) in both cases.\par
Let us finally consider the general case for the segment $AB$: both the points $A$ and $B$ can be vertices of some square in $\A_1$ or in $\A_2$. Nevertheless, as noticed above, in the cases already considered the path $\tilde\gamma$ coincides with $\gamma_0$ from the starting point $x_A'$ to almost the final point $x_B'$, and only a small part near the end has been modified. As a consequence, it is obvious how to deal with the case when both the points $A$ and $B$ are vertices of squares in $\A_1$ or in $\A_2$: we let $\tilde\gamma$ be the path which coincides in a large central part with $\gamma_0$, and we apply one of the above described modifications both near the starting point, and near the endpoint. Of course, for a side $AB$ where we have done two modifications, instead of~(\ref{estihund}) we will have
\begin{equation}\label{genhund}
\H^1(\tilde\gamma)\leq 
\int_{\widetilde{A'B'}} |Dg_2| + \bigg(2040\,\frac{\eps_5}{\eps_4}+ 22 \,\frac{\eps_4}{\eps_3}+ \frac 4 5\,\eps_5\bigg)r + \frac 2 r \int_{\RR_{AB}} |Df|\,,
\end{equation}
where as usual $\RR_{AB}$ is the union of the squares having either $A$ or $B$ as a vertex, which contains both $\S$ and the corresponding square around $A$. In this way, we have finally defined $g_3$ in the whole grid $\widetilde\G$, and by construction it is clear that the map $g_3$ is injective and coincides with $f$ on $\partial\Omega$. We conclude this step by evaluating the integral of $|Dg_3|$ on the boundary of the different squares: for any square $\S$ of the grid, we call again $\S^+$ the intersection with $\Omega$ of the nine squares around $\S$.\par

If $\S\in\A_1$ we do not need any particular estimate, in the next steps we will only need to use that $g_3$ coincides with $\varphi_\S$ on $\partial\S$. If $\S\in\A_2$, instead, we know that $g_3$ coincides with $g_1$ on $\partial\S$, hence we only need to keep in mind the estimate~(\ref{benecosi}) already found in Step~III, which (adding on the four sides of $\S$) gives
\begin{equation}\label{finalinA_2}
\int_{\partial\S} |Dg_3-M\cdot \nu| \,d\H^1
\leq \bigg(4080 \,\frac{\eps_5}{\eps_4}+ 44\,\frac{\eps_4}{\eps_3}+ \frac 85\,\eps_5\bigg) r
\leq K\bigg(\frac{\eps_5}{\eps_4}+\frac{\eps_4}{\eps_3}\bigg) r \qquad\forall \S\in\A_2\,,
\end{equation}
Finally, if $\S\in\A_3\cup\A_4$, then $\int_{\partial\widetilde\S} |Dg_3|$ is the sum of the integrals on the four sides; for each side $\widetilde{A'B'}$, either $g_3=g_2$, and then of course
\begin{equation}\label{goodcs}
\int_{\widetilde{A'B'}} |Dg_3|\, d\H^1 = \int_{\widetilde{A'B'}} |Dg_2|\, d\H^1\,,
\end{equation}
or $\int_{\widetilde{A'B'}} |Dg_3|=\H^1(\gamma)$ for some curve $\gamma=\gamma(A,B)$ defined as above. Since the length of $\gamma$ can be taken as close as we wish to the length of $\tilde\gamma$, from~(\ref{genhund}) we derive
\begin{equation}\label{badcs}
\int_{\widetilde{A'B'}} |Dg_3| \leq 
\int_{\widetilde{A'B'}} |Dg_2| + \bigg(2050\,\frac{\eps_5}{\eps_4}+ 23 \,\frac{\eps_4}{\eps_3}+ \,\eps_5\bigg)r + \frac 3 r \int_{\S^+} |Df|\,.
\end{equation}
Since~(\ref{goodcs}) is stronger than~(\ref{badcs}), we get the validity of~(\ref{badcs}) for any side $AB$ of any square $\S\in\A_3\cup\A_4$. As a consequence, adding~(\ref{badcs}) on the four sides and keeping in mind~(\ref{tul2}) and the fact that $\eps_1\ll 1$, we get
\begin{equation}\label{finalinA_3}\begin{split}
\int_{\partial\widetilde\S} |D&g_3| \,d\H^1
\leq \int_{\partial\widetilde\S} |Dg_2| \,d\H^1 + \bigg(8200 \,\frac{\eps_5}{\eps_4}+ 92 \,\frac{\eps_4}{\eps_3}+ 4\eps_5\bigg) r+ \frac {12} r \int_{\S^+} |Df|\,d\H^2\\
&\leq \frac{412}{\eps_1 r} \int_{\S^+} |Df| + \int_{\partial\S^+\cap\partial\Omega} |Df|
+ 8200\bigg(\frac{\eps_5}{\eps_4}+\frac{\eps_4}{\eps_3}+\eps_5\bigg)r \quad\forall\, \S\in\A_3\cup\A_4\,.\hspace{-1pt}
\end{split}\end{equation}

\step{VI}{The piecewise linear function $g_4:\widetilde\G\to\R^2$.}
In this step, we want to define a piecewise linear function $g_4:\widetilde\G\to\R^2$: this one will finally be the correct map on the grid $\widetilde\G$, in the sense that our approximating function $f_\eps$ will coincide with $g_4$ on $\widetilde\G$. To do so, it is enough to apply Proposition~\ref{notgoodforgood} to the map $g_3$ and call $g_4=\hat g$ the resulting map. As a consequence, $g_4$ is a piecewise linear map on $\widetilde\G$, which coincides with $f$ on $\partial\Omega$. Moreover, since for every side $AB$ of the grid the map $g_4$ on $\widetilde{A'B'}$ is an interpolation of values of $g_3$ on $\widetilde{A'B'}$, of course we still have $g_4=\varphi_\S$ on the boundary of every square $\S\in\A_1$. Moreover, (\ref{check1}) and~(\ref{check2}) imply that the estimates~(\ref{finalinA_2}) and~(\ref{finalinA_3}) are valid also with $g_4$ in place of $g_3$.

\step{VII}{Definition of the approximating function $f_\eps:\Omega\to\R^2$.}
We are almost at the end of the proof, since we can finally define the required function $f_\eps$. We set $f_\eps=g_4$ on the grid $\widetilde\G$, hence in particular $f_\eps=f$ on $\partial\Omega$; by construction, $f_\eps$ is injective and piecewise linear on $\widetilde\G$. To keep the injectivity, is then enough to extend $f_\eps$ in the interior of any square $\widetilde\S$, in such a way that $f_\eps$ remains continuous and injective on it. We will argue differently on the different squares.\par

If $\S\in\A_1$, then we know that $f_\eps=\varphi_\S$ on $\partial\S$, so we extend $f_\eps=\varphi_\S$ on the whole square $\S$: this map is continuous and injective by construction, and by~(\ref{lemmatoprove}) of Lemma~\ref{genlemma} we know that
\begin{equation}\label{trueA1}
\int_\S |Df_\eps - Df| \leq r^2 \eps_5\qquad \forall \S\in\A_1\,.
\end{equation}
Let us now consider a square $\S\in\A_2$: in this case, we want to apply Theorem~\ref{extension2} to the function $\varphi=g_4$ on $\partial\S$. Keeping in mind the generalization of Theorem~\ref{extension2} observed in Remark~\ref{ext2rem}, and calling $f_\eps$ the obtained extension on $\S$, we get
\begin{equation}\label{nes}
\int_\S \big|Df_\eps(x)-M\big|\, dx\leq Kr \int_{\partial \S}\big|D g_4(t)-M\cdot\tau(t)\big|\, d\H^1(t)
\end{equation}
as soon as the estimate
\[
\int_{\partial \S}\big|D g_4(t)-M\cdot\tau(t)\big|\, d\H^1(t)<r\delta_{\rm MAX} \|M\|
\]
holds. Thanks to~(\ref{finalinA_2}), which holds also with $g_4$ in place of $g_3$ as noticed in Step~VI, and recalling that $\|M\|>\eps_3$ by definition of $\A_2$ and that $\delta_{\rm MAX}$ is a small purely geometric constant, the latter estimate is true thanks to~(\ref{here345}). As a consequence, recalling that $\S$ is a Lebesgue square with matrix $M$ and constant $\hat\delta$, that $\hat\delta\ll \eps_5$ by~(\ref{defhatdelta}) and by definition of $\bar\delta$, and using also that $\eps_4\leq 1$, from~(\ref{nes}) and~(\ref{finalinA_2}) we get
\begin{equation}\label{trueA2}
\int_\S |Df_\eps - Df| \leq
\int_\S |Df_\eps - M| + \int_\S |Df - M|
\leq Kr^2 \bigg(\frac{\eps_5}{\eps_4}+\frac{\eps_4}{\eps_3}\bigg)
\qquad \forall \S\in\A_2\,,
\end{equation}
where $K$ is as always a purely geometric constant.\par

Finally, let $\S$ be a square in $\A_3\cup\A_4$. This time, since $g_4$ is piecewise linear on $\partial\widetilde\S$ and $\widetilde\S$ is (a $2$-biLipschitz copy of) a square of side $2r$, we let $f_\eps$ on $\widetilde\S$ be the extension of $g_4$ given by Theorem~\ref{extension}, keeping in mind also the generalization of Remark~\ref{ext1rem}. The estimate~(\ref{hopeext}), together with~(\ref{finalinA_3}), which is valid also with $g_4$ in place of $g_3$ by Step~VI, gives then
\[
\int_{\widetilde\S}|Df_\eps|\leq Kr \int_{\partial \widetilde\S}|Dg_4|\, d\H^1
\leq \frac K{\eps_1} \int_{\S^+} |Df| +Kr\int_{\partial\S^+\cap\partial\Omega} |Df|\, d\H^1
+ K\bigg(\frac{\eps_5}{\eps_4}+\frac{\eps_4}{\eps_3}+\eps_5\bigg)r^2\,.
\]
Since
\[
\int_{\widetilde\S}|Df_\eps-Df| \leq \int_{\widetilde\S}|Df_\eps| + \int_{\widetilde\S}|Df|
\leq \int_{\widetilde\S}|Df_\eps| + \int_{\S^+}|Df|\,,
\]
since $\eps_4\leq 1$, and since $K$ is a purely geometric constant while $\eps_1\leq 1$, we deduce
\begin{equation}\label{trueA3}
\int_{\widetilde\S} |Df_\eps - Df| \leq \frac K{\eps_1} \int_{\S^+} |Df| + Kr\int_{\partial\S^+\cap\partial\Omega} |Df|\, d\H^1
+ K\bigg(\frac{\eps_5}{\eps_4}+\frac{\eps_4}{\eps_3}\bigg)r^2
\quad \forall \S\in\A_4\,.
\end{equation}
Notice that the same estimate holds true also for $\S\in\A_3$; nevertheless, in this case we can say even something more. Indeed, since $\S$ is a Lebesgue square with matrix $M=0$, by Definition~\ref{Lebsquare} we know
\[
\int_{\S^+} |Df| =
\int_{\S^+} |Df-M| \leq 36r^2 \hat \delta \leq 36r^2 \eps_5\,.
\]
As a consequence, for squares in $\A_3$ we can deduce from~(\ref{trueA3})
\begin{equation}\label{trueA33}
\int_{\widetilde\S} |Df_\eps - Df| 
\leq Kr\int_{\partial\S^+\cap\partial\Omega} |Df|\, d\H^1
+ K\bigg(\frac{\eps_5}{\eps_4}+\frac{\eps_4}{\eps_3}+\frac{\eps_5}{\eps_1}\bigg)r^2
\quad \forall \S\in\A_3\,.
\end{equation}

\step{VIII}{Conclusion.}
We are now ready to conclude. Indeed, by construction $f_\eps$ is a finitely piecewise affine homeomorphism which coincides with $f$ on $\partial\Omega$. Moreover, it is immediate by construction that $\|f_\eps-f\|_{L^\infty}$ and $\|f_\eps-f\|_{L^1}$ are as small as we wish (it is enough to have chosen at the beginning $r$ small enough); as a consequence, we can assume that they are smaller than $\eps/4$ each. As a consequence, to get~(\ref{thesis1}) and conclude, we only have to check that
\begin{equation}\label{nohere}
\|Df_\eps-Df\|_{L^1}< \frac \eps 2\,.
\end{equation}
Thanks to~(\ref{trueA1}), (\ref{trueA2}), (\ref{trueA3}) and~(\ref{trueA33}), and calling, for $j=1,\,2,\,3,\,4$,
\[
\Omega_j = \cup \Big\{\widetilde \S_i:\, \S_i \in \A_j \Big\}\,,
\]
we have
\begin{equation}\label{finhere}\begin{split}
\int_\Omega |Df_\eps&-Df|= \int_{\Omega_1} |Df_\eps-Df| + \int_{\Omega_2} |Df_\eps-Df| + \int_{\Omega_3} |Df_\eps-Df|+ \int_{\Omega_4} |Df_\eps-Df|\\
&\leq K \bigg(\frac{\eps_5}{\eps_4}+\frac{\eps_4}{\eps_3}+\frac{\eps_5}{\eps_1}\bigg) |\Omega|
+ \!\!\sum_{i:\, \S_i\in\A_4} \frac K{\eps_1} \int_{\S^+_i} |Df| + \sum_{i:\, \S_i\in\A_3\cup\A_4} Kr \int_{\partial\S^+_i\cap\partial\Omega} |Df|\,.
\end{split}\end{equation}
Let us now recall that, for each square $\S_i$ of the grid, the set $\S_i^+$ is the union of the nine squares around it (to be precise, to those which belong to $\Omega$). As a consequence, calling for brevity $\A_4^+ = \cup_{i:\, \S_i\in\A_4}\, \S_i^+$, also recalling~(\ref{smalleps2}) we have $|\A_4^+|\leq 9|\cup_{i:\,\S_i\in\A_4} \S_i| < \eps_2$. Thus, by~(\ref{defeps2}) we can write
\[
\sum_{i:\, \S_i\in\A_4} \frac K{\eps_1} \int_{\S^+_i} |Df|
\leq 9\,\frac K{\eps_1} \int_{\A_4^+} |Df|
\leq 9\,\frac K{\eps_1} \, \frac{\eps_1\eps}{54K}= \frac \eps 6\,.
\]
Analogously, each side in $\partial\Omega$ can belong to $\partial\S_i^+$ for at most eleven different indices $i$ with $\S_i\in\A_3\cup\A_4$, hence by~(\ref{secondbad}) we get
\[
\sum_{i:\, \S_i\in\A_3\cup\A_4} Kr \int_{\partial\S^+_i\cap\partial\Omega} |Df| \leq 11 Kr \int_{\partial\Omega} |Df|
=11 Kr P\big(f(\Omega)\big)\leq \frac \eps 6\,.
\]
Finally, by~(\ref{here345}) we have
\[
K \bigg(\frac{\eps_5}{\eps_4}+\frac{\eps_4}{\eps_3}+\frac{\eps_5}{\eps_1}\bigg) |\Omega| \leq \frac \eps 6\,.
\]
Inserting the last three estimates inside~(\ref{finhere}) we get~(\ref{nohere}), so the proof is concluded.
\end{proof}

The above proposition shows that, under stronger assumptions than in Theorem~\ref{main}, we can obtain something better than what claimed in Theorem~\ref{main}. Indeed, if $\Omega$ is an $r$-set and $f$ is piecewise linear on $\partial\Omega$, then we do not only get simply a \emph{countably} piecewise affine approximation, but a much better \emph{finitely} piecewise affine one. We can now give the sharpest possible result of this approximation, that is, we can prove the existence of a finitely piecewise affine approximation with the weakest possible assumptions.

\begin{thm}\label{main2}
Let $\Omega\subseteq\R^2$ be a polygon and let $f\in W^{1,1}(\Omega,\R^2)$ be a homeomorphism, continuous up to the boundary and such that $f$ is piecewise linear on $\partial\Omega$. Then, for every $\eps>0$ there exists a finitely piecewise affine homeomorphism $f_\eps:\Omega\to\R^2$ such that $\|f_\eps-f\|_{W^{1,1}}+\|f_\eps-f\|_{L^\infty}<\eps$, and $f_\eps=f$ on $\partial\Omega$.
\end{thm}
\begin{proof}
Since $\Omega$ is a polygon, there exists an $r$-set $\widehat \Omega$ and a finitely piecewise affine homeomorphism $\Phi:\Omega\to\widehat\Omega$. There exists then some constant $H=H(\Omega)$ such that
\begin{align*}
|D\Phi(x)| \leq H\,, &&
\det D\Phi(x)\geq \frac 1 H\,,
\end{align*}
for almost every $x\in\Omega$. Let us then define $\hat f:\widehat\Omega\to \R^2$ as $\hat f = f \circ \Phi^{-1}$; by construction, $\hat f$ belongs to $W^{1,1}(\widehat\Omega,\R^2)$, and it is continuous up to $\partial\widehat\Omega$ and piecewise linear there. As a consequence, we can apply Proposition~\ref{allhere} to $\hat f$ in $\widehat \Omega$, finding a finitely piecewise affine homeomorphism $\hat f_\eps:\widehat\Omega\to\R^2$ which coincides with $\hat f$ on $\partial \widehat\Omega$ and such that
\begin{equation}\label{estinrset}
\| \hat f_\eps - \hat f \|_{W^{1,1}(\widehat\Omega)} + \|\hat f_\eps - \hat f \|_{L^\infty(\widehat\Omega)} \leq \frac\eps{H^2}\,.
\end{equation}
We can then define $f_\eps = \hat f_\eps \circ \Phi$: this is a finitely piecewise affine homeomorphism, of course it coincides with $g$ on $\partial f$, and we have
\[
\|f_\eps-f\|_{L^\infty(\Omega)} = \|\hat f_\eps-\hat f\|_{L^\infty(\widehat\Omega)}\,.
\]
By a simple change of variable argument, we obtain then
\[\begin{split}
\|f_\eps-f\|_{L^1(\Omega)} &= \int_\Omega \big|f_\eps(x)-f(x) \big| \, dx
= \int_\Omega \big|\hat f_\eps\big(\Phi(x)\big)-\hat f\big(\Phi(x)\big) \big| \, dx\\
&= \int_{\widehat\Omega} \frac{\big|\hat f_\eps(y)-\hat f(y) \big|}{\big|\det D\Phi\big(\Phi^{-1}(y)\big)\big|} \, dy
\leq H \|\hat f_\eps-\hat f\|_{L^1(\widehat\Omega)}\,,
\end{split}\]
and similarly
\[\begin{split}
\|Df_\eps-Df\|_{L^1(\Omega)} &= \int_\Omega \big|D\big(\hat f_\eps\circ\Phi\big)(x)-D\big(\hat f\circ\Phi\big)(x)\big| \, dx\\
&= \int_\Omega \Big|\Big(D\hat f_\eps\big(\Phi(x)\big)-D\hat f\big(\Phi(x)\big) \Big)\cdot D\Phi(x)\Big| \, dx
\leq H^2 \|D\hat f_\eps-D\hat f\|_{L^1(\widehat\Omega)}\,.
\end{split}\]
Inserting the last three estimates in~(\ref{estinrset}), we conclude that $f_\eps$ is the desired approximation.
\end{proof}

\begin{remark}\label{remmain2}
It is immediate to observe that the assumption of Theorem~\ref{main2} are sharp. Indeed, assume that a homeomorphism $f\in W^{1,1}(\Omega,\R^2)$ admits a finitely piecewise affine approximation $f_\eps$. Since $f_\eps$ is finitely piecewise affine, it is defined on a polygon, hence $\Omega$ must be a polygon. Similarly, $f_\eps$ is piecewise linear on $\partial\Omega$ by definition, and since $f_\eps=f$ on $\partial\Omega$ the same must be true for $f$.
\end{remark}

We are now in position to give the proof of our main result, Theorem~\ref{main}.

\begin{proof}[Proof of Theorem~\ref{main}]
We will argue in a way quite similar to Proposition~\ref{allhere}, we only need some additional care to reach the boundary of $\Omega$. First of all, we look for a piecewise affine approximation, the smooth one will be found at the end.\par

We again start by selecting the small constants $\eps_i$: first of all, we let $\eps_1$ be a small geometric constant, say $\eps_1=1/10$. Then, since $f\in W^{1,1}(\Omega)$, we can select a constant $\eps_2$ such that
\begin{equation}\label{vve-2}
\int_A |Df| \leq \frac{\eps\eps_1}{72K}\qquad \forall\, A\subseteq\Omega:\, |A| \leq \eps_2\,.
\end{equation}
The next step is to write $\Omega$ as a countable union of $r_n$-sets. More precisely, we can take a sequence of constants $r_n\to 0$ and a sequence of disjoint open sets $\Omega_n\subset\subset \Omega$, in such a way that each $\Omega_n$ is an $r_n$-set, the union of the closures $\overline{\Omega_n}$ is the whole $\Omega$, and for each $n\in\N$ we can divide the boundary of $\Omega_n$ in two disjoint parts $\partial\Omega_n=\partial^-\Omega_n\cup\partial^+\Omega_n$, being
\begin{align*}
\partial^-\Omega_1=\emptyset\,, && \partial^+\Omega_n=\partial^-\Omega_{n+1}\quad \forall\, n\in\N\,.
\end{align*}
Since $f\in W^{1,1}(\Omega)$, we can select these sequences in such a way that
\begin{equation}\label{noinfper}
\int_{\partial\Omega_n} |Df|= P\big(f(\Omega_n)\big) < +\infty\qquad \forall\, n\in\N\,,
\end{equation}
and we can also take $\Omega_1$ large enough so that that
\begin{equation}\label{vve-4}
\int_{\Omega\setminus\Omega_1} |Df|\leq \frac {\eps\eps_1}{72K}\,.
\end{equation}
Naively speaking, the idea now is to try to work on each set $\Omega_n$ separately. However, since $f$ is not necessarily piecewise linear on the boundaries of the sets $\Omega_n$, we cannot simply rely on Proposition~\ref{allhere} for each $\Omega_n$; moreover, since $\Omega$ has not necessarily finite area, estimates like~(\ref{trueA1}) or~(\ref{trueA2}), where the area of a square appears, are not acceptable because they could give an infinite contribution after adding.\par

Let us now concentrate on $\Omega_1$ in order to select the constants $\eps_3,\,\eps_4$ and $\eps_5$: indeed, we will use these constants only inside $\Omega_1$. Arguing exactly as in the proof of Proposition~\ref{allhere}, we first let $\eps_3$ be a constant such that
\begin{equation}\label{vve-2.5}
\bigg|\bigg\{ x\in\Omega_1:\, 0<|Df(x)|< \eps_3 \hbox{ or } |Df(x)|> \frac 1{\eps_3} \hbox{ or } 0<\det(Df(x)) < \eps_3 \bigg\}\bigg|< \frac {\eps_2}{45}\,,
\end{equation}
then we let again
\begin{align*}
\M^+ = \Big\{\eps_3 < \|M\| < \frac 1{\eps_3}\,,\, \det M > \eps_3 \Big\}\,, &&
\M^0 = \Big\{\eps_3 < \|M\| < \frac 1{\eps_3}\,,\, \det M =0\Big\}\,,
\end{align*}
then we let $\eps_5\ll\eps_4\ll \eps_3$ be such that
\begin{align}\label{vve-1}
\frac{\eps_5}{\eps_4}+\frac{\eps_4}{\eps_3}+\frac{\eps_5}{\eps_1} \leq \frac \eps {8K|\Omega_1|}\,, &&
\eps_5\ll \eps_4\eps_3\,, && \eps_4\ll \eps_3^2\,,
\end{align}
and finally we let $\hat\delta=\hat\delta(\eps_3,\eps_5)$ be given by
\[
\hat\delta = \min \big\{ \bar\delta(M,\eps_5),\, M\in \M^+\cup\M^0\big\}\,,
\]
where $\bar\delta$ are the constants of Lemma~\ref{genlemma}. The last thing we have to fix, is the final value of the constants $r_n$: indeed, each $\Omega_n$ is an $r_n$-set, but then it can be regarded as a $r_n/H_n$-set for every constants $H_n\in\N$. As a consequence, we can now decrease the values of $r_n$ (without changing the sets $\Omega_n$, of course); in particular, also thanks to~(\ref{noinfper}), we can assume that $r_1$ is so small that
\begin{align}\label{vve-2.8}
r_1 P(\Omega_1) + \Big| \big\{x\in \Omega_1:\,  \bar r(x,\hat\delta) \leq r_1\big\}\big| < \frac {\eps_2}{180}\,, &&
r_1 P\big(f(\Omega_1)\big) \leq \frac \eps{66K}\,,
\end{align}
being $\bar r(x,\delta)$ the constants of Lemma~\ref{4.6}, while any other $r_n$ is so small that
\begin{align}\label{vve-5}
r_n\leq r_{n-1}\,, &&
r_n \int_{\partial\Omega_n} |Df| \leq \frac \eps{88K\cdot 2^n}\,, &&
r_n \ll {\rm dist} (\partial^-\Omega_n,\partial^+\Omega_n)\,, &&
\forall\, n\geq 2\,;
\end{align}
notice that the last requirement basically means that the ``thickness'' of any $\Omega_n$ is of several squares.\par
Having fixed the sets $\Omega_n$ and the corresponding $r_n$, any $\Omega_n$ is divided in a finite union of squares, all with side $2r_n$. Let us enumerate them by saying that the squares of the grid of $\Omega_n$ are $\S_i^n$ with $1\leq i \leq N(n)$; then, we subdivide the squares of $\Omega_1$ in four groups, namely, $\A_1^1,\, \A_2^1,\, \A_3^1,\, \A_4^1$, as follows,
\begin{align*}
\A_1^1 &= \Big\{\S_i^1 \subset\subset \Omega_1:\, \hbox{$\S_i^1$ is a Lebesgue square with matrix $M_i^1\in \M^+$ and constant $\hat\delta$}\Big\}\,,\\
\A_2^1 &= \Big\{\S_i^1 \subset\subset \Omega_1:\, \hbox{$\S_i^1$ is a Lebesgue square with matrix $M_i^1\in \M^0$ and constant $\hat\delta$}\Big\}\,,\\
\A_3^1 &= \Big\{\S_i^1 \subset\subset \Omega_1:\, \hbox{$\S_i^1$ is a Lebesgue square with matrix $M_i^1=0$ and constant $\hat\delta$}\Big\}\,,\\
\A_4^1 &= \Big\{\S_i^1:\, \S_i^1 \notin \A_1^1\cup\A_2^1\cup\A_3^1\Big\}\,.
\end{align*}
We immediately record that, exactly as in Step~I of the proof of Proposition~\ref{allhere}, from~(\ref{vve-2.5}) and~(\ref{vve-2.8}) it follows that
\begin{equation}\label{vve-3}
\big|\bigcup \{\S_i \in\A_4^1\}\big| 
\leq r_1 P(\Omega_1) +4\bigg(\Big|\big\{x\in\Omega_1:\, \bar r(x,\hat \delta)\leq r_1\big\} \Big|+\frac {\eps_2}{45}\bigg)
< \frac{\eps_2} 9\,.
\end{equation}

For any $n\geq 2$, instead, we simply let $\A_4^n$ be the collection of all the squares $\S_i^n$ of the grid of $\Omega_n$, while $\A_1^n$, $\A_2^n$ and $\A_3^n$ are empty.\par

Let us notice that the only assumption which is true for $\Omega$ in Proposition~\ref{allhere} and may fail now for the generic $\Omega_n$ is the following one: $f$ is assumed to be piecewise linear on $\partial\Omega$ in Proposition~\ref{allhere}, while $f$ needs now not to be piecewise linear on the boundaries $\partial\Omega_n$; on the other hand, this assumption has been used only in Step~VI of the proof of Proposition~\ref{allhere}. As a consequence, we can repeat \emph{verbatim} all the arguments of Steps~II, III, IV and~V of that proof for $\Omega_1$: so, we discovery first that squares in $\A_1^1$ and $\A_2^1$ can never touch, then we define a temptative modified grid $\widetilde\G_1^1$ with a function $g_1^1:\widetilde\G_1^n\to\R^2$, then the correct modified grid $\widetilde\G^1$ with the function $g_2^1:\widetilde\G^1\to\R^2$, and finally the function $g_3^1:\widetilde\G^1\to\R^2$. By definition, the function $g_3^1$ is injective and coincides with $f$ on $\partial\Omega_1$ (this was explicitely decided in Definition~\ref{defabo}), and moreover we have the estimates
\begin{equation}\label{reallyend}\begin{array}{cc}
g_3^1= \varphi_\S \ \hbox{on } \partial \S  &\forall\, \S\in \A^1_1\,, \\
\int_{\partial\S} |Dg_3^1-M\cdot \nu| \,d\H^1
\bal\leq K\bigg(\frac{\eps_5}{\eps_4}+\frac{\eps_4}{\eps_3}\bigg) r_1 \eal&\forall \S\in\A^1_2\,,\\
\bal\int_{\partial\widetilde\S} |Dg_3^1| \,d\H^1\leq \int_{\partial\S^+\cap\partial\Omega_1} |Df|+K\bigg(\frac{\eps_5}{\eps_4}+\frac{\eps_4}{\eps_3}+\frac{\eps_5}{\eps_1}\bigg)r_1\eal & \forall\, \S\in\A^1_3\,,\\
\bal\int_{\partial\widetilde\S} |Dg_3^1| \,d\H^1\leq \frac{K}{\eps_1 r_1} \int_{\S^+} |Df| + \int_{\partial\S^+\cap\partial\Omega_1} |Df|+K\bigg(\frac{\eps_5}{\eps_4}+\frac{\eps_4}{\eps_3}\bigg)r_1\eal & \forall\, \S\in\A^1_4\,,
\end{array}\end{equation}
where $K$ is a purely geometric constant (it suffices to take $K=9000$ here).\par

Let us now consider $\Omega_n$ for any $n\geq 2$. In this case, the situation is much simpler than in Proposition~\ref{allhere}: indeed, by definition we only have squares in $\A^n_4$, so we do not need the arguments of Steps~II and~III and we can directly start with the analogous of Step~IV, which immediately gives us a function $g_2^n:\widetilde\G_n\to\R^2$ satisfying the analogous of~(\ref{tul2}), that is,
\begin{equation}\label{reallyend2}
\int_{\partial\widetilde\S} |Dg_2^n|\, d\H^1 \leq \frac{400}{\eps_1 r_n} \int_{\S^+} |Df|\, d\H^2 + \int_{\partial\S^+\cap\partial\Omega_n} |Df|\, d\H^1 \qquad \forall\, \S\in\A_4^n\,.
\end{equation}
Again since there are no squares in $\A_1^n$, $\A_2^n$ and $\A_3^n$, we do not even need the argument of Step~V, and we can simply set the function $g_3^n=g_2^n$.\par

Let us now observe that every function $g_3^n$ coincides with $f$, by construction and by Definition~\ref{defabo}, on $\partial\Omega_n$. As a consequence, if we call $\widetilde\G$ the union of all the grids $\widetilde\G^n$ and we set $g_3:\widetilde\G\to\R^2$ as $g_3=g_3^n$ on each $\widetilde\G^n$, then the resulting function $g_3$ is also injective.\par

The last thing we have to do, before having the right of treating each $\Omega_n$ separately, is to modify $g_3$ so to become piecewise linear on the boundary of each $\Omega_n$; we will do that by applying Proposition~\ref{goodforstarting}. More precisely, for every $j\geq 2$ we define $\C^j$ the union of all the segments of $\widetilde\G^{j-1}$ and $\widetilde\G^j$, and $\C_0^j=\partial^-\Omega_j$. We apply then Proposition~\ref{goodforstarting} with $\C=\C^j$, $\C_0=\C^j_0$, $\eta \ll r_j\leq r_{j-1}$, and with the function $g=g_3$; thus, we get a function $\hat g_j$, piecewise linear on $\partial^-\Omega_j$, which coincides with $g$ (hence, with $g_3$) out of a $\eta$-neighborhood of $\partial^-\Omega_j$. As a consequence, the function $\hat g_j$ is different from $g_3$ only on the boundary of squares which meet $\partial^-\Omega_j$. Then, let us define the function $\hat g_3:\widetilde\G\to\R^2$ as $\hat g_3=\hat g_j$ on the boundary of squares touching $\partial^-\Omega_j$, and $\hat g_3=g_3$ on the boundaries of all the other squares. By construction and by Proposition~\ref{goodforstarting}, the function $\hat g_3$ is injective, piecewise linear on each $\partial\Omega_j$, and the estimates~(\ref{reallyend}) and~(\ref{reallyend2}) are valid with $\hat g_3$ in place of $g_3$.\par

Now, for every $n\in\N$ we apply Proposition~\ref{notgoodforgood} to the set $\Omega_n$ with the function $\hat g_3$ on $\widetilde \G^n$, and we get a new function $g_4^n$, piecewise linear on $\widetilde\G^n$ and coinciding with $\hat g_3$ on $\partial\Omega_n$. Finally, we define $g_4:\widetilde\G\to \R^2$ as $g_4=g_4^n$ on every $\widetilde\G^n$: also this function satisfies~(\ref{reallyend}) and~(\ref{reallyend2}), and it is piecewise linear on the boundary of each square of any grid.\par

We are now ready to define the piecewise affine approximation $f_\eps$: indeed, for every $n\in\N$, the function $g_4$ on $\widetilde \G^n$ is injective and piecewise linear on the boundary. Exactly as in Step~VII of the proof of Proposition~\ref{allhere}, we can then define $f_\eps^n$ on $\Omega_n$, which is a finitely piecewise affine function coinciding with $g_4$ on $\partial\Omega_n$, and then set $f_\eps:\Omega\to\R^2$ as the function which coincides with $f_\eps^n$ on every $\Omega_n$. This function $f_\eps$ is by construction a countably piecewise affine homeomorphism, and also locally finitely piecewise affine; moreover, it is clear that $\|f-f_\eps\|_{L^\infty(\Omega_n)}$ and $\|f-f_\eps\|_{L^1(\Omega_n)}$ are as small as we wish as soon as the constants $r_n$ have been chosen small enough: in particular, we can think that both them are smaller than $\eps/4$, and moreover we get the fact that $f_\eps=f$ on $\partial\Omega$ whenever $f$ is continuous up to $\partial\Omega$. As a consequence, to conclude the proof of Theorem~\ref{main} for what concerns the piecewise affine approximation, we just have to check that
\begin{equation}\label{isthisthelast}
\|Df_\eps-Df\|_{L^1(\Omega)} \leq \frac \eps 2\,.
\end{equation}
This will be obtained arguing almost exactly as in Steps~VII and~VIII of the proof of Proposition~\ref{allhere}. More precisely, let us start with $\Omega_1$: repeating \emph{verbatim} the arguments bringing to~(\ref{trueA1}), (\ref{trueA2}), (\ref{trueA3}) and~(\ref{trueA33}), this time from~(\ref{reallyend}) we get
\begin{equation}\label{vve1}\begin{array}{cc}
\bal\int_\S |Df_\eps - Df| \leq r_1^2 \eps_5\eal &\forall\, \S\in \A^1_1\,, \\[5pt]
\bal\int_\S |Df_\eps - Df| \leq Kr_1^2 \bigg(\frac{\eps_5}{\eps_4}+\frac{\eps_4}{\eps_3}\bigg)\eal &\forall \S\in\A^1_2\,,\\
\bal\int_{\widetilde\S} |Df_\eps - Df| \leq
Kr_1\int_{\partial\S^+\cap\partial\Omega_1} |Df|\, d\H^1
+ K\bigg(\frac{\eps_5}{\eps_4}+\frac{\eps_4}{\eps_3}+\frac{\eps_5}{\eps_1}\bigg)r_1^2 \eal & \forall\, \S\in\A^1_3\,,\\
\bal\int_{\widetilde\S} |Df_\eps - Df| \leq
\frac K{\eps_1} \int_{\S^+} |Df| + Kr_1\int_{\partial\S^+\cap\partial\Omega_1} |Df|\, d\H^1
+ K\bigg(\frac{\eps_5}{\eps_4}+\frac{\eps_4}{\eps_3}\bigg)r_1^2 \eal & \forall\, \S\in\A^1_4\,.
\end{array}\end{equation}
Instead, for every $n\geq 2$, we apply Theorem~\ref{extension} --recalling also Remark~\ref{ext1rem}-- to the generic $\widetilde\S$, which is a $2$-biLipschitz copy of a square of side $2r$: then, also from~(\ref{reallyend2}), we get
\begin{equation}\begin{split}\label{vve2}
\int_{\widetilde \S} |Df_\eps-Df| &\leq  \int_{\widetilde\S}|Df|+\int_{\widetilde\S}|Df_\eps|
\leq \int_{\S^+}|Df| + Kr_n \int_{\partial\widetilde\S}|Dg_4| \\
&\leq \frac K{\eps_1} \int_{\S^+} |Df| + Kr_n \int_{\partial\S^+\cap\partial\Omega_n} |Df|
\qquad\qquad\qquad\quad\forall\, \S\in\A_4^n\,.
\end{split}\end{equation}
Notice that this estimate is better than the corresponding one for Proposition~\ref{allhere}, namely, (\ref{trueA3}): indeed, there we had also the additional term $K(\eps_5/\eps_4 + \eps_4/\eps_3)r^2$, which now would be quite a problem since in principle $\Omega$ may have infinite area. The reason why we do not have this term now, is that it was coming from the interaction between squares in $\A_4$ touching squares in $\A_1$ or $\A_2$, while now in $\Omega_n$ we only have squares in $\A_4^n$.\par

The very same argument as in Proposition~\ref{allhere} implies again that every square can belong to $\S^+$ at most for $9$ different squares $\S$, and that every side of some $\partial\Omega_n$ can belong to $\partial\S^+\cap\partial\Omega_n$ at most for $11$ different squares of the grid of $\Omega_n$; as a consequence, adding~(\ref{vve1}) and~(\ref{vve2}) for all the squares of the different grid, we find
\[\begin{split}
\int_\Omega |Df_\eps - Df| &\leq K\bigg(\frac{\eps_5}{\eps_4}+\frac{\eps_4}{\eps_3}+\frac{\eps_5}{\eps_1}\bigg)|\Omega_1| 
+ 9\,\frac K{\eps_1} \int_{\A^{1,+}_4} |Df|\\
&\hspace{70pt}+ 9\,\frac K{\eps_1} \int_{\Omega\setminus\Omega_1} |Df| +11K \sum_{n\in\N}  r_n \int_{\partial\Omega_n} |Df|\\
&\leq \frac\eps 8 + \frac \eps8+ \frac \eps 8 + 11K \sum_{n\in\N}\frac{\eps}{88K\cdot 2^n} \leq \frac\eps 2\,,
\end{split}\]
where $\A_4^{1,+}$ is the union of the sets $\S^+$ for all the squares $\S\in\A_4^1$, so that by~(\ref{vve-3}) we have $|\A_4^{1,+}|\leq 9|\cup \{\S_i \in\A_4^1\}|<\eps_2$, and where we have used~(\ref{vve-1}), (\ref{vve-2}), (\ref{vve-4}) and~(\ref{vve-5}). As a consequence, we have established the validity of~(\ref{isthisthelast}), so the proof of the existence of the required piecewise affine approximation is concluded.\par

As already remarked, once the piecewise affine approximation is found, the existence of the required approximating diffeomorphisms is exactly the content of~\cite[Theorem~A]{MP}, so we have finished our proof.
\end{proof}

\bigskip

\subsection*{Acknowledgment}
Part of this research was done when the first author was visiting University of Erlangen, and while the second author was visiting Charles University. They wish to thank both departments for hospitality. Both authors were supported through the ERC CZ grant LL1203 of the Czech Ministry of Education and the ERC St.G. AnOptSetCon of the European Community.

\end{document}